\documentclass[11pt,a4paper]{amsart}
\usepackage{amsmath, amsthm, amsfonts, amssymb}
\usepackage[bookmarks=false]{hyperref}

\usepackage{color}
\numberwithin{equation}{section}

\newtheorem{letterthm}{Theorem}

\newtheorem{thm}{Theorem}[section]
\newtheorem{lem}[thm]{Lemma}

\newtheorem{cor}[thm]{Corollary}
\newtheorem{prop}[thm]{Proposition}

\theoremstyle{definition}
\newtheorem{rem}[thm]{Remark}
\newtheorem*{added}{Added in the proof}

\newtheorem{remark}[thm]{Remarks}
\newtheorem{example}[thm]{Examples}

\newtheorem*{notation}{Notation}
\newtheorem{nota}[thm]{Notation}
\newtheorem{df}[thm]{Definition}

\newtheorem*{claim}{Claim}
\newtheorem*{fact}{Fact}

\newcommand{\R}{\mathbf{R}}
\newcommand{\C}{\mathbf{C}}

\newcommand{\N}{\mathbf{N}}

\newcommand{\B}{\mathbf{B}}
\newcommand{\K}{\mathbf{K}}

\newcommand{\Ad}{\operatorname{Ad}}
\newcommand{\id}{\text{\rm id}}

\newcommand{\rL}{\mathord{\text{\rm L}}}
\newcommand{\rT}{\mathord{\text{\rm T}}}
\newcommand{\rE}{\mathord{\text{\rm E}}}

\newcommand{\tr}{\mathord{\text{\rm tr}}}
\newcommand{\Cao}{\mathcal C_{(\text{\rm AO})}}
\newcommand{\core}{\mathord{\text{\rm c}}}
\newcommand{\Diag}{\mathord{\text{\rm Diag}}}
\newcommand{\dom}{\mathord{\text{\rm dom}}}
\newcommand{\an}{\mathord{\text{\rm an}}}

\newcommand{\Tr}{\mathord{\text{\rm Tr}}}

\newcommand{\ovt}{\mathbin{\overline{\otimes}}}

\newcommand{\AC}{\mathord{\text{\rm AC}}}
\newcommand{\Bic}{\mathord{\text{\rm B}}}

\newcommand{\dpr}{^{\prime\prime}}

\begin{document}

\title[Unique prime factorization and bicentralizer problem]{Unique prime factorization and bicentralizer problem for a class of type III factors}

\begin{abstract}
We show that whenever $m \geq 1$ and $M_1, \dots, M_m$ are nonamenable factors in a large class  of von Neumann algebras that we call $\Cao$ and which contains all free Araki-Woods factors, the tensor product factor $M_1 \mathbin{\overline{\otimes}} \cdots \mathbin{\overline{\otimes}} M_m$ retains the integer $m$ and each factor $M_i$ up to stable isomorphism, after permutation of the indices. Our approach unifies the Unique Prime Factorization (UPF) results from \cite{OP03, Is14} and moreover provides new UPF results in the case when $M_1, \dots, M_m$ are free Araki-Woods factors. In order to obtain the aforementioned UPF results, we show that Connes's bicentralizer problem has a positive solution for all type ${\rm III_1}$ factors in the class $\mathcal C_{(\text{\rm AO})}$.
\end{abstract}

\dedicatory{Dedicated to the memory of Uffe Haagerup}

\author{Cyril Houdayer}
 \address{Laboratoire de Math\'ematiques d'Orsay\\ Universit\'e Paris-Sud\\ CNRS\\ Universit\'e Paris-Saclay\\ 91405 Orsay\\ France}
\email{cyril.houdayer@math.u-psud.fr}
\thanks{CH is supported by ANR grant NEUMANN and ERC Starting Grant GAN 637601}

\author{Yusuke Isono}
\address{RIMS, Kyoto University, 606-8502 Kyoto, Japan}
\email{isono@kurims.kyoto-u.ac.jp}
\thanks{YI is supported by JSPS Research Fellowship}


\maketitle

\section{Introduction and statement of the main results}

The first Unique Prime Factorization (UPF) results in the framework of von Neumann algebras were discovered by Ozawa-Popa in \cite{OP03}. Among other things, they showed that whenever $m \geq 1$ and $\Gamma_1, \dots, \Gamma_m$ are icc nonamenable bi-exact discrete groups \cite{BO08} (e.g.\ Gromov word hyperbolic groups), the tensor product ${\rm II_1}$ factor $\rL(\Gamma_1) \ovt \cdots \ovt \rL(\Gamma_m)$ retains the integer $m$ and each ${\rm II_1}$ factor $\rL(\Gamma_i)$ up to unitary conjugacy and amplification, after permutation of the indices. Ozawa-Popa's strategy \cite{OP03} was based on a combination of Ozawa's C$^*$-algebraic techniques \cite{Oz03} {\em via} the Akemann-Ostrand (AO) property and Popa's intertwining techniques \cite{Po01, Po03}. Very recently, the second named author in \cite{Is14} further developed Ozawa-Popa's strategy for von Neumann algebras endowed with almost periodic states and obtained the first UPF results in the framework of type ${\rm III}$ factors. He showed that whenever $m \geq 1$ and $M_1, \dots, M_m$ are factors associated with free quantum groups, the tensor product factor $M_1 \ovt \cdots \ovt M_m$ retains the integer $m$ and each factor $M_i$ up to stable isomorphism, after permutation of the indices. For other UPF results in the framework of ${\rm II_1}$ factors, we refer the reader to \cite{CKP14, CSU11, Pe06, Sa09, SW11}.

In this paper, we prove UPF results for tensor product factors $M_1 \mathbin{\overline{\otimes}} \cdots \mathbin{\overline{\otimes}} M_m$ where $M_1, \dots, M_m$ are nonamenable factors belonging to the class $\Cao$ defined below. As we will see, our approach unifies the UPF results from \cite{OP03} and \cite{Is14} and moreover provides new UPF results in the framework of type ${\rm III}$ factors. Before stating our main results, we first introduce some terminology. We will say that a von Neumann subalgebra $M \subset \mathcal M$ is with {\em expectation} if there exists a faithful normal conditional expectation $\rE_M : \mathcal M \to M$. Two von Neumann algebras $M$ and $N$ are {\em stably isomorphic} if $M \ovt \mathbf B(\ell^2) \cong N \ovt \mathbf B(\ell^2)$. 

We next introduce the class $\Cao$ of von Neumann algebras as the smallest class that contains all the von Neumann algebras with separable predual satisfying the strong condition (AO) in the sense of Definition \ref{AO^++} and that is stable under taking von Neumann subalgebras with expectation. Without going into details, we simply point out that all known examples of von Neumann algebras satisfying Ozawa's condition (AO) of \cite{Oz03} actually do satisfy the strong condition (AO) of Definition \ref{AO^++}. Examples of von Neumann algebras satisfying the strong condition (AO) include amenable von Neumann algebras, von Neumann algebras associated with bi-exact discrete groups \cite{BO08}, von Neumann algebras associated with free quantum groups \cite{Is13b, VV05} and free Araki-Woods factors \cite{Sh96, Ho07}. The class $\Cao$ is also stable under taking free products with respect to arbitrary faithful normal states (see Example \ref{example}). By Ozawa's result \cite{Oz03}, any von Neumann algebra $M$ belonging to the class $\Cao$ is {\em solid} in the sense that for any diffuse subalgebra with expectation $A \subset M$, the relative commutant $A' \cap M$ is amenable (see also \cite{VV05}). In particular, any nonamenable factor $M$ belonging to the class $\Cao$ is {\em prime} in the sense that whenever $M = M_1 \ovt M_2$, there exists $i \in \{1, 2\}$ such that $M_i$ is a type ${\rm I}$ factor.

Our first main result is the following classification theorem for all tensor products of nonamenable factors belonging to the class $\Cao$ with (possibly trivial) amenable factors.

\begin{letterthm}\label{thmA}
Let $m, n \geq 1$ be any integers. Let $M_1, \dots, M_m, N_1, \dots, N_n$ be any nonamenable factors belonging to the class $\Cao$. Let $M_0$ and $N_0$ be any amenable factors (possibly trivial) with separable predual.

Then the tensor product factors $M_0 \ovt M_1\ovt \cdots \ovt M_m$ and $N_0 \ovt N_1\ovt \cdots \ovt N_n$ are stably isomorphic if and only if $m=n$, $M_0$ and $N_0$ are stably isomorphic and there exists a permutation $\sigma\in \mathfrak{S}_n$ such that $M_{\sigma(j)}$ and $N_j$ are stably isomorphic for all $1 \leq j \leq n$.
\end{letterthm}

We point out that Theorem \ref{thmA} applies to all free Araki-Woods factors \cite{Sh96} and hence provides many new {\em non-isomorphism} results in the framework of type ${\rm III}$ factors. In particular, when $M = \Gamma(\rL^2(\R), \lambda_\R)\dpr$ is the type ${\rm III_1}$ free Araki-Woods factor associated with the left regular representation $\lambda_\R : \R \to \mathcal U(\rL^2(\R))$, we obtain that the tensor product type ${\rm III_1}$ factors $$M^{\ovt m} := \underbrace{M \ovt \cdots \ovt M}_{m \text{ times}}$$ 
are pairwise non-isomorphic for $m \geq 1$.

As we will see, Theorem \ref{thmA} will be a consequence of a general UPF result for tensor products of nonamenable $\Cao$ factors. We need to introduce further terminology. We will say that a $\sigma$-finite factor $M$ possesses a state with {\em large centralizer} if there exists a faithful normal state $\varphi \in M_\ast$ such that $(M^\varphi)' \cap M = \C 1$. Observe that by Connes's results on the classification of type ${\rm III}$ factors \cite{Co72}, all $\sigma$-finite type ${\rm III_\lambda}$ factors, with $0 < \lambda < 1$, always possess a state with large centralizer while $\sigma$-finite type ${\rm III_0}$ factors never do. Hence, a $\sigma$-finite factor that possesses a state with large centralizer is necessarily either of type ${\rm I}_n$, with $1 \leq n < +\infty$, or of type ${\rm II_1}$ or of type ${\rm III_\lambda}$ with $0 < \lambda \leq 1$. By Haagerup's result \cite[Theorem 3.1]{Ha85}, a type ${\rm III_1}$ factor with separable predual possesses a state with large centralizer if and only if $M$ has a {\em trivial bicentralizer}. It is an open problem, known as Connes's bicentralizer problem, to decide whether all type ${\rm III_1}$ factors with separable predual have a trivial bicentralizer \cite{Ha85}. We will provide in Section \ref{bicentralizer} a large new class of type ${\rm III_1}$ factors with separable predual and trivial bicentralizer (see Theorem \ref{semisolid}). 

\begin{notation}
For any von Neumann subalgebras with expectation $A, B \subset M$, we will write $A\sim_M B$ if there exist projections $p\in A$, $p'\in A'\cap M$, $q\in B$ and $q'\in B'\cap M$  and a nonzero partial isometry $v\in M$ such that $v^*v=pp'$, $vv^*=qq'$ and $vpApp'v^*=qBqq'$. 
\end{notation}

Our second main result is a UPF theorem for tensor products of nonamenable factors belonging to the class $\Cao$ under the assumption that each non type {\rm I} factor appearing in the `unknown' tensor product decomposition possesses a state with large centralizer. More precisely, we prove the following result.

\begin{letterthm}\label{thmB}
Let $m, n \geq 1$ be any integers. For each $1 \leq i \leq m$, let $M _i$ be any nonamenable factor belonging to the class $\Cao$. For each $1 \leq j \leq n$, let $N_j$ be any non type ${\rm I}$ factor that possesses a state with large centralizer. 
\begin{enumerate}
\item Assume that 
$$M_1\ovt \cdots \ovt M_m = N_1\ovt \cdots \ovt N_n.$$ 
Then we have $m\geq n$.
\item Assume that $m = n$ and $$M:= M_1\ovt \cdots \ovt M_n = N_1\ovt \cdots \ovt N_n.$$ 
Then there exists a unique permutation $\sigma \in \mathfrak S_n$ such that $N_j  \sim_M M_{\sigma(j)}$ for all $1 \leq j \leq n$. In particular, $M_{\sigma(j)}$ and $N_j$ are stably isomorphic for all $1 \leq j \leq n$.

\vspace{1mm}
\noindent
If $M_1, \dots, M_n$ are moreover type ${\rm III}$ factors, then $M_{\sigma(j)}$ and $N_j$ are unitarily conjugate inside $M$ for all $1 \leq j \leq n$.
\end{enumerate}
\end{letterthm}

We refer to Theorem  \ref{main theorem} for a more general UPF result dealing also with tensor products with amenable factors with separable predual. The proof of Theorem \ref{thmB} follows Ozawa-Popa's original strategy \cite{OP03}. However, unlike \cite{Is14}, we do not appeal to the Connes-Takesaki decomposition theory and we work inside the tensor product factor $M = M_1 \ovt \cdots \ovt M_m$ rather than inside its continuous core $\core(M)$. This is the main novelty of our approach compared to the one developed in \cite{Is14}.

For this strategy to work, we prove a generalization of Popa's intertwining theorem \cite{Po01, Po03}  that allows us to `intertwine with expectation' finite von Neumann subalgebras with expectation $A \subset M$ into {\em arbitrary} von Neumann subalgebras with expectation $B \subset M$ inside $M$ (see Theorem \ref{intertwining for type III}). Using this new intertwining theorem and exploiting the strong condition (AO), we can then locate finite von Neumann subalgebras with expectation in $M$ that have a nonamenable relative commutant (see Theorem \ref{location theorem}). Using this `location' theorem, exploiting the fact that all the factors $N_1, \dots, N_n$ possess a state with large centralizer and reasoning by induction over $n \geq 1$, we can finally locate all the factors $N_1, \dots, N_n$ inside $M$ and prove Theorem \ref{thmB}.

Note that our approach consisting in working inside $M$ rather than inside its continuous core $\core(M)$ allows us to remove the almost periodicity assumption that was necessary in \cite[Theorem A]{Is14}. Therefore, Theorem~\ref{thmB} above generalizes and strengthens \cite[Theorem A]{Is14} and moreover provides new UPF results when $M_1, \dots, M_m$ are free Araki-Woods factors \cite{Sh96}, many of which have no almost periodic state (e.g.\ $\Gamma(\rL^2(\R), \lambda_\R)\dpr$). We strongly believe that this novel approach will have further applications in the structure theory of type ${\rm III}$ factors.

In order to apply Theorem \ref{thmB} and to obtain classification results for tensor products of nonamenable $\Cao$ factors as in Theorem \ref{thmA}, we show that any nonamenable type ${\rm III}$ factor belonging to the class $\Cao$ possesses a state with large centralizer. This is our third main result.

\begin{letterthm}\label{thmC}
Let $M$ be any nonamenable type ${\rm III}$ factor belonging to the class $\Cao$. Then $M$ possesses a state with large centralizer.
\end{letterthm}

Observe that any nonamenable factor $M$ belonging to the class $\Cao$ is necessarily full by \cite[Theorem A]{HR14} and hence cannot be of type ${\rm III_0}$ by \cite[Theorem 2.12]{Co74}. Therefore, we may assume that $M$ is of type ${\rm III_1}$ in the statement of Theorem \ref{thmC}. In that case, we first show that $M$ has a trivial bicentralizer and then use Haagerup's result \cite[Theorem 3.1]{Ha85} to deduce that $M$ possesses a state with large centralizer (see Theorem \ref{semisolid} for a more general result). In order to show that $M$ has a trivial bicentralizer, we prove several new results regarding the structure of bicentralizer algebras.

First, for all $\sigma$-finite type ${\rm III_1}$ factors $(P, \varphi)$ endowed with a faithful normal state, we provide a very simple and yet very useful interpretation of the corresponding bicentralizer algebra $\Bic(P, \varphi)$ in terms of ultraproduct von Neumann algebras (see Proposition~\ref{bicentralizer-interpretation}). Then using Ando-Haagerup's result \cite[Proposition 4.24]{AH12}, we show in Theorem \ref{structure-bicentralizer} the following dichotomy result for the bicentralizer algebra $\Bic(P, \varphi)$, assuming moreover that $P$ has separable predual : either $\Bic(P, \varphi) = \C 1$ or $\Bic(P, \varphi) \subset P$ is a McDuff type ${\rm III_1}$ subfactor (globally invariant under the modular automorphism group $(\sigma_t^\varphi)$). Using this dichotomy result for bicentralizer algebras and Haagerup's result \cite[Theorem 2.3]{Ha85} on amenable type ${\rm III_1}$ factors, and exploting the fact that $M$ belongs to the class $\Cao$ and hence is solid, we prove Theorem \ref{thmC}. We point out that Theorem \ref{thmC} and the more general Theorem \ref{semisolid} provide many new examples of type ${\rm III_1}$ factors with trivial bicentralizer.

\subsection*{Acknowledgments}
The proof of Theorem \ref{intertwining for type III} stemmed from several thought-provoking discussions with Steven Deprez, Sven Raum and Yoshimichi Ueda. We are therefore very grateful to them for sharing their insights with us. We also thank Yoshimichi for his valuable comments. Last but not least, we thank the referees for carefully reading the paper and providing useful remarks.

\tableofcontents

\section{Preliminaries}\label{preliminaries}

For any von Neumann algebra $M$, we denote by $\mathcal Z(M)$ its center, by $\mathcal U(M)$ its group of unitaries and by $M_\ast$ its predual. We say that a von Neumann algebra $M$ is $\sigma$-{\em finite} if it possesses a faithful normal state $\varphi \in M_\ast$.


\subsection*{Ultraproduct von Neumann algebras}

Let $M$ be any $\sigma$-finite von Neumann algebra. Let $(I, \leq)$ be any nonempty directed set and $\mathcal U$ any {\em cofinal} ultrafilter on $I$, that is, for all $i_0 \in I$, we have $\{i \in I : i \geq i_0\} \in \mathcal U$. When $I = \N$, the cofinal ultrafilter $\mathcal U$ is also called a {\em nonprincipal} ultrafilter and we will use the standard notation $\mathcal U = \omega$.

Define
\begin{align*}
\mathcal I_{\mathcal U}(M) &= \left\{ (x_i)_{i} \in \ell^\infty(I, M) : x_i \to 0 \ast\text{-strongly as } i \to \mathcal U \right\} \\
\mathcal M^{\mathcal U}(M) &= \left \{ (x_i)_i \in \ell^\infty(I, M) :  (x_i)_i \, \mathcal I_{\mathcal U}(M) \subset \mathcal I_{\mathcal U}(M) \text{ and } \mathcal I_{\mathcal U}(M) \, (x_i)_i \subset \mathcal I_{\mathcal U}(M)\right\}.
\end{align*}

We have that the {\em multiplier algebra} $\mathcal M^{\mathcal U}(M)$ is a C$^*$-algebra and $\mathcal I_{\mathcal U}(M) \subset \mathcal M^{\mathcal U}(M)$ is a norm closed two-sided ideal. Following \cite{Oc85}, we define the {\em ultraproduct von Neumann algebra} $M^{\mathcal U}$ by $M^{\mathcal U} = \mathcal M^{\mathcal U}(M) / \mathcal I_{\mathcal U}(M)$. Observe that the proof given in \cite[5.1]{Oc85} for the case when $I = \N$ and $\mathcal U = \omega$ applies {\em mutatis mutandis}. We will simply denote the image of $(x_i)_i \in \mathcal M^{\mathcal U}(M)$ by $(x_i)^{\mathcal U} \in M^{\mathcal U}$.

For all $x \in M$, the constant net $(x)_i$ lies in the multiplier algebra $\mathcal M^{\mathcal U}(M)$. We will then identify $M$ with $(M + \mathcal I_{\mathcal U}(M))/ \mathcal I_{\mathcal U}(M)$ and regard $M \subset M^{\mathcal U}$ as a von Neumann subalgebra. The map $\rE_{\mathcal U} : M^{\mathcal U} \to M : (x_i)^{\mathcal U} \mapsto \sigma \text{-weak} \lim_{i \to {\mathcal U}} x_i$ is a faithful normal conditional expectation. For every faithful normal state $\varphi \in M_\ast$, the formula $\varphi^{\mathcal U} = \varphi \circ \rE_{\mathcal U}$ defines a faithful normal state on $M^{\mathcal U}$. Observe that $\varphi^{\mathcal U}((x_i)^{\mathcal U}) = \lim_{i \to {\mathcal U}} \varphi(x_i)$ for all $(x_i)^{\mathcal U} \in M^{\mathcal U}$.

Let $Q \subset M$ be any von Neumann subalgebra with faithful normal conditional expectation $\rE_Q : M \to Q$. Choose a faithful normal state $\varphi$ on $Q$ and still denote by $\varphi$ the faithful normal state $\varphi \circ \rE_Q$ on $M$. We have $\ell^\infty(I, Q) \subset \ell^\infty(I, M)$, $\mathcal I_{\mathcal U}(Q) \subset \mathcal I_{\mathcal U}(M)$ and $\mathcal M^{\mathcal U}(Q) \subset \mathcal M^{\mathcal U}(M)$. We will then identify $Q^{\mathcal U} = \mathcal M^{\mathcal U}(Q) / \mathcal I_{\mathcal U}(Q)$ with $(\mathcal M^{\mathcal U}(Q) + \mathcal I_{\mathcal U}(M)) / \mathcal I_{\mathcal U}(M)$ and regard $Q^{\mathcal U} \subset M^{\mathcal U}$ as a von Neumann subalgebra. Observe that the norm $\|\cdot\|_{(\varphi |_Q)^{\mathcal U}}$ on $Q^{\mathcal U}$ is the restriction of the norm $\|\cdot\|_{\varphi^{\mathcal U}}$ to $Q^{\mathcal U}$. Observe moreover that $(\rE_Q(x_i))_i \in \mathcal I_{\mathcal U}(Q)$ for all $(x_i)_i \in \mathcal I_{\mathcal U}(M)$ and $(\rE_Q(x_i))_i \in \mathcal M^{\mathcal U}(Q)$ for all $(x_i)_i \in \mathcal M^{\mathcal U}(M)$. Therefore, the mapping $\rE_{Q^{\mathcal U}} : M^{\mathcal U}\to Q^{\mathcal U} : (x_i)^{\mathcal U} \mapsto (\rE_Q(x_i))^{\mathcal U}$ is a well-defined conditional expectation satisfying $\varphi^{\mathcal U} \circ \rE_{Q^{\mathcal U}} = \varphi^{\mathcal U}$. Hence, $\rE_{Q^{\mathcal U}} : M^{\mathcal U} \to Q^{\mathcal U}$ is a faithful normal conditional expectation. 

We will be using the ultraproduct von Neumann algebra framework in Section \ref{bicentralizer} for $I = \N$ and in Section \ref{intertwining} for possibly uncountable directed sets $I$. For more on ultraproduct von Neumann algebras, we refer the reader to \cite{AH12, BO08, Oc85}.

\subsection*{Modular theory}
Let $M$ be any von Neumann algebra.  For any faithful normal semifinite weight $\Phi$ on $M$, put \begin{align*}
	\mathfrak{n}_{\Phi}&:= \left\{x \in M\mid \Phi(x^*x)< +\infty \right\},\\
	\mathfrak{m}_{\Phi}&:=(\mathfrak{n}_{\Phi})^*\mathfrak{n}_{\Phi}= \left\{\sum_{i=1}^nx_i^*y_i \mid n \geq 1, x_i, y_i\in \mathfrak{n}_{\Phi} \text{ for all } 1 \leq i \leq n \right\}.
\end{align*}
We will denote by $\rL^2(M, \Phi)$ the $\rL^2$-completion of $\mathfrak n_\Phi$ with respect to $\Phi$, by $\Lambda_\Phi : \mathfrak n_\Phi \to \rL^2(M, \Phi)$ the canonical injection, by $\Delta_\Phi$  the modular operator, by $J_\Phi$ the modular conjugation on $\rL^2(M, \Phi)$, and by $(\sigma^\Phi_t)$ the modular automorphism group associated with $\Phi$. We have $x \Lambda_\Phi(y) = \Lambda_\Phi(xy)$ for all $x \in M$ and all $y \in \mathfrak n_\Phi$. We will simply write $\|x\|_\Phi := \sqrt{\Phi(x^*x)}$ for all $x \in \mathfrak n_\Phi$. 
The Hilbert space $\rL^2(M,\Phi)$, together with its positive part $\mathfrak P_\Phi$ and its modular conjugation $J_\Phi$ is called the {\em standard form} of $M$ and is denoted by $(M, \rL^2(M, \Phi), J_\Phi, \mathfrak P_\Phi)$. The standard form is unique and does not depend on the choice of $\Phi$ thanks to \cite[Theorem 2.3]{Ha73}. The {\em centralizer algebra} $M^\Phi$ is defined as the fixed point algebra of the modular automorphism group $(\sigma_t^\Phi)$. For all $a\in M^\Phi$ and all $x\in\mathfrak{m}_\Phi$, we have $ax, xa\in \mathfrak{m}_\Phi$ and $\Phi(ax)=\Phi(xa)$. 

Let $A\subset M$ be any inclusion of $\sigma$-finite von Neumann algebras with faithful normal conditional expectation $\rE_A : M \to A$ and let $\varphi_A$ be any faithful normal state on $A$. Put $\varphi:=\varphi_A\circ \rE_A$. Then we have $\sigma_t^\varphi|_A=\sigma_t^{\varphi_A}$ for all $t \in \R$ and hence $\sigma_t^{\varphi}(A)=A$ and $\sigma_t^{\varphi}(A'\cap M)=A'\cap M$. Thus by \cite[Theorem IX.4.2]{Ta03}, the inclusion $A'\cap M\subset M$ is with expectation.

\subsection*{Basic construction}

Let $B \subset M$ be any inclusion of $\sigma$-finite von Neumann algebras with faithful normal conditional expectation $\rE_B : M \to B$. Fix a faithful normal semifinite weight $\Phi_B$ on $B$ and put $\Phi := \Phi_B \circ \rE_B$. We simply denote by $(M, H, J, \mathfrak P) = (M, \rL^2(M, \Phi), J_\Phi, \mathfrak P_\Phi)$ the standard form of the von Neumann algebra $M$. 

We have a canonical inclusion $\rL^2(B,\Phi_B) \subset \rL^2(M,\Phi)$. The {\em Jones projection} $e_B^{\Phi_B} \colon \rL^2(M,\Phi) \to \rL^2(B,\Phi_B)$ is defined by $e_B^{\Phi_B}\Lambda_{\Phi}(x) := \Lambda_{\Phi_B}(\rE_B(x))$ for $x\in \mathfrak{n}_{\Phi}$. By the condition $\Phi=\Phi_B\circ \rE_B$ and the inequality $\rE_B(x)^*\rE_B(x) \leq \rE_B(x^*x)$ for all $x\in M$, we have that $e_B^{\Phi_B} $ is well defined and satisfies $(e_B^{\Phi_B})^*=e_B^{\Phi_B}=(e_B^{\Phi_B})^2$ and $e_B^{\Phi_B} x e_B^{\Phi_B} = \rE_B(x)e_B^{\Phi_B} = e_B^{\Phi_B} \rE_B(x)$ for all $x\in M$. Observe that $e_B^{\Phi_B} : \rL^2(M, \Phi) \to \rL^2(B, \Phi_B)$ is nothing but the orthogonal projection from $\rL^2(M, \Phi)$ onto $\rL^2(B,\Phi_B)$. As we observe in Proposition \ref{prop-uniqueness-projection}, the Jones projection $e_B^{\Phi_B} : \rL^2(M, \Phi) \to \rL^2(B, \Phi_B)$ does not depend on the choice of the faithful normal semifinite weight $\Phi_B$ on $B$. Therefore, when no confusion is possible, we will simply denote the Jones projection by $e_B : \rL^2(M) \to \rL^2(B)$. 

The {\em basic construction} of the inclusion $B \subset M$ is defined by 
$$\langle M, B\rangle := \langle M, e_B\rangle = (J B J)' \cap \mathbf B(H).$$
The basic construction $\langle M, B\rangle$ is uniquely determined by $\mathfrak P$ and $\rE_B : M \to B$. We refer to Sections \ref{section-relative} and \ref{section-dimension} of the Appendix for more information on Jones basic construction and the dimension theory for semifinite von Neumann algebras.

\subsection*{Operator valued weights}
Let $M$ be any von Neumann algebra. 
The \textit{extended positive cone} $\widehat{M}^+$ of $M$ is defined as the set of all the lower semicontinuous functions $m\colon M_*^+\rightarrow [0,\infty]$ satisfying
\begin{itemize}
	\item $m(\varphi+\psi)=m(\varphi)+m(\psi)$ for all $\varphi,\psi\in M_*^+$,
	\item $m(\lambda \varphi)=\lambda m(\varphi)$ for all $\varphi\in M_*^+$ and all $\lambda\geq 0$. 
\end{itemize}

	Let $B\subset M$ be any von Neumann subalgebra. Recall that a map $\rT \colon M^+\rightarrow \widehat{B}^+$ is an \textit{operator valued weight} from $M$ to $B$ if it satisfies the following three conditions:
\begin{itemize}
	\item $\rT(\lambda x)=\lambda \rT(x)$ for all $x\in M^+$ and all $\lambda \geq 0$, 
	\item $\rT(x+y)=\rT(x)+\rT(y)$ for all $x,y \in M^+$,
	\item $\rT(b^*xb)=b^*\rT(x)b$ for all $x\in M^+$ and all $b\in B$.
\end{itemize}

Let $\rT : M^+\rightarrow \widehat{B}^+$ be any operator valued weight. Put 
\begin{align*}
	\mathfrak{n}_{\rT}&:= \left\{x\in M\mid \|\rT(x^*x)\|_\infty<+\infty \right\},\\
	\mathfrak{m}_{\rT}&:=(\mathfrak{n}_{\rT})^*\mathfrak{n}_{\rT}= \left\{\sum_{i=1}^nx_i^*y_i \mid n \geq 1, x_i, y_i\in \mathfrak{n}_{\rT} \text{ for all } 1 \leq i \leq n \right\}.
\end{align*}
Then $\rT$ has a unique extension $\rT\colon \mathfrak{m}_{\rT}\rightarrow B$ that is $B$-$B$-bimodular. In particular $\rT$ extends to a conditional expectation if $\rT(1_M)=1_B$. The operator valued weight $\rT$ is said to be 
\begin{itemize}
\item \textit{faithful} if $\rT(x)=0$ $\Rightarrow$ $x=0$, ($x\in M^+$),
\item \textit{normal} if  $\rT(x_i)\nearrow \rT(x)$ \quad whenever $x_i\nearrow x$, ($x_i,x\in M^+$),
\item \textit{semifinite} if $\mathfrak{m}_{\rT}$ is $\sigma$-weakly dense in $M$.
\end{itemize}
In this paper, all the operator valued weights that we consider are assumed to be faithful, normal and semifinite. For more on operator valued weights, we refer the reader to \cite{Ha77a,Ha77b,Ko85,ILP96}. The following two lemmas are well-known.

\begin{lem}\label{operator valued weight1}
	Let $B\subset M$ be any inclusion of von Neumann algebras, $\rT$ any faithful normal semifinite operator valued weight from $M$ to $B$ and $\Phi$ any faithful normal semifinite weight on~$B$.
	\begin{enumerate}
		\item The composition $\Phi\circ \rT$ defines a faithful normal semifinite weight on $M$ that satisfies $\sigma^{\Phi\circ \rT}|_B=\sigma^\Phi$. 
		\item If $\rT : M \to B$ is moreover assumed to be a faithful normal conditional expectation, then there exists a canonical faithful normal semifinite operator valued weight $\rT_M$ from  $\langle M,B\rangle$ to $M$ given by $\rT_M(xe_Bx^*)=xx^*$, where $e_B$ is the Jones projection of the inclusion $B \subset M$. 
	\end{enumerate}
\end{lem}
\begin{proof}
	For $(1)$, see \cite[Proposition 2.3 and Theorem 4.7]{Ha77a}. For $(2)$, see \cite[Lemma 3.1]{Ko85}. 
\end{proof}

\begin{lem}\label{operator valued weight2}
	Let $B\subset M$ be any inclusion of von Neumann algebras, $\rT$ any faithful normal semifinite operator valued weight from $M$ to $B$. Let $p\in B'\cap M$ be any nonzero projection such that $\rT(p)\in B$. 
	
	Then there exists a projection $z\in \mathcal{Z}(B)$ such that $q:=zp\in \mathcal{Z}(B)p$ is nonzero, the element $qT(p)^{-1/2}$ is well defined and the map $qM^+q\ni qxq\mapsto q\rT(p)^{-1/2} \, \rT(qxq) \, qT(p)^{-1/2} \in B^+q$ extends to a faithful normal conditional expectation from $qMq$ onto $Bq$.
\end{lem}
\begin{proof}
Observe that we have $\rT(p) \in \mathcal Z(B)$. The spectral projection $z \in \mathcal Z(B)$ of $\rT(p) \in \mathcal Z(B)$ corresponding to the bounded interval $[\frac12 \|\rT(p)\|_\infty,\|\rT(p)\|_\infty]$ is nonzero. Then $z\rT(p)^{-1/2} \in \mathcal Z(B)$ is well defined and $pz\rT(p)^{-1/2}$ is nonzero since 
$$\rT(pz \rT(p)^{-1/2})= \rT(p)z \rT(p)^{-1/2} = z \rT(p)^{1/2}\neq 0.$$ 
Put $q:=pz$ and observe that $q \neq 0$ since $q \rT(p)^{-1/2} \neq 0$.

Denote by $\rE : qM^+q \to B^+q$ the map as in the statement of the lemma. Then $\rE$ is well-defined and bounded by $q$, since for all $x \in M^+$, we have
$$\rE(qxq)\leq q\rT(p)^{-1/2} \, \rT(q\|qxq\|_\infty q) \, q\rT(p)^{-1/2}=\|qxq\|_\infty \, q\rT(p)^{-1/2} \,  \rT(p) \, q\rT(p)^{-1/2}=\|qxq\|_\infty \, q.$$
By construction, $\rE$ is a normal operator valued weight from $qMq$ to $Bq$. Denote by $z_{B'}(p) \in B$ the central support in $B'$ of the projection $p \in B' \cap M$. Note that $z_{B'}(p) p = p$ and hence $z_{B'}(p)q = z_{B'}(p) pz = pz = q$. The map $\rE : qM^+q \to B^+q$ is faithful since for all $x\in M^+$, we have
$$\rE(qxq)=0 \Rightarrow p\rT(qxq)=0 \Rightarrow z_{B'}(p)\rT(qxq)= \rT(z_{B'}(p) qxq) = \rT(qxq)=0.$$
Finally since $\rE(q)=q$, it is extended to a faithful normal conditional expectation from $qMq$ onto $Bq$. This finishes the proof of Lemma \ref{operator valued weight2}.
\end{proof}

\begin{rem}\label{remark expectation}
Let $B \subset M$ be any unital inclusion of von Neumann algebras with faithful normal conditional expectation $\rE_B : M \to B$. The proof of Lemma \ref{operator valued weight2} above shows that for any nonzero projection $p \in B' \cap M$, there exists an increasing sequence $(z_n)_n$ of nonzero projections in $\mathcal Z(B)p$ (defined by $z_n:= \mathbf 1_{[\frac{1}{n + 1}, 1]}(\rE_B(p))p \in \mathcal Z(B)p$) such that $z_n \to p$ $\sigma$-strongly and the inclusion $B z_n \subset z_n M z_n$ is with expectation for all $n \in \N$. 
\end{rem}

We will also need the following useful fact.

\begin{rem}\label{remark expectation bis}
Let $B \subset M$ be any unital inclusion of $\sigma$-finite von Neumann algebras with expectation. Then for every nonzero central projection $z \in \mathcal Z(B' \cap M)$, the unital inclusion $Bz \subset zMz$ is with expectation. Indeed, choose a faithful normal state $\varphi \in M_\ast$ such that $B$ is globally invariant under the modular automorphism group $(\sigma_t^\varphi)$. Observe that $B' \cap M$ is also globally invariant under the modular automorphism group $(\sigma_t^\varphi)$. This implies that $\sigma_t^\varphi(z) = z$ for every $t \in \R$ whenever $z \in \mathcal Z(B' \cap M)$. Define $\varphi_z = \frac{\varphi(z \, \cdot \, z)}{\varphi(z)} \in (zMz)_\ast$. Then we have $\sigma_t^{\varphi_z}(Bz) = \sigma_t^\varphi(B) z = Bz$ for every $t \in \R$ by \cite[Lemme 3.2.6]{Co72}. This implies that $Bz \subset zMz$ is with expectation. Observe that since $\mathcal Z(M) \subset \mathcal Z(B' \cap M)$, the above result also applies to $z \in \mathcal Z(M)$.
\end{rem}

\begin{added}
Since the first version of this paper has been posted on the arXiv, there have been some  developments related to Remarks \ref{remark expectation} and \ref{remark expectation bis}. Indeed, it is showed in \cite[Proposition 2.2]{HU15} that for any inclusion of $\sigma$-finite von Neumann algebra $B \subset M$ with expectation and any nonzero projection $p \in B' \cap M$, the inclusion $Bp \subset pMp$ is still with expectation. This result improves and supersedes the observations contained in Remarks \ref{remark expectation} and \ref{remark expectation bis}. In the rest of the paper and for the sake of clarity and exposition, we will use the result from \cite[Proposition 2.2]{HU15} whenever we need it.
\end{added}

We finally recall the {\em push down} lemma originally due to Pimsner-Popa in the type ${\rm II_1}$ setting \cite[Lemma 1.2]{PP84}. This will play an important role in order to extend Popa's intertwining techniques to the type ${\rm III}$ setting in Section \ref{intertwining}.

\begin{lem}\label{push down lemma}
Let $B\subset M$ be any inclusion of von Neumann algebras with faithful normal conditional expectation $\rE_B : M \to B$. Denote by $\rT_M$ the canonical faithful normal semifinite operator valued weight from $\langle M,B\rangle$ to $M$ and denote by $e_B$ the Jones projection of the inclusion $B \subset M$. 
	
Then for all $x\in \mathfrak{n}_{\rT_M}$, we have $e_Bx=e_B\rT_M(e_Bx)$.
\end{lem}

\begin{proof}
See the proof of \cite[Proposition 2.2]{ILP96}. We point out that the factoriality assumption on $B$ and $M$ in the statement of \cite[Proposition 2.2]{ILP96} is actually unnecessary.
\end{proof}

\subsection*{A strengthening of Ozawa's condition (AO)}

Recall from \cite{Oz03} that a von Neumann algebra $\mathcal M\subset \mathbf{B}(H)$ satisfies \textit{condition} (AO) if there exist $\sigma$-weakly dense unital C$^*$-subalgebras $A\subset \mathcal M$ and $B\subset \mathcal M'$ such that $A$ is locally reflexive and such that the multiplication map $\nu : A\otimes_{\rm alg}B \to \mathbf{B}(H)/\mathbf{K}(H) : a\otimes b\mapsto ab + \mathbf K(H)$ is continuous with respect to the minimal tensor norm. 

In order to show Theorem \ref{thmB}, we will need to introduce a stronger notion than condition (AO) that behaves well with respect to taking tensor products. We will use the following terminology.

\begin{df}\label{AO^++}
	Let $\mathcal M$ be any von Neumann algebra and $(\mathcal M, H, J, \mathfrak P)$ a standard form for $\mathcal M$. We say that $\mathcal M$ satisfies the \textit{strong condition} $\rm (AO)$ if there exist unital C$^*$-subalgebras $A\subset \mathcal M$ and $\mathcal{C}\subset \B(H)$ such that 
	\begin{itemize}
		\item $A$ is exact and is $\sigma$-weakly dense in $\mathcal M$,
		\item $\mathcal{C}$ is nuclear and contains $A$ and
		\item The commutators $[\mathcal{C}, JAJ] := \left\{[c, JaJ] : c \in \mathcal C, a \in A \right\}$ are contained in $\K(H)$.
	\end{itemize}
\end{df}

\begin{remark}
We point out the following observations.
\begin{enumerate}
\item If $B\subset \B(H)$ is a nuclear C$^*$-algebra, then $B+\K(H)$ is also a nuclear C$^*$-algebra, since it is an extension of $B/(B\cap \K(H))$ by $\K(H)$, both of which are nuclear C$^*$-algebras. Hence in the definition above, we can always assume that $\mathcal{C}$ contains $\K(H)$.
\item It is not difficult to show that the strong condition $\rm (AO)$ of Definition \ref{AO^++} implies Ozawa's condition (AO). In fact, by the last condition in Definition \ref{AO^++}, the multiplication map $\nu\colon \mathcal{C}\otimes_{\rm alg} JAJ\rightarrow \B(H)/\K(H)$ is a well-defined $\ast$-homomorphism. It follows that $\nu$ is continuous with respect to the maximal tensor norm and hence with respect to the minimal tensor norm by nuclearlity of $\mathcal{C}$. The restriction of $\nu$ to $A\otimes_{\rm min}JAJ$ gives condition (AO) for $\mathcal M\subset \B(H)$. 
\item Under the additional assumptions that
\begin{enumerate}
\item $\mathcal{C}$ is separable and
\item $[\mathcal{C},J\mathcal{C}J]\subset \K(H)$,
\end{enumerate}
we can show that the strong condition $\rm (AO)$ implies the condition $\rm (AO)^+$ introduced in \cite[Definition 3.1.1]{Is13a}. Indeed, we obtain continuity of the multiplication map $\nu$ on $\mathcal{C}\otimes_{\rm min} J\mathcal{C}J$, which is separable and nuclear, and hence by the lifting theorem \cite{CE76}, there is a ucp lift from $\mathcal{C}\otimes_{\rm min} J\mathcal{C}J$ into $\B(H)$. Assumptions (a) and (b) are easily verified for all the examples in Examples \ref{example} below.
\end{enumerate}
\end{remark}

\begin{example}\label{example}
We observe that all known examples of von Neumann algebras that satisfy Ozawa's condition (AO) actually do satisfy the strong condition (AO) from Definition \ref{AO^++}.
\begin{enumerate}
	\item Any amenable von Neumann algebra $\mathcal M$ (with separable predual) is AFD \cite{Co75}, and hence we can find a $\sigma$-weakly dense nuclear C$^*$-subalgebra $A\subset \mathcal M$. Obviously, $\mathcal M$ satisfies the strong condition $\rm (AO)$. 
	\item Any bi-exact discrete group $\Gamma$ gives rise to the group von Neumann algebra $\rL(\Gamma)$ that satisfies the strong condition (AO). This follows from \cite[Proposition 15.2.3 (2)]{BO08} (see also the proof of \cite[Lemma 3.1.4]{Is13b}). 
	\item Any free quantum group gives rise to a von Neumann algebra that satisfies the strong condition (AO) \cite{VV08, VV05} (see \cite[Subsection 3.1]{Is13b}). More precisely, any quantum group in the class $\mathcal{C}$ in \cite[Definition 2.2.1 and Proposition 2.2.2]{Is14} gives rise to a von Neumann algebra that satisfies the strong condition (AO). 
	\item Any free Araki-Woods factor \cite{Sh96} satisfies the strong condition $\rm (AO)$ (see Theorem \ref{thm-FAW-AO}).
	\item The strong condition $\rm (AO)$ is stable under taking free products with respect to arbitrary faithful normal states. This follows from \cite[Section 3]{Oz04} (see also \cite[Section 4]{GJ07} and \cite[Proposition 3.2.5]{Is13b}).
\end{enumerate}
\end{example}

Let $m \geq 1$. For all $1 \leq i \leq m$, let $\mathcal M_i$ be any von Neumann algebra with standard form $(\mathcal M_i, H_i, J_i, \mathfrak P_i)$ that satisfies the strong condition (AO) with corresponding C$^*$-algebras $A_i$ and $\mathcal{C}_i$. Assume that the von Neumann algebra $\mathcal M_0$ with standard form $(\mathcal M_0, H_0, J_0, \mathfrak P_0)$ and separable predual is amenable and hence AFD by \cite{Co75}, and assume that $A_0=\mathcal{C}_0\subset \mathcal{M}_0$ is a $\sigma$-weakly dense nuclear C$^*$-algebra. Write $(\mathcal M, H, J, \mathfrak P)$ for the standard form of $\mathcal M:= \mathcal M_0 \ovt \cdots \ovt \mathcal M_m$, $A:=A_0\otimes_{\rm min} \cdots \otimes_{\rm min} A_m$, $\mathcal{C}:=\mathcal{C}_0\otimes_{\rm min} \cdots \otimes_{\rm min} \mathcal{C}_m$ and $\mathcal{J}:=\sum_{i=1}^m\mathcal{K}_i$, where $\mathcal{K}_i$ is given by
	$$\mathbf{B}(H_0)\otimes_{\rm min}\cdots\otimes_{\rm min}\mathbf{B}(H_{i-1})\otimes_{\rm min} \mathbf{K}(H_i)\otimes_{\rm min}\mathbf{B}(H_{i+1})\otimes_{\rm min}\cdots\otimes_{\rm min}\mathbf{B}(H_m).$$
Since $\mathcal K_i$ is a norm closed two-sided ideal in the C$^*$-algebra $\mathbf{B}(H_0)\otimes_{\rm min}\cdots\otimes_{\rm min}\mathbf{B}(H_m)$, it follows that $\mathcal J$ is a C$^*$-algebra. The next proposition will be used in the proof of Theorem \ref{location theorem} (see also \cite[Lemma 10]{OP03} and \cite[Proposition 3.1.2]{Is14} for a similar statement).

\begin{prop}\label{condition AO for tensor product}
Denote by $\mathcal M (\mathcal{J})\subset \B(H)$ the multiplier algebra of $\mathcal{J}$. The following assertions are true: 
\begin{enumerate}
\item The C$^*$-algebra $\mathcal C$ is unital and nuclear and the C$^*$-algebra $A$ is unital, exact and $\sigma$-weakly dense in $\mathcal M$. 	
\item We have 	$[\mathcal C, JAJ] \subset \mathcal J$.
\item The multiplication map $\nu\colon A\otimes_{\rm alg}JAJ \rightarrow \mathcal M (\mathcal{J})/\mathcal{J}$ is continuous with respect to the minimal tensor norm.
\end{enumerate}
\end{prop}

\begin{proof}
$(1)$ It is clear that $\mathcal C$ is a nuclear unital C$^*$-algebra and $A \subset \mathcal C$ is a unital, exact C$^*$-algebra that is $\sigma$-weakly dense in $\mathcal M$. 

$(2)$  For all $0 \leq i, j \leq m$, we have $[\mathcal C_i, JA_jJ] = 0$ if $i \neq j$ and $[\mathcal C_i, JA_jJ] \subset \mathcal K_i \subset \mathcal J$ if $i = j$, where $\mathcal{K}_0:=0$. Since the norm closed two-sided ideal generated by all $[\mathcal{C}_i,JA_jJ]$ in $\mathcal{M}(\mathcal{J})$ contains $[\mathcal{C}, JAJ]$ and since $\mathcal J \subset \mathcal M(\mathcal J)$ is a norm closed two-sided ideal, we obtain $[\mathcal C, JAJ] \subset \mathcal J$.

$(3)$ Finally, the multiplication map $\mathcal C \otimes_{\rm alg} JAJ \rightarrow \mathcal M (\mathcal{J})/\mathcal{J}$ is continuous with respect to the maximal tensor norm and hence with respect to the minimal tensor norm since $\mathcal C$ is nuclear. By restriction, we obtain that the multiplication map $\nu\colon A\otimes_{\rm alg}JAJ \rightarrow \mathcal M (\mathcal{J})/\mathcal{J}$ is continuous with respect to the minimal tensor norm.
\end{proof}

Recall that the class of von Neumann algebras $\Cao$ is defined as the smallest class that contains all the von Neumann algebras with separable predual satisfying the strong condition (AO) in the sense of Definition \ref{AO^++} and that is stable under taking von Neumann subalgebras with expectation. Observe that using item $(5)$ in Examples \ref{example}, it is easy to see that the class $\Cao$ is also stable under taking free products with respect to arbitrary faithful normal states.

\begin{prop}\label{cao}
Let $M$ be any von Neumann algebra. The following assertions are true:
\begin{enumerate}
\item If $M$ belongs to the class $\Cao$, then $M \otimes \mathbf M_n$ and $M \ovt \B(\ell^2)$ belong to the class $\Cao$ for all $n \geq 1$.
\item If $M \ovt \B(\ell^2)$ belongs to class $\Cao$, then $M$ belongs to the class $\Cao$.
\item If $M$ is a factor belonging to the class $\Cao$, then $p Mp$ belongs to the class $\Cao$ for any nonzero projection $p \in M$.
\end{enumerate} 
\end{prop}

\begin{proof}
$(1)$ Let $\mathcal M$ be a von Neumann algebra satisfying the strong condition (AO) and such that the inclusion $M \subset \mathcal M$ is with expectation. Let $A\subset \mathcal{M}$ and $A\subset\mathcal{C}$ be C$^*$-algebras as in the definition of the strong condition (AO) for $\mathcal{M}$. Let $H$ be any Hilbert space and Tr a canonical trace on $\B(H)$. 
Then since $\K(H)J_{\mathrm{Tr}}\K(H)J_{\mathrm{Tr}}\subset \K(\rL^2(\B(H),\mathrm{Tr}))$, it is easy to see that $A\otimes_{\rm min} \K(H)+\C 1_{\mathcal{M}\ovt \B(H)}$ and $\mathcal{C}\otimes_{\rm min}\K(H)+\C1_{\mathcal{M}\ovt\B(H)}$ satisfy the strong condition (AO) for $\mathcal M \ovt \B(H)$. Moreover, the inclusion $M \ovt \B(H) \subset \mathcal M \ovt \B(H)$ is with expectation. This shows that $M \ovt \B(H)$ belong to the class $\Cao$ for any separable Hilbert space $H$.

$(2)$ Let $\mathcal M$ be a von Neumann algebra with separable predual satisfying the strong condition (AO) and such that the inclusion $M \ovt \B(\ell^2) \subset \mathcal M$ is with expectation. Since the inclusion $M \subset M \ovt \B(\ell^2)$ is with expectation, we have that $M$ belongs to the class $\Cao$.

$(3)$ Let $p \in M$ be any nonzero projection. Observe that $p Mp \ovt \mathbf B(\ell^2) \cong M \ovt \mathbf B(\ell^2)$. Since $M$ belongs to the class $\Cao$, so does $M \ovt \mathbf B(\ell^2)$ by item $(1)$. Since $p Mp \ovt \mathbf B(\ell^2)$ belongs to the class $\Cao$, so does $pMp$ by item $(2)$.
\end{proof}

\section{Structure of bicentralizer von Neumann algebras}\label{bicentralizer}

In this section, we show that the bicentralizer algebra as defined by Connes (see \cite{Ha85}) has a simple interpretation in terms of ultraproduct von Neumann algebras. While this result is very elementary, it enables us to provide in Theorem \ref{semisolid} a new class of type ${\rm III_1}$ factors with separable predual and trivial bicentralizer.

\subsection*{Bicentralizer von Neumann algebras in the ultraproduct framework}

\begin{df}
Let $M$ be any $\sigma$-finite von Neumann algebra and $\varphi \in M_*$ any faithful normal state. Let $\omega \in \beta(\N) \setminus \N$ be any nonprincipal ultrafilter on $\N$.

We define the {\em asymptotic centralizer} (resp.\ $\omega$-{\em asymptotic centralizer}) of $\varphi$ by
\begin{align*}
\AC(M, \varphi) &:= \left\lbrace (x_n)_n \in \ell^\infty(\N, M) \mid \lim_{n \to \infty} \|x_n \varphi - \varphi x_n\| = 0 \right\rbrace ,\\
\AC_\omega(M, \varphi) & := \left\lbrace (x_n)_n \in \ell^\infty(\N, M) \mid \lim_{n \to \omega} \|x_n \varphi - \varphi x_n\| = 0 \right\rbrace.
\end{align*}
Here, for all $a,b\in M$, the normal form $a\varphi b\in M_*$ is given by $(a\varphi b)(x):=\varphi(bxa)$ for all $x\in M$. 
We define the {\em bicentralizer} (resp.\ $\omega$-{\em bicentralizer}) of $\varphi$ by
\begin{align*}
\Bic(M, \varphi) &=  \left\lbrace a \in M \mid \lim_{n \to \infty} \|a x_n - x_n a\|_\varphi = 0, \forall (x_n)_n \in \AC(M, \varphi) \right\rbrace \\
\Bic_\omega(M, \varphi) &=  \left\lbrace a \in M \mid \lim_{n \to \omega} \|a x_n - x_n a\|_\varphi = 0, \forall (x_n)_n \in \AC_\omega(M, \varphi) \right\rbrace.
\end{align*}
\end{df}

The following proposition shows that the bicentralizer $\Bic(M, \varphi)$ as defined by Connes coincides with the $\omega$-bicentralizer $\Bic_\omega(M, \varphi)$ for all $\omega \in \beta(\N) \setminus \N$.

\begin{prop}\label{same-bicentralizer}
Let $M$ be any $\sigma$-finite von Neumann algebra and $\varphi \in M_*$ any faithful normal state. Then for every $\omega \in \beta(\N) \setminus \N$, we have
$$\Bic(M, \varphi) = \Bic_\omega(M, \varphi).$$
\end{prop}

\begin{proof}
The proof is essentially the same as the one of \cite[Lemma 1.2]{Ha85}. We include it for the sake of completeness. Since $\AC_\omega(M, \varphi)$ is a unital C$^*$-algebra, it is generated by its group of unitaries given by
$$\mathcal U(\AC_\omega(M, \varphi)) = \left\lbrace (u_n)_n \in \ell^\infty(\N, M) \mid u_n \in \mathcal U(M), \forall n \in \N \text{ and } \lim_{n \to \omega} \|u_n \varphi - \varphi u_n\| = 0 \right\rbrace.
$$
It follows that 
$$\Bic_\omega(M, \varphi) =  \left\lbrace a \in M \mid \lim_{n \to \omega} \|a u_n - u_n a\|_\varphi = 0, \forall (u_n)_n \in \mathcal U(\AC_\omega(M, \varphi)) \right\rbrace
$$
and hence
$$\Bic_\omega(M, \varphi) =  \left\lbrace a \in M \mid \lim_{n \to \omega} \|u_n^*a u_n - a\|_\varphi = 0, \forall (u_n)_n \in \mathcal U(\AC_\omega(M, \varphi)) \right\rbrace.
$$

For all $a \in M$ and all $\delta > 0$, define the $\sigma$-weakly closed convex subset $\mathcal K_{\varphi}(a, \delta)$ of $M$ by
$$\mathcal K_\varphi(a, \delta) = \overline{{\rm co}}^w \lbrace u^* a u \mid u \in \mathcal U(M) \text{ and } \|u \varphi - \varphi u\| \leq \delta\rbrace.$$
Define 
$$\varepsilon(a, \delta) = \sup \lbrace \|u^* a u -  a\|_\varphi \mid u \in \mathcal U(M) \text{ and } \|u \varphi - \varphi u\| \leq \delta\rbrace.$$
Since the map $M \mapsto [0, +\infty) : x \mapsto \|x\|_\varphi$ is $\sigma$-weakly lower semicontinuous, we have $\|x - a\|_\varphi \leq \varepsilon(a, \delta)$ for all $x \in \mathcal K_\varphi(a, \delta)$.

Let $a \in \Bic_\omega(M, \varphi)$. Let $x \in \bigcap_{\delta > 0} \mathcal K_\varphi(a, \delta) = \bigcap_{n \in�\N} \mathcal K_\varphi(a, \frac{1}{n + 1})$. For every $n \in \N$, using the definition of $\varepsilon(a, \frac{1}{n + 1})$, we may choose a unitary $u_n \in \mathcal U(M)$ such that $\|u_n \varphi - \varphi u_n\| \leq \frac{1}{n + 1}$ and $\varepsilon(a, \frac{1}{n + 1}) \leq \|u_n^* a u_n - a\|_\varphi + \frac{1}{n + 1}$. Since $\omega \in \beta(\N) \setminus \N$ is nonprincipal, we have $\lim_{n \to \omega} \|u_n \varphi - \varphi u_n\| = 0$ and hence $(u_n)_n \in \mathcal U(\AC_\omega(M, \varphi))$. Since $a \in \Bic_\omega(M, \varphi)$, we have $\lim_{n \to \omega} \|u_n^* a u_n - a\|_\varphi = 0$ and hence $\lim_{n \to \omega} \varepsilon(a, \frac{1}{n + 1}) = 0$. Since $x \in \bigcap_{n \in�\N} \mathcal K_\varphi(a, \frac{1}{n + 1})$, we have $\|x - a\|_\varphi \leq \varepsilon(a, \frac{1}{n + 1})$ for all $n \in \N$. This implies that $\|x - a\|_\varphi \leq \lim_{n \to \omega} \varepsilon(a, \frac{1}{n + 1}) = 0$ and hence $x = a$. Therefore, $\bigcap_{\delta > 0} \mathcal K_\varphi(a, \delta) = \{a\}$.

Conversely, let $a \notin \Bic_\omega(M, \varphi)$. Then there exists a sequence of unitaries $(u_n)_n \in \mathcal U(\AC_\omega(M, \varphi))$ such that $\lim_{n \to \omega} \|u_n \varphi - \varphi u_n\| = 0$ and $\varepsilon := \lim_{n \to \omega} \|u_n^* a u_n - a\|_\varphi > 0$. Next, define $b := \sigma\text{-weak} \lim_{n \to \omega} u_n^* a u_n \in M$. For every $\delta > 0$, we have $\lbrace n \in \N : \|u_n \varphi - \varphi u_n\| \leq \delta \rbrace \in \omega$. Since $\lbrace n \in \N : \|u_n \varphi - \varphi u_n\| \leq \delta \rbrace \subset \lbrace n \in \N :  u_n^* a u_n \in \mathcal K_{\varphi}(a, \delta) \rbrace$, we have $\lbrace n \in \N :  u_n^* a u_n \in \mathcal K_{\varphi}(a, \delta) \rbrace \in \omega$. This implies that $b \in \mathcal K_\varphi(a, \delta)$ for all $\delta > 0$, that is, $b \in \bigcap_{\delta > 0} \mathcal K_\varphi(a, \delta)$. We next show that $b \neq a$. Indeed, observe that since $\lim_{n \to \omega} \|u_n \varphi - \varphi u_n\| = 0$, we have
$$\lim_{n \to \omega} \|u_n^* a u_n\|_\varphi^2 = \lim_{n \to \omega} \varphi(u_n^* a^* a u_n) = \varphi(a^* a) = \|a\|_\varphi^2.$$
Since $\Lambda_\varphi(u^*_n a u_n) \to \Lambda_\varphi(b)$ weakly, we then obtain
\begin{align*}
\Re \langle \Lambda_\varphi(b), \Lambda_\varphi(a) \rangle_\varphi &= \lim_{n \to \omega} \Re \langle \Lambda_\varphi(u_n^* a u_n), \Lambda_\varphi(a)\rangle_\varphi \\
&= \lim_{n \to \omega}  \frac12(\|u_n^* a u_n\|_\varphi^2 + \|a \|_\varphi^2 - \|u_n^* a u_n - a\|_\varphi^2) \\
&= \|a\|_\varphi^2 - \frac12 \varepsilon^2 < \|a\|_\varphi^2.
\end{align*}
Thus $b \neq a$ and hence $\bigcap_{\delta > 0} \mathcal K_\varphi(a, \delta) \neq \{a\}$.

We have proved that $a \in \Bic_\omega(M, \varphi)$ if and only if $\bigcap_{\delta > 0} \mathcal K_\varphi(a, \delta) = \{a\}$. Therefore $a \in \Bic_\omega(M, \varphi)$ if and only if $a \in \Bic(M, \varphi)$ by \cite[Lemma 1.2]{Ha85}. Thus $\Bic_\omega(M, \varphi) = \Bic(M, \varphi)$.
\end{proof}

From now on, we fix a nonprincipal ultrafilter $\omega \in \beta(\N) \setminus \N$. Recall that for all $(x_n)_n \in \AC_\omega(M, \varphi)$, we have $(x_n)_n \in \mathcal M^\omega(M)$ and $(x_n)^\omega \in (M^\omega)^{\varphi^\omega}$ (see e.g.\ \cite[Proposition 2.4 (2)]{Ho14}).

\begin{prop}\label{bicentralizer-interpretation}
Let $M$ be any $\sigma$-finite von Neumann algebra and $\varphi \in M_*$ any faithful normal state. We have $\Bic(M, \varphi) = ((M^\omega)^{\varphi^\omega})' \cap M$. In particular, $\Bic(M, \varphi) \subset M$ is a von Neumann subalgebra that is globally invariant under the modular automorphism group $(\sigma_t^{\varphi})$.
\end{prop}

\begin{proof}
Thanks to Proposition \ref{same-bicentralizer}, we have
\begin{align*}
\Bic(M, \varphi) &=  \left\lbrace a \in M \mid \lim_{n \to \omega} \|a x_n - x_n a\|_\varphi = 0, \forall (x_n)_n \in \AC_\omega(M, \varphi) \right\rbrace \\
& =\left\lbrace a \in M \mid \|a (x_n)^\omega - (x_n)^\omega a\|_{\varphi^\omega} = 0, \forall (x_n)^\omega \in (M^\omega)^{\varphi^\omega} \right\rbrace \\
& = ((M^\omega)^{\varphi^\omega})' \cap M. \qedhere
\end{align*}
\end{proof}

\begin{prop}\label{proposition-restriction}
Let $M$ be any $\sigma$-finite von Neumann algebra and $\varphi \in M_*$ any faithful normal state. Write $\psi = \varphi |_{\Bic(M, \varphi)}$. Then
$$\Bic(\Bic(M, \varphi), \psi) = \Bic(M, \varphi).$$
In other words, the von Neumann algebra $\Bic(M, \varphi)$ is equal to its own bicentralizer with respect to the state $\varphi |_{\Bic(M, \varphi)}$.
\end{prop}

\begin{proof}
Let $(x_n)_n \in \AC(\Bic(M, \varphi), \psi)$, that is, $\lim_{n \to \infty} \|x_n \psi - \psi x_n\|_{\Bic(M, \varphi)_\ast} = 0$. Denote by $\rE : M \to \Bic(M, \varphi)$ the unique $\varphi$-preserving conditional expectation. For all $x \in M$, we have
\begin{align*}
(x_n \varphi - \varphi x_n)(x) &= \varphi(x x_n - x_n x) \\
& = \varphi( \rE (x x_n - x_n x)) \\
&= \varphi (\rE(x) x_n -  x_n \rE(x)) \\
&= \psi (\rE(x) x_n -  x_n \rE(x)) \\
& = (x_n\psi - \psi x_n) (\rE(x)).
\end{align*}
Therefore $\|x_n \varphi - \varphi x_n\|_{M_*} \leq \|x_n \psi - \psi x_n\|_{\Bic(M, \varphi)_*}$ and so $\lim_{n \to \infty} \|x_n \varphi - \varphi x_n\|_{M_*} = 0$, that is, $(x_n)_n \in \AC(M, \varphi)$. For all $a \in \Bic(M, \varphi)$, we obtain $\lim_{n \to \infty} \|a x_n - x_n a\|_\varphi = 0$ and hence $a \in \Bic(\Bic(M, \varphi), \psi)$. This implies that $\Bic(\Bic(M, \varphi), \psi) = \Bic(M, \varphi)$.
\end{proof}

From the previous propositions, we deduce a very simple and yet very useful dichotomy result for bicentralizer von Neumann algebras. We refer to \cite[Theorem 2.9]{Co74} for the definition of the {\em asymptotic centralizer} $M_\omega$ of a $\sigma$-finite von Neumann algebra $M$.

\begin{thm}\label{structure-bicentralizer}
Let $M$ be any type ${\rm III_1}$ factor with separable predual and $\varphi \in M_*$ any faithful normal state. Then either $\Bic(M, \varphi) = \C 1$ or $\Bic(M, \varphi) \subset M$ is a McDuff type ${\rm III_1}$ subfactor with expectation.
\end{thm}

\begin{proof}
Since $\Bic(M, \varphi)$ is globally invariant under the modular automorphism group $(\sigma_t^\varphi)$, it follows that $\Bic(M, \varphi) \subset M$ is with expectation. Since $\Bic(M, \varphi) = ((M^\omega)^{\varphi^\omega})' \cap M$ by Proposition \ref{bicentralizer-interpretation} and since $(M^\omega)^{\varphi^\omega}$ is a ${\rm II_1}$ factor by \cite[Proposition 4.24]{AH12}, we have 
$$\Bic(M, \varphi)^{\varphi |_{\Bic(M, \varphi)}} = ((M^\omega)^{\varphi^\omega})' \cap M^\varphi \subset ((M^\omega)^{\varphi^\omega})' \cap (M^\omega)^{\varphi^\omega} = \C 1.$$
Therefore, either $\Bic(M, \varphi) = \C 1$ or $\Bic(M, \varphi)$ is a type ${\rm III_1}$ factor by \cite[Lemma 5.3]{AH12}.

Assume that $N := \Bic(M, \varphi)$ is of type ${\rm III_1}$ and put $\psi := \varphi |_N$. By Proposition \ref{proposition-restriction}, we have $N = \Bic(N, \psi)$. This implies that $N = ((N^\omega)^{\psi^\omega})' \cap N$ and hence $(N^\omega)^{\psi^\omega} \subset N' \cap N^\omega$. Since $(N' \cap N^\omega)^{\psi^\omega} = N_\omega$ by \cite[Proposition 2.8]{Co74} (see also \cite[Proposition 4.35]{AH12}), we have $(N^\omega)^{\psi^\omega} \subset N_\omega$ and hence $(N^\omega)^{\psi^\omega} = N_\omega$. Since $(N^\omega)^{\psi^\omega}$ is a ${\rm II_1}$ factor by \cite[Proposition 4.24]{AH12}, $N_\omega$ is a ${\rm II_1}$ factor and hence $N$ is McDuff by \cite[Theorem 2.2.1]{Co75}, that is, $N \cong N \ovt R$ where $R$ is the unique AFD ${\rm II_1}$ factor.
\end{proof}

From the previous theorem, we deduce a new characterization of type ${\rm III_1}$ factors with separable predual and trivial bicentralizer in terms of the existence of a maximal abelian subalgebra with expectation.

\begin{cor}
Let $M$ be any type ${\rm III_1}$ factor with separable predual. The following conditions are equivalent.
\begin{enumerate}
\item $\Bic(M, \varphi) = \C 1$ for some or any faithful normal state $\varphi \in M_\ast$.
\item There exists a maximal abelian subalgebra $A \subset M$ with faithful normal conditional expectation $\rE_A : M \to A$.
\end{enumerate}
\end{cor}

\begin{proof}
$(1) \Rightarrow (2)$ By \cite[Theorem 3.1]{Ha85}, there exists a faithful normal state $\varphi \in M_\ast$ such that $(M^\varphi)' \cap M = \C 1$. Then by \cite[Theorem 3.3]{Po81}, there exists a maximal abelian subalgebra $A \subset M$ such that $A \subset M^\varphi$.

$(2) \Rightarrow (1)$ Fix a faithful normal state $\tau \in A_\ast$ and put $\varphi = \tau \circ \rE_A \in M_\ast$. We have $A \subset M^\varphi$ and hence $\Bic(M, \varphi) \subset (M^\varphi)' \cap M \subset A' \cap M = A$. Applying Theorem \ref{structure-bicentralizer}, it follows that $\Bic(M, \varphi) = \C 1$. By \cite[Corollary 1.5]{Ha85}, we obtain that $\Bic(M, \psi) = \C 1$ for all faithful normal states $\psi \in M_\ast$.
\end{proof}

\subsection*{Semisolid type ${\rm III_1}$ factors have trivial bicentralizer}

Recall that a von Neumann algebra $M$ is {\em solid} if the relative commutant $Q' \cap M$ of any diffuse von Neumann subalgebra $Q \subset M$ with expectation is amenable. We say that a von Neumann algebra $M$ is {\em semisolid} if the relative commutant $Q' \cap M$ of any von Neumann subalgebra $Q \subset M$ with expectation and with no type ${\rm I}$ direct summand is amenable. Obviously, any solid von Neumann algebra is semisolid.

\begin{thm}\label{semisolid}
Let $M$ be any semisolid type ${\rm III_1}$ factor with separable predual. Then $\Bic(M, \varphi) = \C 1$ for any faithful normal state $\varphi \in M_\ast$.
\end{thm}

\begin{proof}
By contradiction, assume that there exists a faithful normal state $\varphi \in M_\ast$ such that $\Bic(M, \varphi) \neq \C 1$. Since $\Bic(M, \varphi) \subset M$ is with expectation and $M$ is semisolid, $\Bic(M, \varphi)$ is semisolid. By Theorem~\ref{structure-bicentralizer}, $\Bic(M, \varphi)$ is a McDuff type ${\rm III_1}$ subfactor with expectation. Therefore, we may replace $M$ by $\Bic(M, \varphi)$ and assume that $M$ is a semisolid McDuff type ${\rm III_1}$ factor with separable predual satisfying $M = \Bic(M, \varphi)$ (see Proposition \ref{proposition-restriction}).

We have $M \cong M \ovt R$ where $R$ is the unique AFD ${\rm II_1}$ factor. Since $M$ is semisolid, we obtain that $M \cong M \otimes \C1 = (\C 1 \otimes R)' \cap (M \ovt R)$ is amenable. Hence $M$ is an amenable type ${\rm III_1}$ factor with separable predual and nontrivial bicentralizer since $M = \Bic(M, \varphi)$. This however contradicts \cite[Theorem 2.3]{Ha85}.
\end{proof}

Based on the $14 \varepsilon$-type lemma \cite[Lemme 4.1]{Va04}, it was showed in \cite{Ho08} that all the free Araki-Woods factors associated with separable orthogonal representations have a trivial bicentralizer. We point out that Theorem \ref{semisolid} above gives a new and more conceptual proof of this result and more generally shows that all the type ${\rm III_1}$ factors that belong to the class $\Cao$ and hence are solid, have a trivial bicentralizer as well. We can now prove Theorem \ref{thmC}.

\begin{proof}[Proof of Theorem \ref{thmC}]
Any nonamenable factor $M$ that belongs to the class $\Cao$ is necessarily full by \cite[Theorem A]{HR14} and hence cannot be of type ${\rm III_0}$ by \cite[Theorem 2.12]{Co74}. Therefore, using  \cite[Th\'eor\`eme 4.2.6]{Co72}, we may further assume that $M$ is of type ${\rm III_1}$. Since $M$ belongs to the class $\Cao$, $M$ is solid by \cite{Oz03, VV05} and hence semisolid. Combining Theorem \ref{semisolid} with \cite[Theorem 3.1]{Ha85}, we deduce that $M$ possesses a state with large centralizer.
\end{proof}

\section{Popa's intertwining techniques for type ${\rm III}$ von Neumann algebras}\label{intertwining}

To fix notation, let $M$ be any $\sigma$-finite von Neumann algebra, $1_A$ and $1_B$ any nonzero projections in $M$, $A\subset 1_AM1_A$ and $B\subset 1_BM1_B$ any von Neumann subalgebras. Popa introduced his powerful {\em intertwining-by-bimodules techniques} in \cite{Po01} in the case when $M$ is finite and more generally in \cite{Po03} in the case when $M$ is endowed with an almost periodic faithful normal state $\varphi \in M_\ast$ for which $1_A, 1_B  \in M^\varphi$, $A \subset 1_A M^\varphi 1_A$ and $B \subset 1_B M^\varphi 1_B$. It was showed in \cite{HV12,Ue12} that Popa's intertwining techniques extend to the case when $B$ is semifinite and with expectation in $1_B M 1_B$ and $A \subset 1_A M 1_A$ is any von Neumann subalgebra.

In this section, we investigate a new setting in which $A \subset 1_A M 1_A$ is any finite von Neumann subalgebra with expectation and $B \subset 1_B M 1_B$ is any von Neumann subalgebra with expectation. This situation is technically more challenging than the one studied in \cite{HV12, Ue12} since $B$ can possibly be of type ${\rm III}$ and hence the basic construction $\langle M,B\rangle$ may no longer carry a faithful normal semifinite trace. Since we can no longer use the fact that $B$ is semifinite as in \cite[Proposition 3.1]{Ue12}, we use instead, as in the proof of \cite[Theorem A.1]{Po01}, the canonical faithful normal semifinite operator valued weight from $\langle M,B\rangle$ to $M$ and exploit the fact that $A$ is a finite von Neumann algebra.

\subsection*{Main result}

We will use the following terminology throughout this section.

\begin{df}\label{definition intertwining}
Let $M$ be any $\sigma$-finite von Neumann algebra, $1_A$ and $1_B$ any nonzero projections in $M$, $A\subset 1_AM1_A$ and $B\subset 1_BM1_B$ any von Neumann subalgebras with faithful normal conditional expectations $\rE_A : 1_A M 1_A \to A$ and $\rE_B : 1_B M 1_B \to B$ respectively.  

We will say that $A$ {\em embeds with expectation into} $B$ {\em inside} $M$ and write $A \preceq_M B$ if there exist projections $e \in A$ and $f \in B$, a nonzero partial isometry $v \in eMf$ and a unital normal $\ast$-homomorphism $\theta : eAe \to fBf$ such that the inclusion $\theta(eAe) \subset fBf$ is with expectation and $av = v \theta(a)$ for all $a \in eAe$.
\end{df}

It is important to observe that in the setting of Definition \ref{definition intertwining}, both of the inclusions $eAe \, vv^* \subset vv^* M vv^*$ and $\theta(eAe) \, v^*v = v^* \, eAe \, v \subset v^*v M v^*v$ are with expectation thanks to \cite[Proposition 2.2]{HU15}.

Fix a standard form $(M, H, J, \mathfrak P)$ for the von Neumann algebra $M$. Put $\widetilde{B}:=B\oplus \C(1_M-1_B)$ and extend $\rE_B : 1_B M 1_B \to B$ to a faithful normal conditional expectation $\rE_{\widetilde{B}} : M \to \widetilde{B}$. Denote by $e_{\widetilde{B}}$ the Jones projection of the inclusion $\widetilde B \subset M$, $\langle M,\widetilde{B} \rangle = (J \widetilde B J)' \cap \mathbf B(H)$ the basic construction and $\rT_M$ the canonical faithful normal semifinite operator valued weight from $\langle M,\widetilde{B} \rangle$ to $M$ given by $\rT_M(xe_{\widetilde{B}}x^*)=xx^*$ for all $x \in M$. Choose a faithful normal state $\varphi \in M_\ast$ such that $\varphi = \varphi \circ \rE_{\widetilde B}$. 

Write $B = B_1 \oplus B_2$ where $B_1$ is the semifinite direct summand of $B$ and $B_2$ is the type ${\rm III}$ direct summand of $B$. Fix a faithful normal semifinite trace $\Tr_{B_1}$ on $B_1$ and denote by $\Tr$ the unique trace on $\langle M,\widetilde{B}\rangle J1_{B_1}J$ satisfying $\Tr( (x^*e_{\widetilde B}x) J1_{B_1}J) =  \Tr_{B_1}( \rE_B(1_{B_1}xx^*1_{B_1}))$ for all $x \in M$ as explained in Section \ref{section-dimension} of the Appendix.

\begin{rem}\label{remark intertwining}
We observe the following basic facts for the embedding $\preceq$. Keep the notation $A,B,M,\rE_A,$ and $\rE_B$ as in Definition \ref{definition intertwining}.
\begin{enumerate}
	\item In the definition of $A\preceq_MB$, we may relax the requirement that the nonzero element $v\in eMf$ is a partial isometry. Indeed, for $e,f, \theta$ as in Definition \ref{definition intertwining}, assume that there exists a nonzero element $x\in eMf$ such that $ax=x\theta(a)$ for all $a\in eAe$. Write $x=v|x|$ for the polar decomposition of $x \in eMf$. Then for all $a\in \mathcal{U}(eAe)$, we have $v|x| = x=ax\theta(a^*)=av\theta(a^*)|x|$ and hence $v=av\theta(a^*)$ by uniqueness of the polar decomposition. Thus, $v \in e M f$ is a nonzero partial isometry such that $av=v\theta(a)$ for all $a\in eAe$.
	\item Let $p\in A$ or $p \in A'\cap 1_AM1_A$ and $q\in B$ or $q \in B'\cap 1_BM1_B$ be any nonzero projections. Then $pAp\subset pMp$ and $qBq\subset qMq$ are with expectation (see \cite[Proposition 2.2]{HU15}). Moreover, $pAp\preceq_MqBq$ implies that $A\preceq_MB$. 

\vspace{1mm}
\noindent
Let $z_i\in \mathcal{Z}(A)$ and $w_j\in\mathcal{Z}(B)$ be any nonzero central projections such that $1_A=\sum_iz_i$ and $1_B=\sum_jw_j$. Then $A\preceq_MB$ if and only if there exist $i, j$ such that $Az_i\preceq_M Bw_j$. Indeed, if $A\preceq_MB$ with $e,f,v, \theta$ as in Definition \ref{definition intertwining}, then there exist $i, j$ such that $z_ivw_j\neq0$ and hence $\theta(ez_i)w_j\neq0$. Observe that the unital inclusion $\theta(eAez_i)w_j \subset \theta(ez_i) Bw_j \theta(ez_i)$ is with expectation by \cite[Proposition 2.2]{HU15}. Then $ez_i, \theta(ez_i) w_j, z_ivw_j, \theta(\, \cdot\, z_i)w_j$ together with item (1) above witness the fact that $Az_i\preceq_M Bw_j$.
	\item Assume $A\preceq_MB$ and take $e,f,v, \theta$ as in Definition \ref{definition intertwining}. Define the unital normal $\ast$-homomorphism $\psi\colon eAe\to \theta(eAe)v^*v : a \mapsto \theta(a)v^*v = v^*av$ and let $z\in\mathcal{Z}(eAe)=\mathcal{Z}(A)e$ be the unique projection such that $\ker(\psi)=eAe(e - z)$. Then up to replacing $e$ by $ez$ and $\theta$ by $\theta|_{eAez}$ ({\em n.b.}~the unital inclusion $\theta(eAez) \subset \theta(ez)  B  \theta(ez)$ is with expectation), we may assume that $\psi$ is injective and moreover $e_0v\neq0$ for any nonzero subprojection $e_0\leq e$ in $A$. In this case, we have $e_0Ae_0\preceq_MB$ for any nonzero subprojection $e_0\leq e$ in $A$ ({\em n.b.}~the unital inclusion $\theta(e_0Ae_0) \subset \theta(e_0)  B  \theta(e_0)$ is with expectation). 
	\item Let $p\in A$ be any projection such that $z_A(p)=1_A$ where $z_A(p)$ denotes the central support in $A$ of the projection $p \in A$. Then $A\preceq_MB$ if and only if $pAp\preceq_MB$. Indeed, assume that $A\preceq_MB$ and take $e,f,v, \theta$ as in item (3). Since $z_A(p)=1$, there exist nonzero  subprojections $e_0\leq e$ and $p_0\leq p$ that are equivalent in $A$. Since $e_0Ae_0\preceq_MB$, we have $p_0Ap_0\preceq_MB$ and hence $pAp\preceq_MB$. A similar statement for $q\in B$ with central support equals to $1_B$ will be proved in Remark \ref{central support for B}.
	\item Obviously, the condition $A\preceq_MB$ does not depend on the choices of $\rE_A$ and $\rE_B$. In Theorem \ref{intertwining for type III} below, when $A$ is finite, we give various characterizations of $A\preceq_MB$, in which we use explicitely the faithful normal conditional expectation $\rE_B$ and a fixed faithful normal semifinite trace $\Tr_{B_1}$ on the semifinite direct summand $B_1$ of $B$. However, since the definition of $A\preceq_MB$ depends neither on $\rE_B$ nor on $\Tr_{B_1}$, all the characterizations in Theorem \ref{intertwining for type III} hold true for any faithful normal conditional expectation $\rE_B : M \to B$ and any faithful normal semifinite trace $\Tr_{B_1}$ on the semifinite direct summand $B_1$ of $B$.
\end{enumerate}
\end{rem}

Keep the same notation as in Definition \ref{definition intertwining} and moreover assume that $A$ is finite. We fix the following setup. 

Put $\widetilde{A}:=A\oplus \C(1_M-1_A)$ and extend $\rE_A : 1_A M 1_A \to A$ to a faithful normal conditional expectation $\rE_{\widetilde{A}} : M \to \widetilde{A}$. Choose a faithful normal trace $\tau_{\widetilde{A}} \in (\widetilde{A})_\ast$ and put $\Psi := \tau_{\widetilde{A}}\circ \rE_{\widetilde{A}}\circ \rT_M$. Observe that $\Psi$ is a faithful normal semifinite weight on $\langle M, \widetilde B\rangle$, $1_A \in \langle M, \widetilde B\rangle^\Psi$ and $A \subset 1_A\langle M, \widetilde B\rangle^\Psi 1_A$. The main result of this section is the following generalization of \cite[Theorem A.1]{Po01}.

\begin{thm}\label{intertwining for type III}
Keep the same notation as in Definition \ref{definition intertwining} and assume that $A$ is finite. Then the following conditions are equivalent.
	\begin{enumerate}
		\item $A \preceq_M B$.
		\item At least one of the following conditions holds:
		\begin{enumerate}
\item There exist a projection $e \in A$ and a finite trace projection $f \in B_1$, a nonzero partial isometry $v \in eMf$ and a unital normal $\ast$-homomorphism $\theta : eAe \to fB_1f$ such that $av = v \theta(a)$ for all $a \in eAe$.
\item There exist projections $e \in A$ and $f \in B_2$, a nonzero partial isometry $v \in eMf$ and a unital normal $\ast$-homomorphism $\theta : eAe \to fB_2f$ such that the inclusion $\theta(eAe) \subset fB_2f$ is with expectation and $av = v \theta(a)$ for all $a \in eAe$.
\end{enumerate}
\item There exist $n \geq 1$, a projection $q \in B \otimes \mathbf M_n$, a nonzero partial isometry $w\in (1_A\otimes e_{1,1})(M\otimes\mathbf{M}_{n})q$ and a unital normal $\ast$-homomorphism $\pi\colon A \rightarrow q(B\otimes\mathbf{M}_n)q$ such that the inclusion $\pi(A) \subset q (B \otimes \mathbf M_n) q$ is with expectation and $(a\otimes 1_n)w = w\pi(a)$ for all $a\in A$, where $(e_{i,j})_{1 \leq i,j \leq n}$ is a fixed matrix unit in $\mathbf{M}_n$. 
\item There exists no net $(w_i)_{i \in I}$ of unitaries in $\mathcal U(A)$ such that $\rE_{B}(b^*w_i a)\rightarrow 0$ in the $\sigma$-$\ast$-strong topology for all $a,b\in 1_AM1_B$.
\item For any $\sigma$-weakly dense subset $X \subset M$, there exists no net $(w_i)_{i \in I}$ of unitaries in $\mathcal U(A)$ such that $\rE_{B}(b^*w_i a)\rightarrow 0$ in the $\sigma$-strong topology for all $a,b\in 1_AX1_B$.
\item There exists a nonzero positive element $d\in A' \cap 1_A \langle M,\widetilde{B}\rangle1_A$ such that $$d \, 1_A \, J1_{B}J=d, \quad \Tr(d \, J1_{B_1}J)< +\infty \quad \text{and} \quad \rT_M(d \, J1_{B_2}J)\in 1_AM1_A.$$
\end{enumerate}
	\end{thm}

Before proving Theorem \ref{intertwining for type III}, we first recall a simple lemma that we will need for the proof of Theorem \ref{intertwining for type III}. 

\begin{lem}\label{Lemma1}
Let $\mathcal M$ be any von Neumann algebra and $\Theta$ any faithful normal semifinite weight on $\mathcal M$. Then the map $x\mapsto \Lambda_\Theta(x)$ is $\sigma$-weak--weak continuous from $\Omega:=\{x\in \mathcal M\mid \Theta(x^*x)\leq1 \}$ to $\rL^2(\mathcal M,\Theta)$. 

It follows that if $\mathcal{C}$ is a $\sigma$-weakly closed convex subset of $\mathcal M$ that is both bounded for the uniform norm and for the $\|\cdot\|_\Theta$-norm, then $\Lambda_\Theta(\mathcal{C})$ is $\|\cdot\|_\Theta$-closed in $\rL^2(\mathcal M,\Theta)$.
\end{lem}
\begin{proof}
Let $x_i\in\Omega$ be a net converging to $x\in \Omega$ in the $\sigma$-weak topology. We show that $\langle \Lambda_\Theta(x-x_i) , \eta\rangle_\Theta$ converges to zero as $i \to \infty$, where $\eta \in \rL^2(\mathcal M, \Theta)$ is of the form $\eta=J_\Theta\sigma_{{\rm i}/2}^\Theta(a)J_\Theta \, \Lambda_\Theta(b)$ for some analytic element $a\in \mathcal M$ with $\Theta(a^*a)<+\infty$ and some $b\in \mathcal M$ with $\Theta(b^*b)<+\infty$. Note that since the subspace spanned by such elements $\eta$ is $\|\cdot\|_\Theta$-dense in $\rL^2(\mathcal M,\Theta)$ and since the vectors $\Lambda_\Theta(x-x_i)$ are bounded in $\rL^2(\mathcal M,\Theta)$, this will indeed prove that $\Lambda_\Theta(x-x_i) \to 0$ weakly in $\rL^2(\mathcal M,\Theta)$ as $i \to \infty$.

We have
	\begin{align*}
		\langle \Lambda_\Theta(x - x_i) , J_\Theta\sigma_{{\rm i}/2}^\Theta(a)J_\Theta \, \Lambda_\Theta(b)\rangle_\Theta
		&= \langle J_\Theta\sigma_{-{\rm i}/2}^\Theta(a^*)J_\Theta \, \Lambda_\Theta(x - x_i) , \Lambda_\Theta(b)\rangle_\Theta \\
		&= \langle \Lambda_\Theta((x - x_i) a) , \Lambda_\Theta(b)\rangle_\Theta \\
		&= \langle (x - x_i)\Lambda_\Theta(a) , \Lambda_\Theta(b)\rangle_\Theta \rightarrow 0 \quad \text{as} \quad i \to \infty.
	\end{align*}
	
Next, let $\mathcal{C}$ be a $\sigma$-weakly closed convex subset of $\mathcal M$ that is both bounded for the uniform norm and for the $\|\cdot\|_\Theta$-norm. Then since $\mathcal{C}$ is $\sigma$-weakly compact and since the map given in the statement is $\sigma$-weak--weak continuous, $\Lambda_\Theta(\mathcal{C})$ is weakly compact and hence is weakly closed in $\rL^2(\mathcal{M},\Theta)$. Since $\Lambda_\Theta(\mathcal{C})$ is a convex set, it is $\|\cdot\|_\Theta$-norm closed in $\rL^2(\mathcal{M},\Theta)$ by the Hahn-Banach separation theorem.
\end{proof}

\begin{proof}[Proof of Theorem \ref{intertwining for type III}]
$(1) \Rightarrow (2)$ By Remark \ref{remark intertwining}.(2), we have either $A\preceq_MB_1$ or $A\preceq_MB_2$. Since the condition $A \preceq_M B_2$ exactly means (2-b),  we only need to show that if $A \preceq_M B_1$, then Condition (2-a) holds. Let $e, f, v, \theta$ witnessing the fact that $A \preceq_M B_1$. Fix a faithful normal trace $\tau \in (\theta(eAe) \oplus \C(1_{B_1} - f))_\ast$ and denote by $\rE : B_1 \to \theta(eAe) \oplus \C(1_{B_1} - f)$ a faithful normal conditional expectation. Define the faithful normal state $\phi = \tau \circ \rE \in (B_1)_\ast$. There exists a positive nonsingular element $T \in \rL^1(B_1, \Tr_{B_1})_+$ such that $\phi = \Tr_{B_1}(\,\cdot \, T)$. Denote by $D \subset B_1$ the abelian von Neumann subalgebra generated by $\left\{ T^{{\rm i}t} : t \in \R\right\}$. We have $(B_1)^\phi = D' \cap B_1$ and hence $\theta(eAe) \oplus \C(1_{B_1} - f) \subset (B_1)^\phi = D' \cap B_1$. We may then choose a nonzero large enough finite trace projection $f_1 \in D$ of the form $f_1 = \mathbf 1_{[\frac1k, +\infty)}(T)$ for some $k \geq 1$ such that $vf_1 \neq 0$. So, up to replacing $f$ by $ff_1$, $v$ by $v f_1$ and $\theta$ by $\theta f_1$ and using Remark \ref{remark intertwining}.(1), we may assume that the projection $f \in B_1$ is of finite trace and that the nonzero partial isometry $v$ satisfies $v \in eMf$. Observe that since $f B_1 f$ is finite, the unital inclusion $\theta(eAe) \subset fB_1 f$ is with expectation.

$(2) \Rightarrow (1)$ It is obvious.

$(1) \Rightarrow (3)$ Let $e, f, v, \theta$ witnessing the fact that $A \preceq_M B$ as in Remark \ref{remark intertwining}.(3). Since $A$ is finite, \cite[Proposition 8.2.1]{KR97} implies that there exist $n \geq 1$ and nonzero pairwise equivalent orthogonal projections $e_1, \dots, e_n \in A$ such that $e_1 \leq e$ and $r = \sum_{i = 1}^n e_i \in \mathcal Z(A)$. Observe that $e_1v \neq 0$ by the choice of $e \in A$ as in Remark \ref{remark intertwining}.(3). For every $1 \leq i \leq n$, let $u_i \in A$ be a partial isometry satisfying $u_i^* u_i = e_1$ and $u_i u_i^* = e_i$. Put $q = \Diag(\theta(e_1))_{i}$, $w = [u_1v \cdots u_nv] \in (1_A\otimes e_{1,1})(M\otimes\mathbf{M}_{n})q$ and $\pi : A \to q(B \otimes \mathbf M_n)q : x \mapsto [\theta(u_i^* x u_j)]_{i, j}$. Note that $w \neq 0$. We have $(a \otimes 1_n) w = w \pi(a)$ for all $a \in A$. Observe that the unital inclusion $\theta(e_1 A e_1)  \oplus \C(1_B - \theta(e_1)) \subset B$ is with expectation and so is the unital inclusion $\left(\theta(e_1 A e_1)  \oplus \C(1_B - \theta(e_1)) \right) \otimes \mathbf M_n \subset B \otimes \mathbf M_n$. This implies that $q\left(\left(\theta(e_1 A e_1)  \oplus \C(1_B - \theta(e_1)) \right) \otimes \mathbf M_n\right) q \subset q(B \otimes \mathbf M_n)q$ is with expectation. Since the inclusion $\pi(A) \subset q\left(\left(\theta(e_1 A e_1)  \oplus \C(1_B - \theta(e_1)) \right) \otimes \mathbf M_n\right) q$ is unital and $q\left(\left(\theta(e_1 A e_1)  \oplus \C(1_B - \theta(e_1)) \right) \otimes \mathbf M_n\right) q$ is finite, this implies that the unital inclusion $\pi(A) \subset q(B \otimes \mathbf M_n)q$ is with expectation.

$(3) \Rightarrow (4)$ Let $n, w, \pi$ as in $(3)$. Denote by $\tr_n$ the normalized trace on $\mathbf M_n$. Put $\varphi_n:=\varphi \otimes \mathrm{tr}_n$ on $M \otimes \mathbf M_n$.  Suppose by contradiction that there exists a net of unitaries $(u_i)_i$ in $\mathcal U(A)$ as in $(4)$. Then we have
	\begin{align*}
		\|\rE_{B\otimes \mathbf{M}_n}(w^*w) \|_{\varphi_n}
				&=\|\pi(u_i)\rE_{B\otimes \mathbf{M}_n}(w^*w) \|_{\varphi_n}\\
		&=\|\rE_{B\otimes \mathbf{M}_n}(\pi(u_i)w^*w) \|_{\varphi_n}\\
		&=\|\rE_{B\otimes \mathbf{M}_n}(w^*(u_i\otimes 1_n)w) \|_{\varphi_n}\rightarrow 0 \quad \text{as} \quad i \to \infty.
	\end{align*}
Thus, we obtain $\rE_{B\otimes \mathbf{M}_n}(w^*w) = 0$ and hence $w=0$, which is a contradiction.

$(4) \Rightarrow (5)$ We prove the implication by contraposition using ultraproduct techniques. Let $X \subset M$ be a $\sigma$-weakly dense subset and $(w_i)_{i \in I}$ a net of unitaries in $\mathcal U(A)$ such that $\rE_B(b^* w_i a) \to 0$ in the $\sigma$-strong topology as $i \to \infty$ for all $a, b \in 1_A X 1_B$. Fix a cofinal ultrafilter $\mathcal U$ on the directed set $I$. We will be working inside the ultraproduct von Neumann algebra $M^{\mathcal U}$. Recall that $M \subset M^{\mathcal U}$ is a von Neumann subalgebra with faithful normal conditional expectation $\rE_{\mathcal U} : M^{\mathcal U} \to M$ (see Section \ref{preliminaries} for further details). 

Since $\widetilde A$ is a finite von Neumann algebra, the uniformly bounded net $(w_i)_{i \in I}$ defines an element $W = (w_i)^{\mathcal U} \in (\widetilde A)^{\mathcal U}$. Since $\widetilde A \subset M$ (resp.\ $\widetilde B \subset M$) is a von Neumann subalgebra with expectation, it follows that $(\widetilde A)^{\mathcal U} \subset M^{\mathcal U}$ (resp.\ $(\widetilde B)^{\mathcal U} \subset M^{\mathcal U}$) is a von Neumann subalgebra with expectation. We then have $W = (w_i)^{\mathcal U} \in (\widetilde A)^{\mathcal U} \subset M^{\mathcal U}$ and $\rE_{(\widetilde B)^{\mathcal U}}(b^* W a) = (\rE_{\widetilde B}(b^* w_i a))^{\mathcal U}$ for all $a, b \in 1_A M 1_B$. Since $\mathcal U$ is a cofinal ultrafilter on the directed set $I$, for all $a, b \in 1_A X 1_B$, we have 
$$\left\|\rE_{(\widetilde B)^{\mathcal U}}(b^* W a)\right\|_{\varphi^{\mathcal U}} = \lim_{i \to \mathcal U} \|\rE_{\widetilde B}(b^* w_i a)\|_\varphi = \lim_{i \to \mathcal U} \|\rE_B(b^* w_i a)\|_\varphi = 0$$ 
and hence $\rE_{(\widetilde B)^{\mathcal U}}(b^*W a) = 0$. Since $\rE_{(\widetilde B)^{\mathcal U}}$ is moreover normal, we obtain $\rE_{(\widetilde B)^{\mathcal U}}(b^* W a) = 0$ for all $a, b \in 1_A M 1_B$. This means that 
$$\lim_{i \to \mathcal U} \|\rE_{B}(b^* w_i a)\|_\varphi^\sharp = \lim_{i \to \mathcal U} \|\rE_{\widetilde B}(b^* w_i a)\|_\varphi^\sharp =  \|\rE_{(\widetilde B)^{\mathcal U}}(b^* W a)\|_{\varphi^{\mathcal U}}^\sharp = 0$$ for all $a, b \in 1_A M 1_B$. Therefore, for every $\varepsilon > 0$ and every finite subset $\mathcal F \subset 1_A M 1_B$, there exists $i = i(\varepsilon, \mathcal F) \in I$ such that $\|\rE_B(b^* w_i a)\|_\varphi^\sharp < \varepsilon$ for all $a, b \in \mathcal F$.

$(5)\Rightarrow (6)$ Consider the $\sigma$-weakly dense subset $X \subset M$ defined by
$$X = \bigcup \left \{ M ( p + (1_M - 1_{B_1})) \mid p \in B_1 \text{ is a finite trace projection} \right \}.$$
By assumption, there exist $\delta>0$, a finite trace projection $p \in B_1$ and a finite subset $\mathcal{F}\subset 1_AM(p + 1_{B_2})$ such that 
$$\sum_{x,y\in \mathcal{F}}\|\rE_B(y^*wx)\|_{\varphi}^2>\delta, \forall w\in\mathcal{U}(A).$$ 
Put $d_0:=\sum_{y\in \mathcal{F}}ye_{\widetilde{B}}y^*\in (1_A\langle M,\widetilde{B} \rangle1_A)^+$ and observe that $\rT_M(d_0) = \sum_{y \in \mathcal F} yy^* \in 1_AM1_A$ and $\Tr(d_0 \, J1_{B_1}J) = \sum_{y\in \mathcal F} \Tr_{B_1}(\rE_B(1_{B_1}y^*y1_{B_1})) < +\infty$. We also have $d_0 \, 1_A \, J 1_{B} J=d_0$, since $1_B \in \mathcal Z(\widetilde B)$ and $e_{\widetilde{B}} \, J1_{B}J=e_{\widetilde{B}}1_{B}$.

Denote by $\mathcal K$ the $\sigma$-weak closure in $1_A\langle M, \widetilde B\rangle1_A$ of the convex  hull of the uniformly bounded subset $\left \{ w^* d_0 w \mid w \in \mathcal U(A)\right \} \subset (1_A \langle M, \widetilde B\rangle 1_A)^+$, that is, 
$$\mathcal K := \overline{{\rm co}}^w \left \{ w^* d_0 w \mid w \in \mathcal U(A)\right \} \subset (1_A \langle M, \widetilde B\rangle 1_A)^+.$$ 
Then $\mathcal K$ is uniformly bounded. Observe that for all $y \in \langle M, \widetilde B\rangle^+$, we have $y^*y = y^{1/2} \, y \, y^{1/2} \leq y^{1/2} \, \|y\|_\infty1 \,y^{1/2} = \|y\|_\infty y$ and hence $\|y\|_\Psi = \Psi(y^* y)^{1/2} \leq (\|y\|_\infty)^{1/2} \Psi(y)^{1/2}$.  
By item $(1)$ in Lemma \ref{operator valued weight1}, we have $\sigma^\Psi |_ {\widetilde A} = \sigma^{\tau_{\widetilde A} \circ \rE_{\widetilde A}} |_{\widetilde A} = \id_{\widetilde A}$. This implies that $\Psi(y) = \Psi(d_0)$ for all $y \in {\rm co} \{ w^* d_0 w \mid w \in \mathcal U(A) \}$. We claim that $\mathcal K$ is bounded in $\|\cdot\|_\Psi$-norm by $(\|d_0\|_\infty)^{1/2} \Psi(d_0)^{1/2}$. Indeed, let $x \in \mathcal K$ and choose a net $(x_i)_{i \in I}$ in ${\rm co} \{ w^* d_0 w \mid w \in \mathcal U(A) \}$ such that $x_i \to x$ $\sigma$-weakly. Since $\Psi(x_i) = \Psi(d_0)$ for all $i \in I$ and since the weight $\Psi$ is $\sigma$-weakly lower semi-continuous on $\langle M, \widetilde B\rangle^+$ by \cite[Theorem VII.1.11 (iii)]{Ta03}, we have 
$$\Psi(x) \leq \liminf_{i \in I} \Psi(x_i) = \Psi(d_0).$$
This implies that $\|x\|_\Psi \leq (\|x\|_\infty)^{1/2} \Psi(x)^{1/2} \leq (\|d_0\|_\infty)^{1/2} \Psi(d_0)^{1/2} < +\infty$.

Using Lemma \ref{Lemma1}, we may regard $\mathcal K$ as a closed convex bounded subset of $\rL^2(\langle M, \widetilde B\rangle, \Psi)$. In particular, there exists a unique element $d \in \mathcal K$ of minimal $\|\cdot\|_\Psi$-norm. We still have $d \, 1_A \, J 1_B J=d$. Since $A\subset 1_A\langle M,\widetilde{B}\rangle^{\Psi}1_A$, we have $\|wdw^*\|_\Psi=\|d\|_\Psi$  for all $w\in \mathcal{U}(A)$. Thus, by minimality of the $\|\cdot\|_\Psi$-norm, we have $wdw^*=d$ for all $w\in \mathcal{U}(A)$ and hence $d\in A'\cap 1_A\langle M, \widetilde B\rangle 1_A$. 

We first show that $d\neq0$. Indeed, for all $w \in \mathcal{U}(A)$, we have
\begin{equation*}
	\sum_{x\in\mathcal{F}}\langle w^*d_0w \, \Lambda_\varphi(x),\Lambda_\varphi(x)\rangle_{\varphi}
	=\sum_{x,y\in\mathcal{F}}\langle e_{\widetilde{B}} \Lambda_\varphi(y^* w x), \Lambda_\varphi(y^* w x)\rangle_{\varphi}
	=\sum_{x,y\in\mathcal{F}}\| \rE_B( y^*w x)\|_{\varphi}^2>\delta.
\end{equation*}
By taking convex combinations and $\sigma$-weak limits, we obtain $\sum_{x\in\mathcal{F}}\langle d  \Lambda_\varphi(x),\Lambda_\varphi(x)\rangle_{\varphi} \geq \delta$ and hence $d \neq 0$.

We next show that $\Tr(d \, J1_{B_1}J)<+\infty$. Indeed, for all $w \in \mathcal{U}(A)$, we have 
$$\Tr(w^*d_0 w \, J1_{B_1}J) = \Tr((w \, J1_{B_1} J)^* \, d_0 J1_{B_1}J \, (w \, J1_{B_1} J)) =  \Tr(d_0 \, J1_{B_1}J)$$
and hence $\Tr(y \, J1_{B_1}J) = \Tr(d_0 \, J1_{B_1}J)$ for all $y \in {\rm co} \{ w^* d_0 w \mid w \in \mathcal U(A) \}$. Let $(x_i)_{i \in I}$ be any net in $\mathrm{co}\{w^*d_0w\mid w\in\mathcal{U}(A)\}$ that converges to $d \in \mathcal K$ in the $\sigma$-weak topology. Since $\Tr(\, \cdot \, J 1_{B_1} J)$ is $\sigma$-weakly lower semi-continuous on $(\langle M, \widetilde B\rangle J 1_{B_1} J)^+$ by \cite[Theorem VII.1.11 (iii)]{Ta03}, we have 
$$\Tr(d \, J 1_{B_1} J)\leq \liminf_{i \in I} \Tr(x_i \, J 1_{B_1} J) = \Tr(d_0 \, J 1_{B_1} J) < +\infty.$$

We finally show that $\rT_M(d \, J1_{B_2}J)\in 1_AM1_A$. This will be a consequence of the next claim.

\begin{claim}
	We have $\phi(\rT_M(d))\leq \|\phi\| \, \|\rT_M(d_0)\|_\infty$ for all normal positive linear functionals $\phi \in M_\ast$. 
\end{claim}

\begin{proof}[Proof of the Claim]
	We fix a normal positive linear functional $\phi \in M_\ast$ and observe that $\phi\circ \rT_M$ is a normal semifinite weight on $\langle M,\widetilde{B}\rangle$. Let $(x_i)_{i \in I}$ be any net in $\mathrm{co}\{w^*d_0w\mid w\in\mathcal{U}(A)\}$ that converges to $d \in \mathcal K$ in the $\sigma$-weak topology. Since $\phi(\rT_M(x_i))\leq \|\phi\| \, \|\rT_M(d_0)\|_\infty$ for all $i \in I$ and since $\phi\circ \rT_M$ is $\sigma$-weakly lower semi-continuous on $\langle M, \widetilde B\rangle^+$ by \cite[Theorem VII.1.11 (iii)]{Ta03}, we have 
$$(\phi\circ \rT_M)(d)\leq \liminf_{i \in I} (\phi\circ \rT_M)(x_i)\leq \|\phi\| \,\| \rT_M(d_0)\|_\infty.$$
This finishes the proof of the claim.
\end{proof}

Recall that any normal linear functional $\phi \in M_\ast$ has a unique decomposition $\phi = (\phi_1 - \phi_2) + {\rm i}(\phi_3 - \phi_4)$ where $\phi_1, \phi_2, \phi_3, \phi_4 \in M_\ast$ are normal positive linear functionals such that $\|\phi_1 - \phi_2\| = \|\phi_1\| + \|\phi_2\|$ and $\|\phi_3 - \phi_4\| = \|\phi_3\| + \|\phi_4\|$. By the Claim, we obtain
\begin{align*}
|\phi(\rT_M(d))| &=  \left|(\phi_1(\rT_M(d)) - \phi_2(\rT_M(d))) + {\rm i}(\phi_3(\rT_M(d)) - \phi_4(\rT_M(d)) \right| \\
& \leq (\|\phi_1\| +\|\phi_2\| +\|\phi_3\| +\|\phi_4\|) \|\rT_M(d_0)\|_\infty \\
& = (\|\phi_1 - \phi_2\| +\|\phi_3 - \phi_4\|) \|\rT_M(d_0)\|_\infty \\
& \leq 2 \|\phi\| \, \|\rT_M(d_0)\|_\infty.
\end{align*}
Therefore, we obtain $\rT_M(d) \in (M_\ast)^* = M$ and hence $\rT_M(d) \in 1_A M 1_A$. In particular, we have $\rT_M(d \, J1_{B_2}J)\in 1_AM1_A$.

$(6)\Rightarrow (2)$ Take a nonzero spectral projection $p$ of $d$ such that $p\leq \lambda d$ for some $\lambda>0$. Then $p$ satisfies exactly the same assumption as $d$, namely $p \, 1_A \, J1_{B}J=p$,  $\Tr(p \, J1_{B_1}J)< +\infty$ and $\rT_M(p \, J1_{B_2}J)\in 1_AM1_A$. Since $p=p \, J1_BJ$, we have either 
 $p \, J1_{B_1}J\neq0$ or $p \, J1_{B_2}J\neq0$.

We first assume that $p \, J1_{B_1}J\neq0$. 
Since $\Tr_{B_1}$ is semifinite, there is a $\Tr_{B_1}$-finite projection $q_0 \in B_1$ which is sufficiently close to $1_{B_1}$. We can particularly choose $q_0$ such that $pJq_0J \neq 0$. Note that this is equivalent to $pJzJ \neq 0$, where $z:=z_{B_1}(q_0)$ denotes the central support in $B_1$ of the projection $q_0 \in B_1$. 

Here we claim that the right dimension of the nonzero $A$-$q_0B_1q_0$-bimodule $pJq_0J\rL^2(M,\varphi)$ with respect to $(q_0B_1q_0, \Tr_{B_1}(q_0\, \cdot \, q_0))$ coincides with $\Tr(pJzJ)$, so that it has a nonzero finite value. 
To see this, put $q_1:=1_{\widetilde{B}} - z$, $q:=q_0+q_1 \in \widetilde{B}$ and observe that $z_{\widetilde{B}}(q)=1_{\widetilde{B}}=1_M$ and $z q=q_0$. Then we apply Proposition \ref{general dimension lemma}(2) and obtain that 
	$$\dim_{(q_0Bq_0,\Tr_B(q_0\, \cdot \, q_0))} p Jq_0J\rL^2(M,\varphi)  = \Tr_{\langle M,\widetilde{B} \rangle JzJ}(pJzJ) ,$$
where $\Tr_{\langle M , \widetilde{B}\rangle JzJ}$ is the unique trace on $\langle M,\widetilde{B}\rangle JzJ$ from Proposition \ref{general dimension lemma}. Since $\Tr_{\langle M,\widetilde{B} \rangle JzJ}=\Tr$ on $\langle M,\widetilde{B} \rangle JzJ$ by uniqueness (see Proposition \ref{general dimension lemma}(1) and Remark \ref{remark-semifinite}), this shows the claim. 

Thus, we obtain a nonzero $A$-$q_0B_1 q_0$-subbimodule of $\rL^2(M,\varphi)$ such that its right dimension is finite, where $q_0\in B_1$ is a $\Tr_{B_1}$-finite projection. Then we can proceed as in the proof of \cite[Proposition 3.1 (i) $\Rightarrow$ (ii) $\Rightarrow$ (iii)]{Ue12} (see also \cite[Theorem 2.3 (2) $\Rightarrow$ (1)]{HV12}) and we deduce that $A\preceq_Mq_0B_1q_0$. This implies (2-a). Observe that we have proved the following fact.

\begin{fact}\label{fact}
Assume that there exists a nonzero positive element $d\in A'\cap 1_A\langle M,\widetilde{B}\rangle1_A$ such that $d=d \, 1_A \, J1_{B_1}J$ and $\Tr(d)< +\infty$. Let $q_0$ be a projection in $B_1$ such that $\Tr_{B_1}(q_0)< +\infty$. If $d \, Jz_{B_1}(q_0)J\neq0$, then we have $A\preceq_Mq_0B_1q_0$.
\end{fact}

We next assume that $p \,J1_{B_2}J\neq0$. Up to replacing $p$ by $p\,J1_{B_2}J$, we may assume that $p=p \, J1_{B_2}J$. Applying Lemma \ref{operator valued weight2} to the inclusion $A \subset 1_A \langle M, \widetilde B\rangle 1_A$ and to the faithful normal semifinite operator valued weight $\rE_A \circ \rT_M(1_A \cdot 1_A)$, up to replacing the nonzero projection $p$ by the nonzero projection $zp$ for some nonzero central projection $z \in \mathcal Z(A)$, we may further assume that the inclusion $Ap \subset p \langle M, \widetilde B\rangle p$ is with expectation.

Since $\langle M,\widetilde{B}\rangle J1_{B_2}J= J B_2'J$ is a type ${\rm III}$ von Neumann algebra and since the central support in $J B_2'J$ of the projection $e_{\widetilde{B}} \, J 1_{B_2} J= e_{\widetilde{B}}1_{B_2}$ is equal to $J 1_{B_2} J$, we have $p\preceq e_{\widetilde{B}}1_{B_2}$ in $\langle M,\widetilde{B}\rangle J1_{B_2}J$ by \cite[Theorem 6.3.4]{KR97}. Let $V\in \langle M,\widetilde{B}\rangle J1_{B_2} J$ be a nonzero partial isometry such that $V^*V=p$ and $VV^*\leq e_{\widetilde{B}}1_{B_2}$. Note that $VV^* \in e_{\widetilde{B}}1_{B_2}\langle M,\widetilde{B}\rangle e_{\widetilde{B}}1_{B_2}=\widetilde{B}e_{\widetilde{B}}1_{B_2}\cong B_2$ and so we can write it as $VV^*=fe_{\widetilde{B}}\simeq f$ for some nonzero projection $f\in B_2$. Then we have 
\begin{equation*}
		\Ad(V) \colon p\langle M,\widetilde{B}\rangle p \xrightarrow{\sim} V \, p\langle M,\widetilde{B}\rangle p \, V^*= fe_{\widetilde{B}}\langle M,\widetilde{B}\rangle fe_{\widetilde{B}}=fB_2f e_{\widetilde{B}}\cong fB_2f.
\end{equation*}

Since $p \in A' \cap 1_A \langle M, \widetilde B\rangle 1_A$ and $p = V^*V$, the map $\theta : A \to fB_2f : x \mapsto VxV^*$ defines a unital normal $\ast$-homomorphism. Observe that since the inclusion $Ap \subset p \langle M, \widetilde B\rangle p$ is with expectation and since $V^*V = p$ and $VV^*=fe_{\widetilde{B}}\simeq f$, so is the inclusion 
$$\theta(A) = V A V^* = V \, Ap \, V^* \subset V \,p \langle M, \widetilde B \rangle q\, V^* \cong fB_2f.$$ 
Since $V^*V = p \in A' \cap 1_A \langle M, \widetilde B\rangle 1_A$, we have $\theta(a)V=Va$ for all $a\in A$. Since $V^*V = p$ and $\rT_M(p) \in M$, we have $V \in \mathfrak n_{\rT_M}$. Since $V = e_{\widetilde{B}}V$ and $e_{\widetilde{B}} \in \mathfrak n_{\rT_M}$, we also have $V \in \mathfrak m_{\rT_M}$. We may then apply $\rT_M$ to the equation $\theta(a)V=Va$ and we obtain $\theta(a)\rT_M(V)=\rT_M(V)a$ for all $a \in A$. Since $V = e_{\widetilde B} V= e_{\widetilde B}\rT_M(e_{\widetilde{B}}V) = e_{\widetilde B}\rT_M(V)$ by Proposition \ref{push down lemma} and since $V \neq 0$, we have $\rT_M(V) \neq 0$. Finally, writing $\rT_M(V) = v^* |\rT_M(V)|$ for the polar decomposition of $\rT_M(V)$ in $M$, by Remark \ref{remark intertwining}.(1) we obtain (2-b). This concludes the proof of Theorem \ref{intertwining for type III}.
\end{proof}

\begin{rem}\label{central support for B}
Keep the same notation as in Definition \ref{definition intertwining}. Let $q\in B$ be any projection such that $z_B(q)=1_B$ where $z_B(q)$ denotes the central support in $B$ of the projection $q \in B$. 
Then $A\preceq_MB$ if and only if $A\preceq_MqBq$.

Indeed, by Remark \ref{remark intertwining}.(2), we may assume that $B$ is either semifinite or of type ${\rm III}$. The second case is trivial since $q$ and $z_B(q) = 1_B$ are equivalent in $B$. So, we may  assume that $B$ is semifinite. By definition of the embedding $A \preceq_M B$, only the semifinite direct summand of $A$ can be embedded with expectation into $B$ inside $M$. Therefore, using Remark \ref{remark intertwining}.(4), we may further assume that $A$ is finite. Since $B$ is semifinite, there exists a finite projection $q_0\leq q$ in $B$ such that $z_B(q_0)=1$. Let us fix a faithful normal semifinite trace $\Tr$ on $\widetilde B := B\oplus \C(1_M - 1_B)$ such that $\mathrm{Tr}(q_0)<+\infty$. Since $A\preceq_MB$, one can take a nonzero positive element $d\in A'\cap 1_A\langle M,\widetilde{B}\rangle1_A$ as in Theorem \ref{intertwining for type III}.(6). Since $d \, Jz_{B}(q_0)J=d\neq0$, we can use the Fact in the proof of (6) $\Rightarrow$ (2) in Theorem \ref{intertwining for type III} to obtain that $A\preceq_Mq_0Bq_0$. Thus, we have $A\preceq_MqBq$. 
\end{rem}

In the next lemma, we study the effect of taking tensor products.

\begin{lem}\label{lemma intertwining in tensor}
Let $M$ and $N$ be any $\sigma$-finite von Neumann algebras, $1_A$ and $1_B$ any nonzero projections in $M$, $A \subset 1_A M 1_A$ any finite von Neumann subalgebra with expectation and $B\subset 1_{B}M1_{B}$ any von Neumann subalgebra with expectation $\rE_B : 1_B M 1_B \to B$. 
We will simply denote by $B \ovt N$ the von Neumann subalgebra of $(1_B \otimes 1_N)(M \ovt N)(1_B \otimes 1_N)$ generated by $B \otimes \C 1_N$ and $\C1_B \otimes N$.

The following conditions are equivalent:
\begin{enumerate}
\item $A \preceq_M B$.
\item $A \otimes \C 1_N \preceq_{M \ovt N} B \otimes \C 1_N$.
\item  $A \otimes \C 1_N \preceq_{M \ovt N} B \ovt N$.
\end{enumerate}
\end{lem}

\begin{proof}
It is obvious that $(1) \Rightarrow (2) \Rightarrow (3)$.

 $(3) \Rightarrow (1)$ By contraposition, we assume that $A \npreceq_M B$ and we show that $A \otimes \C 1_N \npreceq_{M \ovt N} B \ovt N$. By Theorem \ref{intertwining for type III} (4), there exists a net of unitaries $(w_i)_{i \in I}$ in $\mathcal U(A)$ such that $\rE_B(b^* w_i a) \to 0$ $\sigma$-strongly as $i \to \infty$ for all $a, b \in 1_A M 1_B$. Observe that the mapping
$$\rE_{B \ovt N} : (1_B \otimes 1_N)(M \ovt N)(1_B \otimes 1_N) \to B \ovt N : b \otimes y \mapsto \rE_B(b) \otimes y$$ defines a faithful normal conditional expectation. Put $v_i = w_i \otimes 1_N \in \mathcal U(A \otimes \C 1_N)$ for all $i \in I$. For all $a, b \in 1_A M 1_B$ and all $x, y \in N$, we have $\rE_{B \ovt N}((b \otimes y)^* v_i (a \otimes x)) = \rE_B(b^* w_i a) \otimes y^*x$ and hence $\rE_{B \ovt N}((b \otimes y)^* v_i (a \otimes x)) \to 0$ $\sigma$-strongly as $i \to \infty$. By Theorem \ref{intertwining for type III} (5), this implies that $A \otimes \C 1_N \npreceq_{M \ovt N} B \ovt N$.
\end{proof}

The next corollary will be useful in the proof of Theorem \ref{location theorem} below.

\begin{cor}\label{cor F}
Let $M$ be any von Neumann algebra with separable predual, $1_A$ and ${1_B}_n$ $(n \in \N)$ any nonzero projections in $M$, $A \subset 1_A M 1_A$ any finite von Neumann subalgebra with expectation and $B_n\subset 1_{B_n}M1_{B_n}$ any von Neumann subalgebra with expectation for all $n\in \N$. If $A\npreceq_MB_n$ for all $n\in \N$, then there exists a diffuse abelian von Neumann subalgebra $A_0\subset A$ such that $A_0\not\preceq_MB_n$ for all $n\in \N$.
\end{cor}

\begin{proof}
Since $M$ is assumed to have separable predual, the net that appears in item $(5)$ of Theorem \ref{intertwining for type III} for $A \npreceq_M B_n$ can be taken to be a sequence for all $n \in \N$. Then the proof of \cite[Corollary F.14]{BO08} applies {\em mutatis mutandis}.
\end{proof}

\subsection*{Further results}

In this subsection, we gather various useful facts and permanence properties of the symbol $A \preceq_M B$ when $A \subset 1_AM1_A$ and $B \subset 1_B M1_B$ are any von Neumann subalgebras with expectation. We start by studying the effect of taking {\em unital} subalgebras of $A$.

\begin{lem}\label{intertwining subalgebra}
Let $M$ be any $\sigma$-finite von Neumann algebra, $1_A$ and $1_B$ any nonzero projections in $M$, $A\subset 1_AM1_A$ and $B\subset 1_BM1_B$ any von Neumann subalgebras with expectation. Let $D\subset A$ be any {\em unital} von Neumann subalgebra with expectation. If $A\preceq_MB$, then $D\preceq_MB$.
\end{lem}
\begin{proof}
Write $1_A=z_1+z_2$ with central projections $z_1, z_2\in\mathcal{Z}(A)$ such that $Az_1$ is semifinite and $Az_2$ is of type ${\rm III}$. By Remark \ref{remark intertwining}.(2), we have either $Az_1\preceq_MB$ or $Az_2\preceq_MB$.

First assume that $Az_2\preceq_MB$. Since the unital inclusion $Dz_2 \subset Az_2$ is with expectation by Remark \ref{remark expectation bis}, we may assume without loss of generality that $z_2 = 1_A$. Take $e, f, v, \theta$ as in Definition \ref{definition intertwining}. Then $e$ is equivalent to its central support $z_{A}(e)$ in $A$ and hence we may assume that $e\in \mathcal{Z}(A)$. Using Remark \ref{remark intertwining}.(3), we may further assume that the unital normal $\ast$-homomorphism $\psi : Ae \to v^*v \theta(Ae) : a \mapsto v^*v \theta(a) = v^*av$ is injective. This implies in particular that the unital normal $\ast$-homomorphism $\theta : Ae \to fBf$ is injective. By Remark \ref{remark expectation bis}, the unital inclusion $De \subset Ae$ is with expectation. Since $\theta : Ae \to \theta(Ae)$ is a unital normal $\ast$-isomorphism, the unital inclusion $\theta(De) \subset \theta(A e)$ is also with expectation. Since the unital inclusion $\theta(A e) \subset \theta(e) B \theta(e)$ is with expectation, so is the unital inclusion $\theta(De) \subset \theta(e) B \theta(e)$. Then, taking the restriction $\theta|_{De}$ of $\theta : Ae \to fBf$ to $De$ shows that $De\preceq_MB$. Thus, we obtain $D\preceq_MB$.

Next assume that $Az_1\preceq_MB$. Since the unital inclusion $Dz_1 \subset Az_1$ is with expectation by Remark \ref{remark expectation bis}, we may assume without loss of generality that $z_1 = 1_A$. We first prove that $D\preceq_MB$ in the case when $D$ is finite. Since $A$ is semifinite and $D \subset A$ is finite with expectation, the same reasoning as in the proof of $(1) \Rightarrow (2)$ in Theorem \ref{intertwining for type III} shows that any faithful normal semifinite trace on $A$ is still semifinite on the relative commutant $D'\cap A$. In particular, there exists an increasing sequence $(p_n)_n$ of projections in $D'\cap A$ such that $p_n$ converges to $1_A$ $\sigma$-strongly and each projection $p_n$ is finite in $A$. Denote by $z_A(p_n)$ the central support in $A$ of the projection of $p_n\in A$. Since $z_A(p_n)$ converges to $1_A$ $\sigma$-strongly, there exists $n$ such that $Az_A(p_n)\preceq_MB$ (see e.g.\ Remark \ref{remark intertwining}.(2)). Then by Remark \ref{remark intertwining}.(4), we have $p_nAp_n\preceq_MB$. Since $p_nAp_n$ is finite and since $p_n$ commutes with $D$, it follows that $Dp_n \subset p_n A p_n$ is a unital von Neumann subalgebra and we have $Dp_n\preceq_MB$ by Theorem \ref{intertwining for type III} (3). Thus, we obtain $D\preceq_MB$.

We finally prove the general case. Since $D$ is semifinite, there exists a finite projection $p\in D$ such that $z_D(p)=1_D=1_A$. Since $D\subset A$, we have $z_D(p)\leq z_A(p)\leq1_A$ and hence $z_A(p)=1_A$. By Remark \ref{remark intertwining}.(4), we have $pAp\preceq_MB$. Since $pDp \subset pAp$ is finite, we have $pDp\preceq_MB$ by the previous case. Thus, we obtain $D\preceq_MB$.
\end{proof}

We next study the effect of taking relative commutants. Recall from \cite[Lemma 2.1]{Po81} that whenever $A \subset M$ is an inclusion of von Neumann algebras and $p \in A$ is a nonzero projection, we have $(pAp)'\cap pMp=(A'\cap M)p$. We will use this result quite often in the sequel without explicit reference.

\begin{lem}\label{relative commutants}
Let $M$ be any $\sigma$-finite von Neumann algebra, $1_A$ and $1_B$ any nonzero projections in $M$, $A\subset 1_AM1_A$ and $B\subset 1_BM1_B$ any von Neumann subalgebras with expectation. If $A \preceq_M B$, then $B' \cap 1_B M 1_B \preceq_M A' \cap 1_A M 1_A$. 
\end{lem}

\begin{proof}
This proposition follows by the same argument as in \cite[Lemma 3.5]{Va07} (see also \cite[Lemma 2.3.10]{Is14}). However, for the reader's convenience, we give a complete proof below.

By assumption, there exist projections $e \in A$ and $f \in B$, a nonzero partial isometry $v \in eMf$ and a unital normal $\ast$-homomorphism $\theta : eAe \to fBf$ such that the inclusion $\theta(eAe) \subset fBf$ is with expectation and $av = v \theta(a)$ for all $a \in eAe$.

Since the inclusion $A \subset 1_A M 1_A$ is with expectation, so is the inclusion $ A' \cap 1_A M 1_A \subset 1_A M 1_A$. Since $v^*v \in \theta(eAe)' \cap fMf$ and $vv^* \in (eAe)' \cap eMe=(A'\cap 1_AM1_A)e$, we may define 
\begin{align*}
D &= v^*v (\theta(eAe)' \cap fMf) v^*v \\ 
& =(\theta(eAe) v^* v)' \cap v^* v M v^*v \\
&= (v^* A v)' \cap v^* v M v^*v \\
&= (v^* \, A vv^* \,  v)' \cap v^* \, vv^* M vv^* \, v \\
&= v^* ((A vv^*)' \cap vv^* M vv^*) v \\
&= v^* \, vv^*(A'\cap 1_A M 1_A)vv^* \, v \\
&= v^*(A' \cap 1_A M1_A) v.
\end{align*}
Write $1_D := v^*v$. Since the inclusion $\theta(eAe) \subset fMf$ is with expectation, so are the inclusions $\theta(eAe)' \cap fMf \subset fMf$ and $D \subset 1_D M 1_D$. Since $$v D v^* = vv^* (A' \cap 1_A M 1_A) vv^*$$ we have $$v^*v (\theta(eAe)' \cap fMf) v^*v = D \preceq_M vDv^* = vv^*(A' \cap 1_A M 1_A)vv^*,$$
and hence by Remark \ref{remark intertwining}.(2)
$$\theta(eAe)' \cap fMf\preceq_M A' \cap 1_A M 1_A.$$
Since the inclusion $fBf \subset fMf$ is with expectation, so is the inclusion $(B'\cap 1_BM1_B)f = (fBf)' \cap fMf \subset fMf$. Therefore, the unital inclusion $(B'\cap 1_BM1_B)f\subset \theta(eAe)' \cap fMf$ is also with expectation and by Lemma \ref{intertwining subalgebra}, we have 
$$(B'\cap 1_BM1_B)f\preceq_M A' \cap 1_A M 1_A.$$
Thus, we obtain $B'\cap 1_BM1_B\preceq_M A' \cap 1_A M 1_A$.
\end{proof}

We next prove a useful characterization of $A \preceq_M B$ when $A$ is either finite or of type ${\rm III}$.

\begin{lem}\label{lemma equivalence}
Let $M$ be any $\sigma$-finite von Neumann algebra, $1_A$ and $1_B$ any nonzero projections in $M$, $A\subset 1_AM1_A$ and $B\subset 1_BM1_B$ any von Neumann subalgebras with expectation. Assume moreover that $A$ is either finite or of type ${\rm III}$. The following conditions are equivalent:
\begin{enumerate}
\item $A \preceq_M B$ (in the sense of Definition \ref{definition intertwining}).
\item There exist $n \geq 1$, a projection $q \in B \otimes \mathbf M_n$, a nonzero partial isometry $w\in (1_A\otimes e_{1,1})(M\otimes\mathbf{M}_{n})q$ and a unital normal $\ast$-homomorphism $\pi\colon A \rightarrow q(B\otimes\mathbf{M}_n)q$ such that the inclusion $\pi(A) \subset q (B \otimes \mathbf M_n) q$ is with expectation and $(a\otimes 1_n)w = w\pi(a)$ for all $a\in A$, where $(e_{i,j})_{1 \leq i,j \leq n}$ is a fixed matrix unit in $\mathbf{M}_n$. 
\end{enumerate}
\end{lem}

\begin{proof}
The proof is essentially contained in \cite[Proposition 2.3.8]{Is14}. In the case when $A$ is finite, the equivalence between (1) and (2) was already proved in Theorem \ref{intertwining for type III} $(1) \Leftrightarrow (3)$. We may next assume that $A$ is of type ${\rm III}$.

Assume (1) holds. Take $e, f, v, \theta$ witnessing the fact that $A \preceq_M B$ as in Definition \ref{definition intertwining}. Denote by $z_A(e)$ the central support in $A$ of the projection $e \in A$. Since $A$ is of type ${\rm III}$, there exists a partial isometry $u \in A$ such that $u^*u = e$ and $uu^* = z_A(e)$. Put $\iota : A \to A z_A(e) : a \mapsto a z_A(e)$. Therefore $(2)$ holds for $n = 1$, $q = f$, $w = uv$ and $\pi = \theta \circ \Ad(u^*) \circ \iota : A \to fBf$.

Assume (2) holds. Since $\pi(A)$ is of type ${\rm III}$ and since the unital inclusion $\pi(A) \subset q (B \otimes \mathbf M_n) q$ is with expectation, using Remark \ref{remark expectation}, we may assume that $B$ is of type ${\rm III}$. The central support in $B \otimes \mathbf M_n$ of the projection $q \in  B \otimes \mathbf M_n$ is of the form $z \otimes 1$ with $z \in \mathcal Z(B)$. Since the central support in $B \otimes \mathbf M_n$ of the projection $z \otimes e_{1, 1} \in B \otimes \mathbf M_n$ is also $z \otimes 1$, there exists a partial isometry $u \in B \otimes \mathbf M_n$ such that $u^*u = q$ and $uu^* = z \otimes e_{1, 1}$ by \cite[Corollary 6.3.5]{KR97}. Define the partial isometry $v \in 1_A M z$ such that $v \otimes e_{1, 1} = wu^*$ and $\theta : A \to Bz$ the unital normal $\ast$-homomorphism such that $\theta (a) \otimes e_{1, 1} = u \pi(a) u^*$ for all $a \in A$. Then $\theta(A) \subset Bz$ is with expectation and $a v = v \theta(a)$ for all $a \in A$. Therefore $(1)$ holds.
\end{proof}

Let $M,A,B$ be as in Definition \ref{definition intertwining}. Following \cite[Definition 3.1]{Va07}, we write $A\preceq_M^{\rm f}B$ if for every nonzero projection $p\in A'\cap 1_AM1_A$, we have $Ap\preceq_MB$. Recall from \cite[Proposition 2.2]{HU15} that for every nonzero projection $p\in A'\cap 1_AM1_A$, the inclusion $Ap \subset pMp$ is still with expectation.

\begin{lem}\label{lemma good choice}
Let $M$ be any $\sigma$-finite von Neumann algebra, $1_A$ and $1_B$ any nonzero projections in $M$, $A\subset 1_AM1_A$ and $B\subset 1_BM1_B$ any von Neumann subalgebras with expectation. Assume moreover that $A$ is either finite or of type ${\rm III}$ and $A\preceq_M^{\rm f}B$.

Then one can choose $n, q, w, \pi$ as in Condition (2) of Lemma \ref{lemma equivalence} such that the projection $ww^* \in A' \cap 1_A M 1_A$ is arbitrarily close to $1_A$ in the $\sigma$-strong topology. 
\end{lem}

\begin{proof}
By Zorn's lemma, there exists a maximal family $((n_i, q_i, w_i, \pi_i))_{i \in I}$ (with respect to inclusion) of quadruples $(n_i, q_i, w_i, \pi_i)$ witnessing the fact that $A \preceq_M B$ as in Condition (2) of Lemma \ref{lemma equivalence} such that the projections $p_i = w_iw_i^* \in A' \cap 1_A M 1_A$ are pairwise orthogonal for all $i \in I$. We claim that $\sum_{i \in I} p_i = 1_A$. Indeed, if not, put $p = 1_A - \sum_i p_i \in A' \cap 1_A M 1_A$. Since $A\preceq_M^{\rm f}B$, we have $A p \preceq_M B$. Take $(n, q, w, \pi)$ as in Condition (2) of Lemma \ref{lemma equivalence} witnessing the fact that $A p \preceq_M B$. We regard $\pi : A \to q(B \otimes \mathbf M_n)q$ such that the unital inclusion $\pi(A) \subset q(B \otimes \mathbf M_n)q$ is with expectation and $ww^* \leq p$. Then the family $(((n_i, q_i, w_i, \pi_i))_{i \in I}, (n, q, w, \pi))$ contradicts the maximality of the family $((n_i, q_i, w_i, \pi_i))_{i \in I}$. This shows that $\sum_{i \in I} p_i = 1_A$.

Let $\mathcal V \subset A$ be any $\sigma$-strong neighborhood of $1_A$. There exists a finite subset $\mathcal F \subset I$ such that $p_\mathcal F := \sum_{i \in \mathcal F} p_i \in \mathcal V$. Put $n = \sum_{i \in \mathcal F} n_i$, $q = \Diag(q_i)_{i \in \mathcal F} \in B \otimes \mathbf M_n$, $w = [w_{n_i}]_{i \in \mathcal F} \in (1_A \otimes e_{1, 1}) (M \otimes \mathbf M_n)q$ and $\pi : A \to q(B \otimes \mathbf M_n)q : a \mapsto \Diag(\pi_i(a))_{i \in \mathcal F}$. Then we have $(a \otimes 1_n) w = w \pi(a)$ for all $a \in A$ and $ww^* = p_{\mathcal F} \in \mathcal V$. Moreover, the unital inclusion $\pi(A) \subset q(B \otimes \mathbf M_n)q$ is with expectation. Indeed, first observe that the unital inclusions $\Diag(\pi_i(A))_{i \in \mathcal F} \subset \Diag(q_i (B \otimes \mathbf M_{n_i}) q_i)_{i \in \mathcal F}$  and  $ \Diag(q_i (B \otimes \mathbf M_{n_i}) q_i)_{i \in \mathcal F} \subset q(B \otimes \mathbf M_n)q$ are with expectation. Thus, the unital inclusion $\Diag(\pi_i(A))_{i \in \mathcal F} \subset q(B \otimes \mathbf M_n)q$ is with expectation. We next show that $\pi(A) \subset \Diag(\pi_i(A))_{i \in \mathcal F}$ is with expectation. To do so, for every $i \in \mathcal F$, denote by $z_i\in\mathcal{Z}(A)$ the unique central projection that satisfies $\ker \pi_i=A(1_A-z_i)$. We will identify $\pi_i(A)$ with $Az_i$ {\em via} the unital normal $\ast$-isomorphism $\pi_i(A) \to Az_i : \pi_i(a) \mapsto az_i$. We have $\Diag(\pi_i(A))_{i \in \mathcal F}\cong \bigoplus_{i \in \mathcal F} A z_i$. We may and will write $\pi(a)=(az_i)_{i\in\mathcal F} \in \bigoplus_{i \in \mathcal F} A z_i$ for all $a\in A$. Let $\varphi$ be a faithful normal state on $A$ and define $\Phi \in (\bigoplus_{i \in \mathcal F} A z_i)_\ast$ by $\Phi((a_i z_i)_{i \in \mathcal F}) = \sum_{i \in \mathcal F} \lambda_i \varphi(a_i z_i)$ for all $a_i \in A$, where  $\lambda_i > 0$ satisfy $\sum_{i \in \mathcal F} \lambda_i \varphi(z_i) = 1$. By \cite[Lemme 3.2.6]{Co72}, the modular automorphism group of $\Phi$ is given by $\sigma_t^\Phi((a_i z_i)_{i \in \mathcal F})= (\sigma_t^\varphi(a_i) z_i)_{i \in \mathcal F}$ and hence $\pi(A)$ is globally invariant under the modular automorphism group $(\sigma_t^\Phi)$. Thus the unital inclusion $\pi(A) \subset \bigoplus_{i \in \mathcal F} A z_i$ is with expectation. This shows that the unital inclusion $\pi(A) \subset \Diag(\pi_i(A))_{i \in \mathcal F}$ is with expectation. Therefore, we finally obtain that the unital inclusion $\pi(A) \subset q(B \otimes \mathbf M_n)q$ is with expectation.
\end{proof}

The next lemma shows that $\preceq_M^{\rm f}$ provides a sufficient condition for the transitivity of $\preceq_M$.

\begin{lem}\label{intertwining transitivity}
Let $M$ be any $\sigma$-finite von Neumann algebra, $1_A$, $1_B$ and $1_C$ any nonzero projections in $M$, $A\subset 1_AM1_A$, $B\subset 1_BM1_B$ and $C\subset 1_CM1_C$ any von Neumann subalgebras with expectation. 
\begin{enumerate}
	\item Let $p\in A$ (resp.\ $r\in B$) be any projection whose central support in $A$ (resp.\ in $B$) is equal to $1_A$ (resp.\ $1_B$). Then $A\preceq_M^{\rm f}B$ if and only if $pAp\preceq_M^{\rm f}rBr$.
	\item Assume that $A'\cap 1_AM1_A$ is a factor. Then $A\preceq_MB$ if and only if $A\preceq_M^{\rm f}B$.
	\item If $A\preceq_MB$ and $B\preceq_M^{\rm f}C$, then we have $A\preceq_MC$.
\end{enumerate}
\end{lem}

\begin{proof}
$(1)$ By Remark \ref{central support for B}, we may assume that $r=1_B$. Assume that $pAp\preceq_M^{\rm f}B$ and take a nonzero projection $q\in A'\cap 1_AM1_A$. Since $pq\neq0$ and $pq\in (pAp)'\cap pMp$, we have $pApq\preceq_MB$ by assumption. Thus, we obtain that $Aq\preceq_MB$. This shows that $A\preceq_M^{\rm f}B$.

Assume next that $A\preceq_M^{\rm f}B$ and take a nonzero projection $q\in (pAp)'\cap pMp=(A'\cap 1_AM1_A)p$. We may choose a projection $\widetilde{q}\in A'\cap 1_AM1_A$ such that $\widetilde{q}p=q$. By assumption, we have $A\widetilde q\preceq_MB$. Since the central support in $A\widetilde q$ of the projection $q = \widetilde q p \in A\widetilde q$ is equal to $\widetilde q$, we have $qAq\preceq_MB$ by Remark \ref{remark intertwining}.(4). This shows that $pAp \preceq_M^{\rm f}B$.

$(2)$ The `if' direction is trivial and does not require $A'\cap 1_AM1_A$ to be a factor. We next assume that $A'\cap 1_AM1_A$ is a factor and $A\preceq_MB$. Take $e,f, v,\theta$ as in Remark \ref{remark intertwining}.(3). Observe that $vv^* \in (eAe)' \cap eMe$. Let $p \in (A' \cap 1_A M 1_A)e \subset (eAe)'$ be any nonzero projection such that $p \leq vv^*$. Denote by $z \in \mathcal Z(eAe)$ the central support in $(eAe)'$ of the projection $p \in (eAe)'$. We have $eAe z \, p = eAe \,p$ and $eAez \cong eAep$. Define the unital normal $\ast$-isomorphism $\iota : eAep \to eAez$. We may then define the unital normal $\ast$-homomorphism $\widetilde{\theta} : eAep\to\theta(z)B\theta(z) : y \mapsto \theta(\iota(y))$. By assumption and since $z \in eAe$ is a nonzero projection, we have $zv \neq 0$ and $\theta(z) \neq 0$. The unital inclusion $\widetilde \theta(eAep) = \theta(eAe z) \subset \theta(z)B\theta(z)$ is moreover with expectation. We finally have $a \, pv = \iota(a) \, pv = p \, \iota(a) v =  p \, v  \theta(\iota(a)) = pv \, \widetilde \theta(a)$ for all $a \in eAe p$. Since $pv \neq 0$ ({\em n.b.~$p \leq vv^*$}), the above reasoning shows that $eAep \preceq_M B$ for any nonzero projection $p \in (eAe)' \cap eMe$ such that $p \leq vv^*$. 

Let now $q \in A'\cap 1_AM1_A$ be any nonzero projection. Since $A'\cap 1_AM1_A$ is a factor, the projection $qe \in (A' \cap 1_A M 1_A)e$ is nonzero. Since $qe, vv^* \in (A' \cap 1_A M 1_A)e$ and since $(A' \cap 1_A M 1_A)e$ is a factor, there exist nonzero subprojections $q_0\leq qe$ and $p\leq vv^*$ that are equivalent in $(A'\cap 1_AM1_A)e = (eAe)' \cap  eMe$. Since $p\sim q_0$ in $(A'\cap 1_AM1_A)e$ and since $eAep\preceq_MB$ by the first part of the proof, we conclude that $eAeq_0\preceq_MB$ and hence $eAeq\preceq_MB$. This implies that $Aq \preceq_M B$ and finally shows that $A\preceq_M^{\rm f}B$.

$(3)$ Let $A=A_1\oplus A_2$ and $B=B_1\oplus B_2$ be the unique decompositions such that $A_1$ and $B_1$ are semifinite and $A_2$ and $B_2$ are of type III. Let $p\in A_1$ and $q\in B_1$ be finite projections whose central supports are $1_{A_1}$ and $1_{B_1}$. By Remark \ref{remark intertwining}.(2),(4) and Remark \ref{central support for B}, we have either $pA_1p\preceq qB_1q$, or $pA_1p\preceq B_2$ or $A_2\preceq_MB_2$. Also by the first item of Lemma \ref{intertwining transitivity} that we have already proved, we have $qB_1q\preceq_M^{\rm f}C$ and $B_2\preceq_M^{\rm f}C$. Therefore, we may assume that each of the von Neumann algebras $A$ and $B$ is either finite or of type ${\rm III}$.

Then we can proceed as in the proof of \cite[Lemma 3.7]{Va07}. Since $A \preceq_M B$, take $e, f, v, \theta$ witnessing the fact that $A \preceq_M B$ as in Definition \ref{definition intertwining}. Since $B$ is either finite or of type ${\rm III}$ and $B \preceq_M^{\rm f}C$, Lemma \ref{lemma good choice} shows that we can choose $n, q, w, \pi$ witnessing the fact that $B \preceq_M C$ as in Condition (2) of Lemma \ref{lemma equivalence} in such a way that $(v \otimes 1_n)w \neq 0$. Observe that we have $(a \otimes 1_n) \,(v \otimes 1_n)w =  (v \otimes 1_n)w\, \pi(\theta(a))$ for all $a \in eAe$. The unital inclusion $\pi (\theta(eAe) )\subset \pi(f)(C \otimes \mathbf M_n)\pi(f)$ is moreover with expectation. Indeed, first observe that the unital inclusion $\pi(fBf) \subset \pi(f) (C \otimes \mathbf M_n) \pi(f)$ is with expectation. Next, denote by $z \in \mathcal Z(fBf)$ the unique central projection such that $\pi : fBfz \to \pi(fBf)$ is a unital normal $\ast$-isomorphism. Since the unital inclusion $\theta(eAe)z \subset fBfz$ is with expectation by Remark \ref{remark expectation bis}, it follows that the unital inclusion $\pi(\theta(eAe)) = \pi(\theta(eAe)z) \subset \pi(fBfz) = \pi(fBf)$ is with expectation. Therefore, the unital inclusion $\pi( \theta(eAe) )\subset \pi(f)(C \otimes \mathbf M_n)\pi(f)$ is with expectation. Write $(v \otimes 1_n)w = u |(v \otimes 1_n)w|$ for the polar decomposition of $(v \otimes 1_n)w \in M \otimes \mathbf M_n$. Since $A$ is either finite or of type ${\rm III}$, $u \neq 0$ and $(a \otimes 1_n) \,u =  u\, \pi(\theta(a))$ for all $a \in eAe$ and $\pi( \theta(eAe) )\subset \pi(f)(C \otimes \mathbf M_n)\pi(f)$ is with expectation, Lemma \ref{lemma equivalence} finally shows that $eAe \preceq_M C$ and hence $A \preceq_M C$.
\end{proof}

Recall that whenever $P \subset M$ is an inclusion of $\sigma$-finite von Neumann algebras with expectation such that $P$ is a factor and $M = P \vee (P' \cap M)$, we have $M \cong P \ovt (P' \cap M)$. In that case, we will simply write $M = P \ovt (P' \cap M)$. The next intertwining lemma inside tensor product factors will be crucial in the proof of Theorem \ref{thmB}.

\begin{lem}\label{intertwining tensor}
Let $M_1, M_2, N_1, N_2$ be any $\sigma$-finite diffuse factors. Put $M := M_1 \ovt M_2$ and assume that $M = N_1 \ovt N_2$. If $M_1 \preceq_M N_1$, then for every $i\in \{1,2\}$, there exist  projections $p_i\in M_i$, $q_i\in N_i$ and a nonzero partial isometry $v \in M$ with $v^*v =p_1p_2=:p$ and $vv^*=q_1q_2=:q$ such that the inclusion $v M_1 v^* \subset qN_1q$ is with expectation. 

Moreover, $P = (v M_1 v^*)' \cap qN_1q \subset qN_1 q$ is a subfactor with expectation satisfying
$$q N_1 q = vM_1v^* \ovt P \quad \text{and} \quad vM_2v^* =  P \ovt qN_2q.$$
If $M_2$ (resp.\ $N_2$) is a type $\rm III$ factor, then we can take $p_2=1$ (resp.\ $q_2=1$).
\end{lem}

\begin{proof}
We first prove the existence of the nonzero partial isometry $v$ as in the first part of the statement. This will follow from the proofs of \cite[Proposition 12]{OP03} and \cite[Lemma 3.3.2]{Is14}.

Let $e,f,v,\theta$ witnessing the fact that $M_1\preceq_MN_1$ as in Definition \ref{definition intertwining}. Then we have 
$$vv^*\in (eM_1e)'\cap eMe= M_2e \quad \text{and}  \quad v^*v\in \theta(eAe)'\cap fMf=(\theta(eAe)'\cap fN_1f)\ovt N_2f.$$
Put $L:=\theta(eM_1e)'\cap fN_1f$ and observe that $\theta(eM_1e)'\cap fMf=L\ovt N_2f$. Denote by  $z\in \mathcal{Z}(L)$ the unique central projection such that $Lz$ is semifinite and $L(1-z)$ is of type III. Then, up to replacing $\theta$ and $v$ by $\theta(\cdot )z$ and $vz$ (resp.\ by $\theta(\cdot)(1-z)$ and $v(1-z)$) and observing that the unital inclusion $\theta(eM_1e)z \subset zN_1z$ (resp.\ $\theta(eM_1e)(1 - z) \subset (1 - z)N_1(1 - z)$) is with expectation by Remark \ref{remark expectation bis}, we may assume that $L$ is either semifinite or of type ${\rm III}$. Then $L\ovt N_2f$ is also either semifinite or of type ${\rm III}$. In this setting, we first show that for every $i\in \{1,2\}$, there exist a nonzero projection $q_i\in N_i$  such that $q_1q_2\in \theta(eM_1e)'\cap fMf$ and $q_1q_2\precsim v^*v$ in $\theta(eM_1e)'\cap fMf$. 

First assume that $L\ovt N_2f$ is of type ${\rm III}$. Note that this is always the case if $N_2$ is of type ${\rm III}$. If we denote by $z\otimes 1_{N_2}f$ the central support in $L\ovt N_2f$ of the projection $v^*v \in L\ovt N_2f$, with $z\in \mathcal{Z}(L)$, we have that $v^*v\sim z\otimes 1_{N_2}f$ in $L\ovt N_2f$. Thus, we can put $q_1:=z$ and $q_2:=1_{N_2}$.

Next assume that $L\ovt N_2f$ is semifinite. Then $N_2f$ is a type $\rm II_1$ or type $\rm II_\infty$ factor. Fix a faithful normal semifinite trace $\Tr_{N_2f}$ on $N_2f$ and a faithful normal semifinite extended center valued trace $\rT_L$ on $L$ (see e.g.\ \cite[Theorem 2.7]{Ha77a} applied to $\mathcal Z(L) \subset L$ for the proof of the existence of such an extended center valued trace and observe moreover that an extended center valued trace is an operator valued weight). Then $\widetilde{\rT}:=\rT_L\otimes \Tr_{N_2f}$ is a faithful normal semifinite extended center valued trace on $L\ovt N_2f$ (see e.g.\ \cite[Theorem 5.5]{Ha77b} for the proof of the existence of the tensor product of operator valued weights which is in our case an extended center valued trace). Since $\widetilde{\rT}(v^*v)$ is nonzero, there exists a nonzero central projection $z_1\in \mathcal{Z}(L)$ and $\lambda>0$ such that $\lambda z_1\leq \widetilde{\rT}(v^*v)$. Then there exists a nonzero (finite) projection $q_1 \in L z_1$ such that $\rT_L(q_1)\leq n z_1$ for some $n\in \N$. Let $q_2\in N_2$ be a nonzero (finite) projection such that $\Tr_{N_2f}(q_2f)\leq \lambda/n$. Then $q_1q_2 \in L \ovt N_2f$ is finite and we have $\widetilde{\rT}(q_1 q_2)=\Tr_{N_2f}(q_2f)\rT_L(q_1)\leq \Tr_{N_2f}(q_2f)nz_1\leq \lambda z_1\leq \widetilde{\rT}(v^*v)$. This implies that $q_1q_2\precsim v^*v$ in $L\ovt N_2f$.

Let $u\in L\ovt N_2f = \theta(eM_1e)'\cap fMf$ be a partial isometry such that $uu^*\leq v^*v$ and $u^*u=q_1q_2$. Then the nonzero partial isometry $vu$ satisfies 
$$avu = v\theta(a)u = vu\theta(a) \quad \text{for all} \quad a\in eM_1e.$$
Thus, up to replacing the partial isometry $v$ by $vu$, we may assume that $v^*v=q_1q_2$ and $vv^* \in M_2 e$. Since $vv^*\in M_2e$, we may write $vv^*=ep_2$ for some projection $p_2\in M_2$. Put $p_1:=e$. Now, consider the mapping
$$\widetilde{\theta}\colon p_1p_2M_1p_1p_2\ni x\mapsto v^*xv\in v^*vN_1v^*v=q_1q_2N_1q_1 q_2.$$
The map $\widetilde{\theta}$ defines a unital normal $\ast$-isomorphism from $p_1p_2M_1p_1p_2$ into $q_1q_2N_1q_1q_2$. Since the inclusion $p_1p_2M_1p_1p_2\subset p_1p_2Mp_1p_2$ is with expectation, so is the inclusion $$\widetilde{\theta}(p_1p_2M_1p_1p_2) = v^* \,p_1p_2M_1p_1p_2 \,   v \subset v^* \,p_1p_2Mp_1p_2 \,   v = q_1q_2Mq_1 q_2.$$ 
Hence the inclusion $\widetilde{\theta}(p_1p_2M_1p_1p_2) \subset q_1q_2N_1q_1 q_2$ is with expectation. Thus, we obtained the desired partial isometry $v^*$ which satisfies the first part of the statement of Lemma \ref{intertwining tensor}. We note that if $M_2$ is of type ${\rm III}$, we have $p_2\sim 1_{M_2}$ in $M_2$ and hence we may replace $p_1p_2$ by $p_1$. Also, if $N_2$ is of type ${\rm III}$, then by the above proof, we can take $q_2=1_{N_2}$.

Now, put $P = (v M_1 v^*)' \cap qN_1q$ as in the second part of the statement of Lemma \ref{intertwining tensor} and observe that $P \subset qN_1q$ is with expectation. Since $vMv^* = qMq = qN_1q \ovt qN_2q$ and since $vM_1v^* \subset qN_1q$, we obtain 
$$vM_2v^* = (vM_1v^*)' \cap vMv^* = (vM_1v^*)' \cap qMq = ((vM_1v^*)' \cap qN_1q) \ovt qN_2q =  P \ovt qN_2 q$$ 
Likewise, we obtain $q N_1 q = vM_1v^* \ovt P$.
\end{proof}

\section{Proofs of Theorems \ref{thmA} and \ref{thmB}}\label{proofs}

In this section, all the von Neumann algebras that we consider are assumed to have separable predual. We will be using the following notations. For any von Neumann algebra $M$, the standard form of $M$ will be denoted by $(M, \rL^2(M), J, \mathfrak P)$ and we will regard $M \subset \B (\rL^2(M))$. For any $m \geq 1$, any tensor product of von Neumann algebras $M = M_0 \ovt M_1\ovt \cdots \ovt M_i\ovt \cdots \ovt M_m$ and any $1 \leq i \leq m$, put 
$$M_i^c:= M_0 \ovt M_1\ovt \cdots \ovt M_{i-1}\ovt \C1_{M_i}\ovt M_{i+1} \ovt \cdots \ovt M_m.$$
Observe that when $M_i$ is a factor, then we have $M_i^c = (M_i)' \cap M$. When we consider a faithful normal conditional expectation from $M$ onto a tensor component (e.g.\ $M_i$ or $M_i^c$), it will be always assumed to preserve a fixed faithful normal product state $\varphi=\varphi_0\otimes\varphi_1\otimes \cdots\otimes \varphi_n$. Note that as we mentioned in Remark \ref{remark intertwining}.(5), the notion of embedding with expectation $A\preceq_M B$ does not depend on the choice of $\rE_B$. Therefore, the faithful normal states $\varphi_i$ can always be replaced when we consider $A\preceq_MB$ with $B$ a tensor component.

Recall that $R_\infty$ denotes the unique amenable type ${\rm III_1}$ factor with separable predual. It follows from the classification of amenable factors with separable predual \cite{Co72, Co75, Ha85} and \cite[Corollary 6.8]{CT76} that for any amenable factor $P$ with separable predual, we have that $R_\infty \ovt P$ is an amenable type ${\rm III_1}$ factor and hence
$$ R_\infty \ovt P \cong R_\infty.$$ 

\subsection*{A key intermediate result}
	We first prove the following key rigidity result for tensor products of von Neumann algebras that belong to the class $\Cao$. This is a generalization of \cite[Proposition 11]{OP03}. 
	
\begin{thm}\label{location theorem}
	Let $m \geq 1$. For all $1 \leq i \leq m$, let $\mathcal M_i$ be any von Neumann algebra that satisfies the strong condition $\rm (AO)$ and $M_i \subset \mathcal M_i$ any von Neumann subalgebra with expectation $\rE_i : \mathcal M_i \to M_i$. Let $M_0=\mathcal M_0$ be any amenable von Neumann algebra (possibly trivial). 
	
Put $M:=M_0\ovt \cdots \ovt M_m$ and $\mathcal M:=\mathcal M_0\ovt \cdots \ovt \mathcal M_m$. Let $Q\subset M$ be any finite von Neumann subalgebra with expectation $\rE_Q : M \to Q$. 

Then at least one of the following conditions hold:
	\begin{enumerate}
		\item The relative commutant $Q'\cap M$ is amenable. 
		\item There exists $1 \leq i \leq m$ such that $Q \preceq_{M} M_i^c$. 
	\end{enumerate}
\end{thm}
\begin{proof}
We first prove the result under the additional assumption that $M_0$ is of type ${\rm III}$. Since for every $1 \leq i \leq m$, $M_0 \subset M_i^c$ is with expectation, $M_i^c$ is also of type ${\rm III}$. We use the same notation as in Proposition \ref{condition AO for tensor product} for the tensor product von Neumann algebra $\mathcal M$. Write $e_M^{\mathcal M}$ for the Jones projection of the inclusion $M\subset \mathcal M$. We assume that $Q \npreceq_{M} M_i^c$ for all $1 \leq i \leq m$ and show that the relative commutant $Q'\cap M$ is amenable. By Corollary \ref{cor F}, there exists a diffuse abelian (hence AFD) subalgebra $Q_0 \subset Q$ such that $Q_0 \npreceq_M M_i^c$ for all $1 \leq i \leq m$. Since $Q' \cap M \subset Q_0' \cap M$ is with expectation, it suffices to show that $Q_0' \cap M$ is amenable. Without loss of generality, we may assume that $Q$ is AFD.

Since $Q$ is AFD, we can define a {\em proper} conditional expectation $\Psi_Q\colon \mathbf{B}(\rL^2(\mathcal M)) \rightarrow Q'$, that is, $\Psi_{Q}(x)\in Q' \cap \overline{\rm co}^w \left\{uxu^*\mid u\in \mathcal{U}(Q) \right \}$ for all $x\in \B(\rL^2(\mathcal M))$. We note that the properness of $\Psi_Q$ implies that ${(\tau_Q\circ \rE_Q^{\mathcal{M}}\circ \Psi_Q)}|_{\mathcal{M}}=\tau_Q\circ \rE_Q^{\mathcal{M}}$ for any faithful normal trace $\tau_Q \in Q_\ast$ and any faithful normal conditional expectation $\rE_Q^{\mathcal{M}} : \mathcal{M} \to Q$. Hence ${\Psi_Q}|_{ \mathcal{M}} : \mathcal{M} \to Q' \cap \mathcal{M}$ is a faithful {\em normal} conditional expectation. For all $1 \leq i \leq m$, denote by $e_{\mathcal M_i^c}^{\mathcal M} \in \B(\rL^2(\mathcal M))$ the Jones projection of the inclusion $\mathcal M_i^c \subset \mathcal M$ and by $e_{M_i^c}^M \in \B(\rL^2(M))$ the Jones projection of the inclusion $M_i^c \subset M$. We will use the identification $\B (\rL^2(M)) = e_M^{\mathcal M} \B(\rL^2(\mathcal M)) e_M^{\mathcal M}$. Under this identification, we may and will simply denote $e_M^{\mathcal M} e_{\mathcal M_i^c}^{\mathcal M} e_M^{\mathcal M}$ by $e_{M_i^c}^M$. We may and will also identify $e_M^{\mathcal M} \Psi_Q(\cdot) e_M^{\mathcal M}|_M$ with $\Psi_Q|_M$. 

We show that $e_M^{\mathcal M}\Psi_Q(e_{\mathcal M_i^c}^{\mathcal M})e_M^{\mathcal M} = 0$ for all $1 \leq i \leq m$. By contradiction, assume that there exists $1 \leq i \leq m$ such that $e_M^{\mathcal M}\Psi_Q(e_{\mathcal M_i^c}^{\mathcal M})e_M^{\mathcal M} \neq 0$. By the properness of $\Psi_Q$, we have 
$$0 \neq d:= e_M^{\mathcal M}\Psi_Q(e_{\mathcal M_i^c}^{\mathcal M})e_M^{\mathcal M}=\Psi_Q(e_M^{\mathcal M}e_{\mathcal M_i^c}^{\mathcal M}e_M^{\mathcal M}) \in Q' \cap \overline{{\rm co}}^w\left\{ u e_{M_i^c}^{M} u^* \mid u \in \mathcal U(Q) \right \} \subset Q' \cap \langle M, M_i^c\rangle.$$ 
 Let $\rT_{M}$ be the canonical faithful normal semifinite operator valued weight from $\langle M,M_i^c\rangle$ to $M$. Since $d \in \overline{\rm co}^w \left\{ue_{M_i^c}^Mu^*\mid u\in \mathcal{U}(Q) \right\}$ and since $\rT_{M}(e_{M_i^c}^M)=1$, we obtain $\rT_M(d)\in M$ (see the Claim in the proof of Theorem \ref{intertwining for type III} $(5) \Rightarrow (6)$). Recall that $M_i^c$ is of type ${\rm III}$. Then $d$ satisfies Condition $(6)$ in Theorem~\ref{intertwining for type III} ({\em n.b.}~since $B = M_i^c$ is of type ${\rm III}$ here, we have $B_1 = 0$), which contradicts our assumption. Thus, we obtain $e_M^{\mathcal M}\Psi_Q(e_{\mathcal M_i^c}^{\mathcal M})e_M^{\mathcal M} = 0$ for all $1 \leq i \leq m$.

Recall that $\mathcal K_i = \mathbf{B}(\rL^2(\mathcal M_0))\otimes_{\rm min}\cdots\otimes_{\rm min}\mathbf{B}(\rL^2(\mathcal M_{i - 1}))\otimes_{\rm min} \mathbf{K}(\rL^2(\mathcal M_i))\otimes_{\rm min}\mathbf{B}(\rL^2(\mathcal M_{i + 1}))\otimes_{\rm min}\cdots\otimes_{\rm min}\mathbf{B}(\rL^2(\mathcal M_m))$ and hence $\mathcal K_i \subset \mathbf{K}(\rL^2(\mathcal M_i))\otimes_{\rm min} \mathbf{B}(\rL^2(\mathcal M_i^c))$ for all $1 \leq i \leq m$. Recall moreover that $\mathcal J = \sum_{i = 1}^m \mathcal K_i$.

\begin{claim}
For all $1 \leq i \leq m$, we have $e_M^{\mathcal M}\Psi_Q(\mathbf{K}(\rL^2(\mathcal M_i))\otimes_{\rm min} \mathbf{B}(\rL^2(\mathcal M_i^c)))e_M^{\mathcal M}=0$ and hence 
$$e_M^{\mathcal M}\Psi_Q(\mathcal{J})e_M^{\mathcal M}=0.$$
\end{claim}

\begin{proof}[Proof of the Claim]
Fix $1 \leq i \leq m$. Since $e_M^{\mathcal M}\Psi_Q(\cdot)e_M^{\mathcal M}$ is a ucp map, it suffices to show that $e_M^{\mathcal M}\Psi_Q(\mathbf{K}(\rL^2(\mathcal M_i))\otimes_{\rm min} \C 1)e_M^{\mathcal M}=0$. 

Fix a faithful normal state $\varphi \in \mathcal M_\ast$ such that $\sigma_t^\varphi(\mathcal M_i) = \mathcal M_i$ and $\sigma_t^\varphi(M_i) = M_i$ for all $t \in \R$ and all $1 \leq i \leq m$.  Since $\mathbf{K}(\rL^2(\mathcal M_i))$ is generated by the rank one operators $S_{a, b} : \rL^2(\mathcal M_i) \to \C J a \xi_\varphi : \zeta \mapsto \langle \zeta, J b \xi_\varphi \rangle \, J a \xi_\varphi$ with $a, b \in \mathcal M_i$, it suffices to show that $e_M^{\mathcal M}\Psi_Q(S_{a, b} \otimes 1)e_M^{\mathcal M} = 0$ for all $a, b \in \mathcal M_i$. Denote by $\rE_M^{\mathcal M} : \mathcal M \to M$ the unique $\varphi$-preserving faithful normal conditional expectation.

Let $a, b \in \mathcal M_i$. Since 
\begin{align*}
e_M^{\mathcal M}\Psi_Q(S_{a, b}\otimes 1)e_M^{\mathcal M} &= e_M^{\mathcal M}\Psi_Q( e_M^{\mathcal M}(S_{a, b}\otimes 1)e_M^{\mathcal M})e_M^{\mathcal M}  \\
&= e_M^{\mathcal M} \Psi_Q( e_M^{\mathcal M}(S_{\rE_M^{\mathcal M}(a), \rE_M^{\mathcal M}(b)}\otimes 1) e_M^{\mathcal M}) e_M^{\mathcal M} \\
&= e_M^{\mathcal M} \Psi_Q(S_{\rE_M^{\mathcal M}(a), \rE_M^{\mathcal M}(b)}\otimes 1) e_M^{\mathcal M} 
\end{align*}
and since $\rE_M^{\mathcal M}(a) \in M_i$ and $\rE_M^{\mathcal M}(b) \in M_i$, we may assume that $a, b \in M_i$. A simple computation shows that 
$$(J a J \otimes 1) \, (S_{1, 1} \otimes 1) \, (J b^* J \otimes 1) = S_{a, b} \otimes 1.$$
Since $a, b \in M_i \subset M$, we have $JaJ \otimes 1, Jb^*J \otimes 1\in Q' \cap \{e_M^{\mathcal M}\}' \cap \B(\rL^2(\mathcal M))$ and hence 
$$e_M^{\mathcal M}\Psi_Q(S_{a, b}\otimes 1)e_M^{\mathcal M} = (J a J \otimes 1) \, e_M^{\mathcal M}\Psi_Q(S_{1, 1}\otimes 1)e_M^{\mathcal M} \, (J b^* J \otimes 1).$$
Since $S_{1, 1} \otimes 1 = e_{\mathcal M_i^c}^{\mathcal M}$ and $e_M^{\mathcal M}\Psi_Q(e_{\mathcal M_i^c}^{\mathcal M})e_M^{\mathcal M} = 0$, we finally obtain $e_M^{\mathcal M}\Psi_Q(S_{a, b}\otimes 1)e_M^{\mathcal M} = 0$. This finishes the proof of the Claim.
\end{proof}

Using Proposition \ref{condition AO for tensor product} and the above Claim, we may define the composition map
\begin{equation*}
\Phi_Q\colon A\otimes_{\rm min} JAJ \xrightarrow{\nu} \mathcal M(\mathcal{J})/\mathcal{J} \xrightarrow{e_M^{\mathcal M}\Psi_Q(\cdot)e_M^{\mathcal M}} Q' \cap \mathbf B(\rL^2(M)).
\end{equation*}
Observe that $\Phi_Q(a \otimes JbJ) = e_M^{\mathcal M} \Psi_Q(a) \, JbJ e_M^{\mathcal M}$ for all $a, b \in A \subset \mathcal M$. 

\begin{claim}
The ucp map $\Phi_Q$ can be extended to $\mathcal{M} \otimes_{\rm min} J\mathcal{M}J$ which satisfies $\Phi_Q(x \otimes JyJ) = e_M^{\mathcal M}\Psi_Q(x) \, JyJe_M^{\mathcal M}$ for all $x,y \in \mathcal{M}$.
\end{claim}

\begin{proof}[Proof of the Claim]
Consider the algebraic ucp map $\Phi_Q^\circ\colon \mathcal{M}\otimes_{\rm alg} J\mathcal{M}J \to Q' \cap \B(\rL^2(M))$ given by $\Phi^\circ_Q(x \otimes JyJ) = e_M^{\mathcal M}\Psi_Q(x) \, JyJe_M^{\mathcal M}$ for $x,y \in \mathcal{M}$. We know that $\Phi_Q^\circ$ is continuous on $A \otimes_{\rm alg} JAJ$ with respect to the minimal tensor norm since it coincides with $\Phi_Q$. Moreover, $\Phi^\circ_Q$ is normal on $\mathcal{M}\otimes_{\rm alg} \C 1$ and $\C 1 \otimes_{\rm alg} J\mathcal{M}J$ since $\Psi_Q$ is normal on $\mathcal{M}$. By exactness of $A$ (which is equivalent to {\em property C}, see \cite[Theorem 9.3.1 and Exercise 9.3.1]{BO08}), we can apply \cite[Lemma 9.2.9]{BO08} to obtain that $\Phi_Q^\circ$ is actually continuous on $\mathcal{M} \otimes_{\rm alg} J\mathcal{M}J$ with respect to the minimal tensor norm. Thus, by extending $\Phi_Q^\circ$ to $\mathcal{M} \otimes_{\rm min} J\mathcal{M}J$, we obtain the desired ucp map.
\end{proof}

Put $N:=Q'\cap M$ and let $e_N$ be a Jones projection corresponding to $N\subset M$. We restrict the resulting map $\Phi_Q$ in the Claim to the one on $N\otimes_{\rm min} JNJ$. By identifying $JNJ \cong J_NNJ_N$ and $e_Ne_M^{\mathcal M}\B(\rL^2(\mathcal{M}))e_M^{\mathcal M}e_N =\B(\rL^2(N))$, we obtain a ucp map 
$$\pi:=e_N\Phi_Q(\cdot) e_N : N\otimes_{\rm min} J_NNJ_N \to \B(\rL^2(N))$$ 
which satisfies $\pi(x\otimes J_NyJ_N) = x \, J_NyJ_N$ for all $x,y \in N$. Thus, $N=Q'\cap M$ is  amenable. This finishes the proof of Theorem \ref{location theorem} in the case when $M_0$ is of type ${\rm III}$.

In the general case, put $\widetilde M := R_\infty \ovt M$ and $\widetilde M_0 :=R_\infty \ovt M_0$. Then we have that $\widetilde M_0$ is of type ${\rm III}$ and $\widetilde M = \widetilde M_0 \ovt M_1 \ovt \cdots \ovt M_m$. Put $\widetilde M_i^c:= \widetilde M_0 \ovt M_1\ovt \cdots \ovt M_{i-1}\ovt \C1_{M_i}\ovt M_{i+1} \ovt \cdots \ovt M_m$. Let $Q \subset M$ be any finite von Neumann subalgebra with expectation. By the first part of the proof and regarding $Q \subset \widetilde M$, we obtain that $Q' \cap \widetilde M$ is amenable (and hence $Q' \cap M$ is amenable) or there exists $1 \leq i \leq m$ such that $Q \preceq_{\widetilde M} \widetilde M_i^c$. In that case, since $Q \subset M$ is finite, $\widetilde M = R_\infty \ovt M$ and $\widetilde M_i^c = R_\infty \ovt M_i^c$,  Lemma \ref{lemma intertwining in tensor} implies that $Q \preceq_M M_i^c$. This finishes the proof of Theorem \ref{location theorem} in the general case.
\end{proof}

\subsection*{Proofs of Theorems \ref{thmA} and \ref{thmB}}

We start by proving three lemmas.

\begin{lem}\label{intermediate lemma0}
Let $m, n \geq 1$ be any integers. For each $1 \leq i \leq m$, let $M _i$ be any nonamenable factor belonging to the class $\Cao$. For each $1 \leq j \leq n$, let $N_j$ be any nonamenable factor. Assume moreover that all the factors $N_1, \dots, N_{n - 1}$ possess a state with large centralizer. Finally, let $M_0$ and $N_0$ be any amenable factors (possibly trivial). 

Assume that $M:= M_0\ovt M_1\ovt \cdots \ovt M_m = N_0\ovt N_1\ovt \cdots \ovt N_n$. Then there exists $1 \leq i_0 \leq m$ such that $M_{i_0} \preceq_M N_n$. Moreover there exist a nonzero partial isometry $v \in M$, projections $p_i \in M_i$ for every $0 \leq i \leq m$ and projections $q_j \in N_j$ for every $0 \leq j \leq n$ such that $v^*v = p_0 p_1 \cdots p_m := p$, $vv^* = q_0 q_1 \cdots q_n := q$ and $v M_{i_0} v^* \subset qN_nq$ is with expectation. The subfactor $P = (v M_{i_0} v^*)' \cap qN_nq \subset qN_nq$ is also with expectation and satisfies
\begin{align*}
qN_nq &= v M_{i_0} v^* \ovt P \\
P \ovt qN_0q \ovt \cdots \ovt qN_{n- 1}q &= vM_0v^* \ovt \cdots \ovt vM_{i_0 - 1} v^* \ovt vM_{i_0 + 1}v^* \ovt \cdots \ovt vM_mv^*.
\end{align*}
\end{lem}

\begin{proof}
Since $M = M_0\ovt M_1\ovt \cdots \ovt M_m = N_0\ovt N_1\ovt \cdots \ovt N_n$, we also have 
$$\widetilde M := R_\infty \ovt (M_0 \ovt M_1\ovt \cdots \ovt M_m)  = R_\infty \ovt (N_0\ovt N_1\ovt \cdots \ovt N_n).$$ 
By assumption, for every $1 \leq j \leq n - 1$, there exists an irreducible type ${\rm II_1}$ subfactor $L_j \subset N_j$ with expectation. Since $\widetilde N_0 := R_\infty \ovt N_0 \cong R_\infty$, there exists an irreducible type ${\rm II_1}$ subfactor $L_0 \subset \widetilde N_0$ with expectation. Indeed, it is well known that the Araki-Woods factor $R_\infty$ of type ${\rm III_1}$ possesses a state with large centralizer (see e.g.\ \cite[Example 1.6]{Ha85}). Put $\widetilde M_0 = M_0 \ovt R_\infty \cong R_\infty$. We have $\widetilde M = \widetilde M_0 \ovt M_1 \ovt \cdots \ovt  M_m$. 

We first prove that there exists $1 \leq i_0 \leq m$ such that $M_{i_0} \preceq_M N_n$. Because of various considerations on relative commutants, we first need to work inside the larger von Neumann algebra $\widetilde M$ before descending back to the original von Neumann algebra $M$. Put $L_n^c :=L_0\ovt L_1\ovt \cdots \ovt L_{n-1}$. Since $(L_n^c)' \cap \widetilde M = N_n$ is nonamenable, there exists $1 \leq {i_0} \leq m$ such that $L_n^c \preceq_{\widetilde M} (M_{i_0})' \cap \widetilde M$ by Theorem \ref{location theorem}. Applying Lemma \ref{relative commutants}, we obtain $ M_{i_0}\preceq_{\widetilde M} N_n$. Write $M_{i_0} = \widetilde M_{i_0} \ovt I$ where $\widetilde M_{i_0}$ is either of type ${\rm II_1}$ or of type ${\rm III}$ and $I$ is a type ${\rm I}$ factor. By Proposition \ref{cao}, $\widetilde M_{i_0}$ belongs to the class $\Cao$ and hence there exists an irreducible type ${\rm II_1}$ subfactor $K_{i_0} \subset \widetilde M_{i_0}$ with expectation by Theorem \ref{thmC} ({\em n.b.}~$\widetilde M_{i_0}$ is nonamenable). Since $M_{i_0}\preceq_{\widetilde M} N_n$, we have $\widetilde M_{i_0}\preceq_{\widetilde M} N_n$ and hence $K_{i_0} \preceq_{\widetilde M} N_n$ by Lemma \ref{intertwining subalgebra}. Since $K_{i_0} \subset M$ is finite, $\widetilde M = R_\infty \ovt M$ and $N_n \subset M$, Lemma \ref{lemma intertwining in tensor} implies that $K_{i_0}\preceq_{M} N_n$. Since $(K_{i_0})' \cap M = (\widetilde M_{i_0})' \cap M$, applying twice Lemma \ref{relative commutants}, we obtain $\widetilde M_{i_0} \preceq_M N_n$. For any minimal projection $e \in I$, we have $eM_{i_0}e = \widetilde M_{i_0}e$ and $\widetilde M_{i_0}e \preceq_M N_n$ by Lemma \ref{intertwining transitivity}.(2). This implies that $eM_{i_0}e \preceq_M N_n$ and hence $M_{i_0} \preceq_M N_n$ by Remark \ref{remark intertwining}.(2).

By Lemma \ref{intertwining tensor}, there exist a nonzero partial isometry $v \in M$ and projections $p_{i_0} \in M_{i_0}$, $p_{i_0}' \in M_{i_0}^c$, $q_n \in N_n$, $q_n' \in N_n^c$ such that $v^*v = p_{i_0} p_{i_0}' := p$, $vv^* = q_n q_n' := q$ and $v M_{i_0} v^* \subset qN_nq$ is with expectation. The subfactor $P = (v M_{i_0} v^*)' \cap qN_nq \subset qN_nq$ is also with expectation and satisfies
$$qN_n q = v M_{i_0} v^* \ovt P \quad \text{and} \quad v M_{i_0}^c v^* = q N_n^cq \ovt P.$$
By factoriality and since $M_1, \dots, M_{i_0 - 1}, M_{i_0 +1}, \dots, M_m$ are diffuse factors, $p_{i_0}' \in M_{i_0}^c$ is equivalent in $M_{i_0}^c$ to a projection of the form $p_0 p_1 \cdots p_{i_0 - 1} p_{i_0 + 1} \cdots p_m$ where $p_i \in M_i$ is a nonzero projection for all $i \neq i_0$. Likewise, by factoriality and since $N_1, \dots, N_{n - 1}$ are diffuse factors, $q_n' \in N_n^c$ is equivalent in $N_n^c$ to a projection of the form $q_0 \cdots q_{n - 1}$ where $q_j \in N_j$ is a nonzero projection for all $0 \leq j \leq n - 1$. Thus, we may assume that $p_{i_0}' = p_0 p_1 \cdots p_{i_0 - 1} p_{i_0 + 1} \cdots p_m$, $q_n' = q_0 \cdots q_{n - 1}$ and hence $v^*v = p_0 \cdots p_m$, $vv^* = q_0 \cdots q_n$. Then the von Neumann subalgebras $v M_i v^* \subset qMq$ pairwise commute for all $0 \leq i \leq m$ as well as the von Neumann subalgebras $q N_j q \subset qMq$ for all $0 \leq j \leq n$. Thus, we have
$$qMq = vM_0v^* \ovt vM_1v^* \ovt \cdots \ovt vM_mv^* = qN_0q \ovt qN_1q \ovt \cdots \ovt qN_nq.$$
Finally, we obtain
\begin{align*}
qN_nq &= v M_{i_0} v^* \ovt P \\
P \ovt qN_0q \ovt \cdots \ovt qN_{n- 1}q &= vM_0v^* \ovt \cdots \ovt vM_{i_0 - 1} v^* \ovt vM_{i_0 + 1}v^* \ovt \cdots \ovt vM_mv^*.
\end{align*}
This finishes the proof of Lemma \ref{intermediate lemma0}.
\end{proof}

\begin{lem}\label{intermediate lemma1}
Keep the same assumption as in Lemma \ref{intermediate lemma0}. 

Assume that $M:= M_0\ovt M_1\ovt \cdots \ovt M_m = N_0\ovt N_1\ovt \cdots \ovt N_n$. Then we have $m\geq n$. 
\end{lem}

\begin{proof}
We prove the result by induction over $n \geq 1$. When $n = 1$, this is obvious since by assumption we have $m \geq 1 = n$. 

Next, assume that $n \geq 2$. By Lemma \ref{intermediate lemma0}, there exists $1 \leq i_0 \leq m$ such that $M_{i_0} \preceq_M N_n$. To simplify the notation, we may assume that $i_0 = m$. By Lemma \ref{relative commutants}, we also have $N_n^c \preceq_M M_m^c$. If $m=1$, then $N_n^c \preceq_M M_0$ yields an embedding with expectation of a nonamenable factor into an amenable factor, a contradiction. Hence, we may assume that $m\geq 2$.

Since $M_m \preceq_M N_n$, choose $v, p_0, \dots, p_m, q_0, \dots, q_n$ and define $P = (v M_m v^*)' \cap qN_nq$ as in the conclusion of Lemma \ref{intermediate lemma0}. We have
$$vM_0 v^* \ovt vM_{1}v^* \ovt \cdots \ovt v M_{m - 1} v^* = qN_0q \ovt \cdots \ovt qN_{n - 2}q \ovt (qN_{n- 1}q \ovt P).$$
Observe that for every $1 \leq i \leq m - 1$, we have $vM_i v^* \cong p_i M_i p_i$ and hence $vM_iv^*$ belongs to the class $\Cao$ by Proposition \ref{cao}. Since $qN_0q$ is amenable, since $qN_{n- 1}q \ovt P$ is nonamenable and since the factor $q N_jq \cong q_j N_j q_j$ has a state with large centralizer for every $1 \leq j \leq n - 1$, we may apply the induction hypothesis and we obtain $m - 1 \geq n - 1$. Thus, $m \geq n$. This finishes the proof of the induction and hence the one of Lemma~\ref{intermediate lemma1}.
\end{proof}

\begin{lem}\label{intermediate lemma2}
Let $m \geq 1$ be any integer. For all $1 \leq i \leq m$, let $M _i$ be any nonamenable factor belonging to the class $\Cao$. Then the tensor product factor $M_1 \ovt \cdots \ovt M_m$ is not McDuff.
\end{lem}

\begin{proof}
Assume first that $M = M_1 \ovt \cdots \ovt M_m$ is semifinite. Then $M$ is full and hence not McDuff by \cite[Corollary 2.3]{Co75}. 

Assume next that $M = M_1 \ovt \cdots \ovt M_m$ is of type ${\rm III}$. By Proposition \ref{cao}, we may further assume that each of the factors $M_1, \dots, M_m$ is either of type ${\rm II_1}$ or of type ${\rm III}$ and belongs to the class $\Cao$. By contradiction, assume that $M$ is McDuff. Then there exist $m \geq 1$ such that the following property $\mathcal P_m$ holds: there exist nonamenable factors $M_1, \dots, M_m$ belonging to the class $\Cao$ such that each of the factors $M_1, \dots, M_m$ is either of type ${\rm II_1}$ or of type ${\rm III}$  and $M = M_1 \ovt \cdots \ovt M_m$ is McDuff, that is, 
$$M = M_1 \ovt \cdots \ovt M_m \cong M_1 \ovt \cdots \ovt (M_m \ovt R).$$ 
We may further assume that $m \geq 1$ is the minimum integer for which the property $\mathcal P_m$ holds. By \cite[Theorem A]{HR14}, we necessarily have $m \geq 2$.

By Theorem \ref{thmC}, each of the factors $M_1, \dots, M_m$ possesses a state with large centralizer. Applying Lemma \ref{intermediate lemma0} to $m = n$, $M_0 = N_0 = \C1$, $N_1 = M_1, \dots, N_{m - 1} = M_{m - 1}$ and $N_m = M_m \ovt R$, there exists $1 \leq i_0 \leq m$ such that $M_{i_0} \preceq_M N_m$. To simplify the notation, we may assume that $i_0 = m$. Since $M_m \preceq_M N_m$, choose $v, p_0, \dots, p_m, q_0, \dots, q_m$ and define $P = (v M_m v^*)' \cap qN_m q$ as in the conclusion of Lemma \ref{intermediate lemma0}. We have
$$qN_mq = vM_mv^* \ovt P \quad \text{and} \quad vM_{1}v^* \ovt \cdots \ovt v M_{m - 1} v^* = qN_1q \ovt \cdots \ovt qN_{m - 1}q \ovt P.$$ 

By Lemma \ref{intermediate lemma1} and since $m \geq 1$ is the minimum integer for which the property $\mathcal P_m$ holds, the second equation implies that $P$ is a type ${\rm I}$ factor ({\em n.b.}~Any diffuse amenable factor is McDuff). Since moreover $q N_m q = v M_m v^* \ovt P$, it follows that $M_m$ and $N_m = M_m \ovt R$ are stably isomorphic and hence $M_m$ is McDuff. This however contradicts \cite[Theorem A]{HR14}.
\end{proof}

\begin{nota}\label{notation equivalence}
Let $M$ be any $\sigma$-finite von Neumann algebra, $1_A$ and $1_B$ any nonzero projections in $M$, $A \subset 1_A M 1_A$ and $B \subset 1_B M 1_B$ any von Neumann subalgebras with expectation. We will write $A\sim_M B$ if there exist projections $p\in A$, $p'\in A'\cap 1_AM1_A$, $q\in B$, $q'\in B'\cap 1_BM1_B$ and a nonzero partial isometry $v\in 1_BM1_A$ such that $v^*v=pp'$, $vv^*=qq'$ and $vpApp'v^*=qBqq'$. 
\end{nota}

Keep $M, A, B$ as in Notation \ref{notation equivalence}. Obviously, $A\sim_MB$ implies that $A\preceq_MB$ and $B\preceq_MA$. Moreover, it is easy to see that when $A$ and $A' \cap 1_A M1_A$ are both factors, we have $A \sim_M B$ if and only if $rArr' \sim_M B$ for some (or any) nonzero projections $r \in A$ and $r' \in A' \cap 1_A M 1_A$. Finally, when $A, B \subset M$ are {\em unital} von Neumann subalgebras and $A, B, A' \cap M, B' \cap M$ are type ${\rm III}$ factors, we have $A\sim_MB$ if and only if $A$ and $B$ are unitarily conjugate inside $M$, that is, there exists $u \in \mathcal U(M)$ such that $u A u^* = B$.

Theorems \ref{thmA} and \ref{thmB} will be consequences of the following theorem that generalizes \cite[Corollary 3]{OP03}. 

\begin{thm}\label{main theorem}
Let $n \geq 1$ be any integer. For all $1 \leq i \leq n$, let $M _i$ be any nonamenable factor belonging to the class $\Cao$. For all $1 \leq j \leq n$, let $N_j$ be any non-McDuff factor that possesses a state with large centralizer. Finally, let $M_0$ and $N_0$ be any amenable factors (possibly trivial) with separable predual.

Assume that  
$$M:= M_0\ovt M_1\ovt \cdots \ovt M_n = N_0\ovt N_1\ovt \cdots \ovt N_n.$$
Then there exists a unique permutation $\sigma \in \mathfrak S_n$ such that
$$N_0 \sim_M M_0 \quad \text{and} \quad N_j \sim_M M_{\sigma(j)} \quad \text{for all} \quad 1 \leq j \leq n.$$
In particular, $M_0$ and $N_0$ are stably isomorphic and $M_{\sigma(j)}$ and $N_j$ are stably isomorphic for all $1 \leq j \leq n$.
\end{thm}

\begin{proof}
We first prove the existence of a permutation $\sigma \in \mathfrak S_n$ satisfying the property as in the statement of Theorem \ref{main theorem} by induction over $n \geq 1$.

Assume that $n=1$. By Lemma \ref{intermediate lemma0}, we have $M_1\preceq_M N_1$. Choose $v, p_0, p_1, q_0, q_1$ and define $P = (v M_1 v^*)' \cap qN_1q$ as in the conclusion of Lemma \ref{intermediate lemma0}. We have
$$qN_1q = vM_1v^* \ovt P \quad \text{and} \quad vM_{0}v^* = qN_0q \ovt P.$$
Note that $P$ is amenable since $M_0$ is amenable and $P \subset vM_0 v^*$ is with expectation. Since $N_1$ is not McDuff, $qN_1 q$ is not McDuff either and hence $P$ is a type ${\rm I}$ factor. Therefore, we obtain $N_0 \sim_M M_0$ and $N_1 \sim M_1$. This proves the case when $n=1$ and hence the first step of the induction.

Next, assume that $n \geq 2$. By Lemma \ref{intermediate lemma0}, there exists $1 \leq i_0 \leq n$ such that $M_{i_0} \preceq_M N_n$. To simplify the notation, we may assume that $i_0 = n$. Since $M_n \preceq_M N_n$, choose $v, p_0, \dots, p_n, q_0, \dots, q_n$ and define $P = (v M_n v^*)' \cap qN_nq$ as in the conclusion of Lemma \ref{intermediate lemma0}. We have
\begin{align*}
qN_nq &= v M_nv^* \ovt P \\
vM_{0}v^* \ovt vM_1 v^* \ovt \cdots \ovt v M_{n - 1} v^* &= qN_0q \ovt \cdots \ovt qN_{n - 2}q \ovt qN_{n- 1}q \ovt P.
\end{align*}
Observe that using Lemma \ref{intermediate lemma1}, the second equation implies that $P$ is amenable. Since $N_n$ is not McDuff, $qN_nq$ is not McDuff either and hence the first equation implies that $P$ is a type ${\rm I}$ factor. This implies in particular that $N_n \sim_M M_n$.

We may apply the induction hypothesis to 
$$vM_n^cv^* = vM_0 v^* \ovt vM_{1}v^* \ovt \cdots \ovt v M_{n - 1} v^* = (P \ovt qN_0q) \ovt \cdots \ovt qN_{n - 2}q \ovt qN_{n- 1}q$$ 
and we obtain a permutation $\sigma \in \mathfrak S_{n - 1}$ such that $P \ovt qN_0q \sim_{vM_n^cv^*} vM_0v^*$ and $q N_jq  \sim_{vM_n^cv^*} vM_{\sigma(j)}v^*$ for all $1 \leq j \leq n - 1$. This implies that $N_0 \sim_{M} M_0$ ({\em n.b.}~$P$ is a type ${\rm I}$ factor) and $N_j  \sim_{M} M_{\sigma(j)}$ for all $1 \leq j \leq n - 1$. Letting $\sigma (n) = n$ and regarding $\sigma \in \mathfrak S_n$, we also have $N_n \sim_M M_{\sigma(n)}$. This finishes the proof of the induction and hence the one of the existence of a permutation $\sigma \in \mathfrak S_n$ satisfying the property as in the statement of Theorem \ref{main theorem}.

We finally prove the uniqueness of the permutation $\sigma \in \mathfrak S_n$ satisfying the property  as in the statement of Theorem \ref{main theorem}. Assume that there exists another permutation $\tau\in\mathfrak{S}_n$ satisfying the property  as in the statement of Theorem \ref{main theorem}.
Observe that for all $1\leq j\leq n$, we have $M_{\sigma(j)}\preceq_M N_j$ and $N_j\preceq_M M_{\tau(j)}$. By Lemma \ref{intertwining transitivity}.(2), we have $N_j\preceq_M^{\rm f}M_{\tau(j)}$ and hence by Lemma \ref{intertwining transitivity}.(3), we have $M_{\sigma(j)}\preceq_M M_{\tau(j)}$ for all $1\leq j\leq n$. 

Now suppose by contradiction that $\sigma\neq \tau$ and fix $1 \leq j \leq n$ such that $\sigma(j)\neq \tau(j)$. Let $K\subset M_{\sigma(j)}$ be a diffuse abelian von Neumann subalgebra with expectation. By Lemma \ref{intertwining subalgebra}, we have $K\preceq_M M_{\tau(j)}$ and hence $K\preceq_M M_{\sigma(j)}^c$ since $\sigma(j)\neq \tau(j)$. Write $k := \sigma(j)$ for simplicity. Since $K$ is diffuse, there exists a net $(u_i)_{i \in I}$ of unitaries in $\mathcal U(K)$ that converges to zero $\sigma$-weakly as $i \to \infty$. Fix a product faithful normal state $\psi:=\psi_0\otimes \psi_1\otimes\cdots\otimes \psi_n \in M_\ast$. Denote by $\rE_{M_k} : M \to M_k$ and $\rE_{M_k^c} : M \to M_k^c$ the corresponding unique $\psi$-preserving conditional expectations. Then for all $a,b\in M_k$ and all $c,d\in M_k^c$, we have
\begin{align*}
\limsup_i \|\rE_{M_k^c}((b\otimes d)^* u_i (a\otimes c) ) \|_{\psi} &= \limsup_i\|d^* \, \rE_{M_k^c}(b^* u_i a) \, c \|_{\psi} \\
&= \limsup_i \|d^* c \|_{\psi} \, |\psi_k(b^* u_i a)| \\
& = 0.
\end{align*}
Hence by Theorem \ref{intertwining for type III}.(5), we obtain $K\not\preceq_MM_k^c$, a contradiction. This shows the uniqueness of the permutation $\sigma \in \mathfrak S_n$ satisfying the property as in the statement of Theorem \ref{main theorem}. This finishes the proof of Theorem \ref{main theorem}.
\end{proof}

\begin{proof}[Proof of Theorem \ref{thmA}]
We only need to show that if $M = M_0\ovt M_1 \ovt \cdots \ovt M_m$ and $N = N_0\ovt N_1 \ovt \cdots \ovt N_n$ are stably isomorphic then $m = n$, $M_0$ and $N_0$ are stably isomorphic, and there exists a permutation $\sigma \in \mathfrak S_n$ such that $M_{\sigma(j)}$ and $N_j$ are stably isomorphic for all $1 \leq j \leq n$. 

By assumption, we have $M \ovt \B(\ell^2) \cong N \ovt \B(\ell^2)$. Using Proposition \ref{cao} and up to replacing $M_1$ by $M_1 \ovt \B(\ell^2)$ and $N_1$ by $N_1 \ovt \B(\ell^2)$, we may further assume that $M = N$. 

First, assume that $M$ is semifinite. In this case, we may assume that $M_0, N_0$ are semifinite and $M_1, \dots, M_m, N_1, \dots, N_n$ are ${\rm II_1}$ factors belonging to the class $\Cao$ (see Proposition \ref{cao}). Then Theorem \ref{thmA} follows by applying twice Lemma \ref{intermediate lemma1} and Theorem \ref{main theorem}. 

Next, assume that $M$ is of type ${\rm III}$. In that case, we may assume that each of the factors $M_1, \dots, M_m, N_1, \dots, N_n$ is either of type ${\rm II_1}$ or of type ${\rm III}$ and belongs to the class $\Cao$ (see Proposition \ref{cao}). Then Theorem \ref{thmC} implies that each of the factors $M_1, \dots, M_m, N_1, \dots, N_n$ possesses a state with large centralizer. Then Theorem \ref{thmA} follows again by applying twice Lemma \ref{intermediate lemma1} and Theorem \ref{main theorem}.
\end{proof}

\begin{proof}[Proof of Theorem \ref{thmB}]
Observe that the factors $N_1, \dots, N_n$ are not McDuff and hence nonamenable by Lemma \ref{intermediate lemma2}. Then we can apply Lemma \ref{intermediate lemma1} to obtain $(1)$ and Theorem \ref{main theorem} to obtain $(2)$.

If $M_1, \dots, M_n$ are moreover type ${\rm III}$ factors and $n \geq 2$, then $M_i, M_i^c, N_j, N_j^c$ are type ${\rm III}$ factors for all $1 \leq i, j \leq n$. Therefore, $M_{\sigma(j)}$ and $N_j$ are unitarily conjugate inside $M$ for all $1 \leq j \leq n$.
\end{proof}

\appendix

\section{Relative modular theory and Jones basic construction}\label{section-relative}

	Let $N \subset M$ be any inclusion of $\sigma$-finite von Neumann algebras with faithful normal conditional expectation $\rE_N : M \to N$. Let $\varphi_N$ be any faithful normal semifinite weight on $N$ and put $\varphi:=\varphi_N\circ \rE_N$. We have a canonical inclusion $\rL^2(N,\varphi_N) \subset \rL^2(M,\varphi)$. Recall that the {\em Jones projection} $e_N^{\varphi_N} \colon \rL^2(M,\varphi) \to \rL^2(N,\varphi_N)$ is defined by $e_N^{\varphi_N}\Lambda_{\varphi}(x) := \Lambda_{\varphi_N}(\rE_N(x))$ for $x\in \mathfrak{n}_{\varphi}$. In this appendix, we observe that the Jones projection $e_N^{\varphi_N} : \rL^2(M, \varphi) \to \rL^2(N, \varphi_N)$ does not depend on the choice of the faithful normal semifinite weight $\varphi_N$ on $N$. We refer to \cite[Lemma A]{Ko88} for a similar statement involving states instead of weights.

	The following lemma is well known (see e.g.\ the proof of \cite[Theorem IX.4.2]{Ta03}).
\begin{lem}\label{takesaki lemma}
	The following relations hold true: 
	$$\Delta_{\varphi}e_N^{\varphi_N} = e_N^{\varphi_N}\Delta_\varphi = \Delta_{\varphi_N} \quad \text{and} \quad J_{\varphi}e_N^{\varphi_N} = e_N^{\varphi_N}J_\varphi = J_{\varphi_N}.$$
\end{lem}

	Fix another faithful normal semifinite weight $\psi_N$ on $N$ and put $\psi:=\psi_N\circ \rE_N$. Since the standard representation is unique thanks to \cite[Theorem 2.3]{Ha73}, there is a unique unitary mapping $U_{\psi,\varphi}\colon \rL^2(M,\varphi)\to \rL^2(M,\psi)$ that preserves the standard representation structure. Recall that $U_{\psi, \varphi}$ is uniquely determined by the following two conditions (see \cite[Theorem 2.3 and Remark 2.11]{Ha73} or \cite[Lemma IX.1.5]{Ta03}):
\begin{equation}\label{eq-uniqueness}
U_{\psi,\varphi}\mathfrak{P}_{\varphi} = \mathfrak{P}_{\psi} \quad \text{and} \quad U_{\psi,\varphi}\pi_\varphi(x) = \pi_\psi(x) U_{\psi,\varphi} \quad (x\in M),
\end{equation}
where $\pi_\varphi$ and $\pi_\psi$ are the corresponding GNS representations. Using the above uniqueness result, we have the following two possible constructions for $U_{\psi,\varphi}$. 

	The first construction deals with the case when $\varphi$ and $\psi$ are faithful normal states on $M$ (see \cite[Lemma IX.1.8]{Ta03} for this construction). One can define the unitary mapping $V_{\psi, \varphi} : \rL^2(M, \varphi) \to \rL^2(M, \psi)$ by $V_{\psi,\varphi}\pi_\varphi(x)\Lambda_\varphi(1) = \pi_\psi(x) \xi_\varphi$ for $x\in M$, where $\xi_\varphi \in \mathfrak{P}_\psi$ is the unique vector implementing $\varphi \in M_\ast$. Then $V_{\psi, \varphi}$ is a well defined unitary mapping which satisfies the uniqueness conditions \eqref{eq-uniqueness}. Therefore, we have $U_{\psi, \varphi} = V_{\psi, \varphi}$.

	The second construction deals with the more general case when $\varphi$ and $\psi$ are faithful normal semifinite weights on $M$ (see the proof of \cite[Lemma IX.1.5]{Ta03} for this construction). Consider the balanced weight $\theta:=\psi\oplus \varphi$ on $M\otimes \mathbf M_2$ and recall that, using the notation in \cite[VIII.\S3]{Ta03}, its corresponding $J$-map has the following form:
$$
J_\theta=\left[ 
\begin{array}{cccc}
J_{\psi} & 0 & 0 &  0 \\
0 & 0 & J_{\psi, \varphi} &  0 \\
0 & J_{\psi,\varphi}^{-1} & 0 &  0 \\
0 & 0 & 0 &  J_{\varphi} \\
\end{array} 
\right].
$$
The unitary mappings $J_\psi J_{\psi, \varphi}$ and $J_{\psi,\varphi}J_\varphi$ satisfy the uniqueness conditions \eqref{eq-uniqueness}. Therefore, we have $U_{\psi, \varphi} = J_\psi J_{\psi, \varphi} = J_{\psi,\varphi}J_\varphi$.

	Using this second construction, if we put $N_2:=N\otimes \mathbf M_2 \subset M \otimes \mathbf M_2$ with expectation $\rE_{N_2}:=\rE_N \otimes \id_{\mathbf M_2} : M \otimes \mathbf M_2 \to N \otimes \mathbf M_2$, then we have $\theta\circ \rE_{N_2}=\theta$. So by Lemma \ref{takesaki lemma}, we have $e_{N_2}^{\theta|_{N_2}} J_\theta=J_\theta e_{N_2}^{\theta|_{N_2}}$. Observe that $e_{N_2}^{\theta|_{N_2}}$ is of the form 
$$
e_{N_2}^{\theta|_{N_2}}=\left[ 
\begin{array}{cccc}
e_N^{\psi_N} & 0 & 0 &  0 \\
0 & e_N^{\psi_N} & 0 &  0 \\
0 & 0 & e_N^{\varphi_N} &  0 \\
0 & 0 & 0 & e_N^{\varphi_N} \\
\end{array} 
\right].
$$
Thus, we obtain $e_N^{\psi_N}J_{\psi,\varphi}=J_{\psi,\varphi}e_N^{\varphi_N} $ and hence $e_N^{\psi_N} U_{\psi,\varphi}=U_{\psi,\varphi}e_N^{\varphi_N}$. By construction and the condition $J_\theta e_N^{\theta|_{N_2}}=J_{\theta|_{N_2}}$, we conclude that $U_{\psi_N,\varphi_N} = e_N^{\psi_N} U_{\psi,\varphi}=U_{\psi,\varphi}e_N^{\varphi_N}$. We summarize this observation as follows.

\begin{prop}\label{prop-uniqueness-projection}
	Let $\varphi_N$ and $\psi_N$ be any faithful normal semifinite weights on $N$ and put $\varphi:=\varphi_N\circ \rE_N$ and $\psi:=\psi_N\circ \rE_N$. Then we have $U_{\psi_N,\varphi_N} = e_N^{\psi_N} U_{\psi,\varphi}=U_{\psi,\varphi}e_N^{\varphi_N}$. Thus, we have
	$$e_N^{\psi_N} = U_{\psi,\varphi} \, e_N^{\varphi_N} \, (U_{\psi, \varphi})^*.$$
\end{prop}

	The above proposition exactly means that the Jones projection $e_N^{\varphi_N} : \rL^2(M, \varphi) \to \rL^2(N, \varphi_N)$ does not depend on the choice of the faithful normal semifinite weight $\varphi_N$ of $N$. Therefore, we simply denote it by $e_N : \rL^2(M) \to \rL^2(N)$.

Combined with the first construction of $U_{\psi, \varphi}$ above for faithful normal states $\varphi, \psi \in M_\ast$, we recover \cite[Lemma A]{Ko88}. Indeed, recall that $U_{\psi,\varphi} : \rL^2(M, \varphi) \to \rL^2(M, \psi)$ is given by $U_{\psi, \varphi} \pi_\varphi(x) \Lambda_\varphi(1) = \pi_\psi(x) \xi_\varphi$ for $x \in M$, where $\xi_\varphi \in \mathfrak{P}_\psi$ is the unique vector implementing $\varphi \in M_\ast$. Thus, for all $x\in M$, we have
	$$e_N^\psi\pi_\psi(x) \xi_\varphi= U_{\psi , \varphi}  e_N^\varphi\pi_\varphi(x) \Lambda_\varphi(1) =U_{\psi, \varphi}  \pi_\varphi(\rE_N(x)) \Lambda_\varphi(1) =\pi_\psi(\rE_N(x)) \xi_\varphi.$$
This is exactly \cite[Lemma A]{Ko88}.

\section{Dimension theory for semifinite von Neumann algebras}\label{section-dimension}

\subsection*{Dimension theory for right modules over semifinite von Neumann algebras}

Let $B$ be any semifinite von Neumann algebra with a faithful normal semifinite tracial weight $\Tr_B$. Let $\mathcal K=\mathcal K_B$ be a right $B$-module (i.e.\ there is a unital normal $\ast$-homomorphism $B^{\rm op} \to \B(\mathcal K)$). We note that injectivity of the $\ast$-homomorphism is {\em not} assumed in this subsection. Then by \cite[Theorem IV.5.5]{Ta02}, there is a Hilbert space $\ell^2(I)$ and a projection $p\in B\ovt \B(\ell^2(I))$ such that $\mathcal K \cong p \left(\rL^2(B,\Tr_B)\otimes \ell^2(I)\right)$ as right $B$-modules. We define the {\em right dimension of $\mathcal K$ with respect to $(B,\Tr_B)$} by $\dim_{(B,\Tr_B)} \mathcal K := \Tr(p)$,  where $\Tr:=\Tr_B\otimes \Tr_{\mathbf B(\ell^2(I))}$ for the unique trace $\Tr_{\mathbf B(\ell^2(I))}$ on $\B(\ell^2(I))$ satisfying $\Tr_{\mathbf B(\ell^2(I))}(e)=1$ for all minimal projections $e$. This definition depends neither on the choice of the Hilbert space $\ell^2(I)$ nor on the choice of the projection $p \in B\ovt \B(\ell^2(I))$. Indeed, if $V,W : \mathcal K \to \rL^2(B,\Tr_B)\otimes \ell^2(I)$ are right $B$-modular isometries, since $WV^* \in (J_BBJ_B\otimes \C)' = B \ovt \B(\ell^2(I))$, we have
$$\Tr(VV^*)=\Tr(VW^*WV^*)=\Tr(WV^*VW^*)=\Tr(WW^*).$$

For any Hilbert space $\ell^2(I)$, any projections $p\in B\ovt\B(\ell^2(I))$ and $q\in B$ with central support $z:=z_B(q) \in \mathcal Z(B)$, regarding $J_BqJ_B\rL^2(B,\Tr_B)=\rL^2(Bq,\Tr_B)$, we have that 
\begin{align}\label{eq-dimension}
\dim_{(qBq,\Tr_B(q\, \cdot \, q))}p \left(\rL^2(Bq,\Tr_B)\otimes \ell^2(I) \right) &=\dim_{(Bz,\Tr_B(\, \cdot \, z))}p \left( \rL^2(Bz,\Tr_B)\otimes \ell^2(I) \right) \\ \nonumber
&=\Tr(p(z\otimes 1)).
\end{align}
To see this, put $p_0 := p(z\otimes 1)$ and take projections $\{p_j \}_j \subset B\ovt \B(\ell^2(I))$ such that $p_0=\sum_j p_j$ and $p_j\sim q_j \otimes e$ for some projections $q_j\in qBq$ and a fixed minimal projection $e\in \B(\ell^2(I))$ (use \cite[Lemma V.1.7]{Ta02} and Zorn's lemma). Then we have the following right $qBq$-module isomorphisms 
	$$p_0 \left(\rL^{2}(Bq)\otimes \ell^2(I) \right) = \sum_j p_j \left(\rL^{2}(Bq)\otimes \ell^2(I) \right) \cong  \bigoplus_j (q_j\otimes e) \left(\rL^{2}(Bq)\otimes \ell^2(I) \right) \cong \bigoplus_j q_j\rL^{2}(Bq),$$ 
where $qBq$ acts diagonally by right multiplication on the right hand side. Since $\Tr_B(q_j)=\Tr(q_j\otimes e)=\Tr(p_j)$, we obtain 
$$\dim_{(qBq,\Tr_B(q\, \cdot \, q))}p_0 \left(\rL^{2}(Bq)\otimes \ell^2(I) \right) = \sum_j \Tr_B(q_j)= \sum_j \Tr_B(p_j) = \Tr(p_0).$$
Finally considering the case $q=z$, we obtain the desired equation \eqref{eq-dimension}.

\subsection*{Dimension theory and Jones basic construction}

Let $B\subset M$ be any inclusion of von Neumann algebras with expectation $\rE_B$ and $\varphi$ any faithful normal semifinite weight on $B$. Extend $\rE_B : M \to B$ to $\rE_B : \langle M,B\rangle \to B$ by the formula $\rE_B(x)e_B=e_Bxe_B$ for all $x\in \langle M,B\rangle$. We define the corresponding faithful normal semifinite weight $\widehat \varphi$ on $\langle M,B\rangle$ as follows. 
Let $q\in B$ be any projection such that $z_B(q)= 1_B = 1_M$ (possibly $q=1_B$), where $z_B(q)$ is the central support in $B$ of the projection $q \in B$. Then it is easy to see that $z_{\langle M,B\rangle}(e_Bq)=Jz_B(q)J=1_M$, where $J$ is the modular conjugation on $\rL^2(M,\varphi\circ \rE_B)$. Hence, there exists a family of partial isometries $(v_i)_{i\in  I}$ in $\langle M,B\rangle$ such that $v_i^*v_i\leq qe_B$ for all $i\in I$ and $\sum_{i\in I}v_iv_i^*=1_M$ (use \cite[Lemma V.1.7]{Ta02} and Zorn's lemma). Then using the identification $qe_B\langle M,B\rangle qe_B=qBqe_B\cong qBq$, we can define the (not necessarily unital) normal $\ast$-homomorphism 
$$\pi : \langle M,B\rangle \to qBq\ovt \B(\ell^2(I)):x\mapsto \sum_{i,j\in I}\rE_B(v_i^*xv_j)\otimes e_{i,j}$$ 
where $(e_{i,j})_{i, j}$ is a matrix unit in $\B(\ell^2(I))$. Then $\widehat \varphi := (\varphi\otimes \Tr_{\B(\ell^2(I))})\circ \pi$ defines a faithful normal weight on $\langle M,B\rangle$. We have  
$$\widehat \varphi(x)=\sum_{i\in I}\varphi(q\rE_B(v_i^*xv_i)q), \quad x\in \langle M,B\rangle^+.$$ 
Note that $\varphi(q\rE_B(v_i^*xv_i)q)=\langle xv_i\Lambda_\varphi(q), v_i\Lambda_\varphi(q)\rangle_{\varphi}$ if $\varphi(q)< +\infty$ ({\em n.b.}~$e_B \Lambda_{\varphi\circ \rE_B}(q) = \Lambda_\varphi(q)$). 
We denote by $q_i \in qBq$ the unique projection such that $q_ie_B = qe_B v_i^*v_i qe_B =v_i^*v_i$ via $qe_B\langle M,B\rangle qe_B=qBqe_B$. Then  $\widehat{\varphi}$ is semifinite if $\varphi(q\rE_B(\, \cdot \,)q)$ is semifinite on $v_i^*\langle M,B \rangle v_i = q_i B q_i e_B$ for all $i \in I$, which is equivalent to the semifiniteness of $\varphi$ on $q_i Bq_i$ for all $i \in I$. If we assume $\varphi$ to be tracial on $B$, then $\varphi$ is semifinite on $rBr$ for any projection $r\in B$ and hence $\widehat{\varphi}$ is semifinite on $\langle M,B \rangle$.

Let $z\in \mathcal{Z}(B)$ be any nonzero projection such that $Bz$ is semifinite and observe that $z^\circ:= JzJ \in \mathcal Z(\langle M,B\rangle)$. Assume that the weight $\varphi$ is tracial on $Bz$ and write $\Tr_{Bz}:=\varphi(\, \cdot \, z) |_{Bz}$ on $Bz$. For any projection $r\in M^{\varphi\circ \rE_B}$, write $\rL^2(Mr,\varphi\circ \rE_B):=JrJ\rL^2(M,\varphi\circ \rE_B)$. In this setting, we show the following proposition.

\begin{prop}\label{general dimension lemma}
Keep the same notation as above. The following statements hold true.
\begin{enumerate}

\item The weight $\widehat \varphi (\, \cdot \, z^\circ)$ on $\langle M,B \rangle z^\circ$ is the unique faithful normal semifinite tracial weight which satisfies
\begin{equation}\label{eq-unique-trace}
\widehat{\varphi}((x^*e_Bx) z^\circ) = \Tr_{Bz}(\rE_B(zxx^*z))
\end{equation} 
for all $x \in M$. In particular, $\widehat \varphi (\, \cdot \, z^\circ)$ depends neither on the choice of the projection $q$ nor on the choice of the partial isometries $\{v_i\}_{i}$ as above.  We denote this unique trace by $\Tr_{\langle M,B \rangle z^\circ}$. 

\item For any projection $p\in \langle M,B\rangle$, we have
$$\dim_{(Bz,\Tr_{Bz})} p \rL^2(Mz,\varphi\circ \rE_B) = \dim_{(qBqz,\Tr_B(q\, \cdot \, qz))} p \rL^2(Mqz,\varphi\circ \rE_B)  = \Tr_{\langle M,B \rangle z^\circ}(pz^\circ).$$
\end{enumerate}
\end{prop}

In the case when $z=1_B$, we obtain the following corollary. 

\begin{cor}\label{semifinite dimension lemma}
Assume that $B$ is semifinite with semifinite trace $\Tr_B$. Let $q\in B$ be a projection with central support $1_B$ in $B$. Then for any projection $p\in \langle M,B\rangle$, we have
	$$\dim_{(B,\Tr_B)} p \rL^2(M,\Tr\circ \rE_B) = \dim_{(qBq,\Tr_B(q\, \cdot \, q))} p \rL^2(Mq,\Tr\circ \rE_B)  = \Tr_{\langle M,B\rangle}(p) ,$$
where $\Tr_{\langle M,B\rangle}$ is the unique trace on $\langle M,B \rangle$ satisfying $\Tr_{\langle M,B\rangle}(x^*e_Bx) = \Tr_{B}(\rE_B(xx^*))$ for all $x \in M$.
\end{cor}

\begin{rem}
In Corollary \ref{semifinite dimension lemma}, if we further assume that $\Tr_B(q)< +\infty$, then the dimension with respect to a semifinite trace $\Tr_B$ can be given by the one with respect to the finite valued trace $\Tr_B(q\, \cdot \, q)$. Thus, we can use dimension theory for finite von Neumann algebras.
\end{rem}

\begin{rem}\label{remark-semifinite}
We mention that if $B$ is not semifinite but possesses a nonzero semifinite direct summand, we may replace $B \subset M$ by a semifinite von Neumann subalgebra $\widetilde B \subset M$ with expectation. Indeed, let $z\in\mathcal{Z}(B)$ be a nonzero projection such that $Bz$ semifinite. Then put $\widetilde{B}:=Bz \oplus \C(1_M-z)$ and observe that $\widetilde{B} \subset M$ is a semifinite von Neumann subsalgebra with expectation $\rE_{\widetilde{B}}$ that coincides with $\rE_B$ on $zMz$  (here we assume that $M(1_M - z)$ is $\sigma$-finite). Observe that $e_B JzJ = e_Bz = e_{\widetilde{B}}z = e_{\widetilde{B}}JzJ$ and
$$\langle M,\widetilde{B}\rangle JzJ=(J\widetilde{B}J)'JzJ=(JBJ)'JzJ=\langle M,B\rangle JzJ.$$
It holds that $\Tr_{\langle M,\widetilde{B}\rangle JzJ}=\Tr_{\langle M,B\rangle JzJ}$ on $\langle M,\widetilde{B}\rangle JzJ$ by Proposition \ref{general dimension lemma}(1). 
\end{rem}

\begin{proof}[Proof of Proposition \ref{general dimension lemma}]
$(1)$ It is easy to see that $\pi(z^\circ)=\pi(1)(z\otimes 1)$ and hence $\pi(\langle M,B \rangle z^\circ)$ is contained in $qBqz\ovt \B(\ell^2(I))$. Thus the weight $\widehat \varphi (\, \cdot\, z^\circ) = (\varphi\otimes \Tr_{\B(\ell^2(I))})\circ \pi(\, \cdot\, z^\circ)$ is tracial on $\langle M,B \rangle z^\circ$. Observe that $z^\circ qe_B=qze_B$ and $\rE_B(z^\circ xe_By)=z\rE_B(x)\rE_B(y)$ for all $x,y\in \langle M, B \rangle$. Then for all $x\in M$, we have 
\begin{align*}
\widehat\varphi((x^*e_Bx)z^\circ)
&=\sum_{i\in I}\varphi(q\rE_B(z^\circ v_i^*x^*e_Bxv_i)q)\\
&=\sum_{i\in I}\Tr_{Bz}(zq\rE_B(v_i^*x^*)\rE_B(xv_i)q)\\
&=\sum_{i\in I}\Tr_{Bz}(z\rE_B(xv_i)q\rE_B(v_i^*x^*))\\
&=\sum_{i\in I}\varphi(\rE_B(z^\circ xv_iqe_Bv_i^*x^*))\\
&=\varphi(\rE_B(z^\circ x\sum_{i\in I}(v_ie_Bqv_i^*)x^*))\\
&=\varphi(\rE_B(z^\circ xx^*))\\
&=\Tr_{Bz}( \rE_B(zxx^*z)).
\end{align*}
This shows that $\widehat{\varphi}$ satisfies \eqref{eq-unique-trace} for all $x \in M$. Let now $p_n\in Bz$ be an increasing sequence of projections converging to $1_M$ $\sigma$-strongly and such that $\Tr_{Bz}(p_n)<+\infty$ for all $n \in \N$. Then the trace $\widehat \varphi(\, \cdot\,  z^\circ)$ takes finite values on the $\sigma$-weakly dense subset $\bigcup_{n \in \N} Mp_ne_Bp_nM $ in $\langle M,B \rangle z^\circ$ and hence, by \cite[Proposition VIII 3.15]{Ta03}, $\Tr_{\langle M,B \rangle z^\circ} := \widehat \varphi(\, \cdot\, z^\circ)$ is the unique faithful normal semifinite trace on $\langle M,B \rangle z^\circ$ which satisfies 
$$\Tr_{\langle M,B \rangle z^\circ}((x^*e_Bx) z^\circ) = \Tr_{Bz}(\rE_B(zxx^*z))$$ 
for all $x \in \bigcup_n p_nM$. 

$(2)$ Let $p \in \langle M,B\rangle$ be any projection. Up to replacing $p$ by $pz^\circ$ if necessary, we may assume that $p\leq z^\circ$. Since $qe_B$ has central support $1_M$, there exists families of partial isometries $(u_{i_1})_{i_1\in  I_1}$, $(w_{i_2})_{i_2\in I_2}$ in $\langle M,B\rangle$ such that $u_{i_1}^*u_{i_1}, w_{i_2}^* w_{i_2}\leq qe_B$ for all $i_1\in I_1, i_2\in I_2$, $\sum_{i_1\in I_1}u_{i_1}u_{i_1}^*=p$ and $\sum_{i_2\in I_2} w_{i_2}w_{i_2}^*=1_M-p$. Put $I:=I_1\cup I_2$ and in this case, the $\ast$-homomorphism $\pi$ above is given by
$$\pi : \langle M,B\rangle \to qBq\ovt \B(\ell^2(I)):x\mapsto \sum_{i_1,j_1\in I_1}\rE_B(u_{i_1}^*xu_{j_1})\otimes e_{i_1,j_1} + \sum_{i_2,j_2\in I_2}\rE_B(w_{i_2}^*xw_{j_2})\otimes e_{i_2,j_2}. $$ 
Since $\widehat{\varphi}(\, \cdot \, z^\circ)$ does not depend on the choices of $q$ and $\{v_i\}_{i\in I}$, one has 
$$\Tr_{\langle M,B \rangle z^\circ}(pz^\circ) = \widehat{\varphi}(p z^\circ) = \sum_{i_1\in I_1}\varphi(q\rE_B(u_{i_1}^* u_{i_1})q) = \sum_{i_1\in I_1}\varphi(p_{i_1}),$$
where $p_{i_1}:=q\rE_B(u_{i_1}^* u_{i_1})q = \rE_B(u_{i_1}^* u_{i_1})\in qBq$. We claim that $p_{i_1}$ is a projection in $qBqz$. For this, recall that $u_{i_1}^* u_{i_1}\leq qe_B$ and $u_{i_1} u_{i_1}^* \leq p \leq z^\circ$. Since $z^\circ$ is in the center of $\langle M,B \rangle$, one has $u_{i_1}^* u_{i_1} = u_{i_1}^*z^\circ u_{i_1} \leq qe_Bz^\circ = qze_B$. Then observe that
$$p_{i_1}e_B = \rE_B(u_{i_1}^* u_{i_1}) e_B = e_B u_{i_1}^* u_{i_1}e_B = u_{i_1}^* u_{i_1}$$
which shows that $p_{i_1}$ is a projection in $qBqz$. 

We next show that $\dim_{(qBqz,\Tr_B(q\, \cdot \, qz))} p JqzJ\rL^2(M,\varphi\circ \rE_B)  = \sum_{i_1\in I_1}\varphi(p_{i_1})$. Since the right hand side coincides with $\Tr_{\langle M,B \rangle z^\circ}(pz^\circ)$ and is independent of the choice of $q$, we also have $\dim_{(Bz,\Tr_{Bz})} p \rL^2(Mz,\varphi\circ \rE_B) = \Tr_{\langle M,B \rangle z^\circ}(pz^\circ)$ by considering $q = 1_B$. This will therefore complete the proof. 

Since $p=\sum_{i_1} u_{i_1}u_{i_1}^* \in \langle M,B\rangle = (JB J)'$, we have a right $B$-module isomorphism 
$$p\rL^2(M,\varphi\circ \rE_B) = \sum_{i_1} u_{i_1}u_{i_1}^* \rL^2(M,\varphi\circ \rE_B) \xrightarrow{\oplus_{i_1} u_{i_1}^*} \bigoplus_{i_1} u_{i_1}^* \rL^2(M,\varphi\circ \rE_B),$$
where $B$ acts by right multiplication diagonally on the right hand side. We have $u_{i_1}^* \rL^2(M,\varphi\circ \rE_B)=u_{i_1}^*u_{i_1} \rL^2(M,\varphi\circ \rE_B)$, since $u_{i_1}$ is a partial isometry. Observe then that
$$u_{i_1}^*u_{i_1} \rL^2(M,\varphi\circ \rE_B) = p_{i_1}e_B \rL^2(M,\varphi\circ \rE_B) = p_{i_1}\rL^2(B,\varphi) = p_{i_1}\rL^2(qzB,\varphi)$$ 
and hence we have the following right $qBqz$-module isomorphism
$$p JqzJ\rL^2(M,\varphi\circ \rE_B) \cong \bigoplus_{i_1} p_{i_1}JqzJ \rL^2(qzB,\varphi)=\bigoplus_{i_1} p_{i_1} \rL^2(qBqz,\varphi).$$ 
This shows that $\dim_{(qBqz,\Tr_B(q\, \cdot \, qz))} p JqzJ\rL^2(M,\varphi\circ \rE_B)  = \sum_{i_1\in I_1}\varphi(p_{i_1})$.
\end{proof}

\section{Free Araki-Woods factors satisfy the strong condition (AO)}

It was proved in \cite[Chapter 4]{Ho07} that all the free Araki-Woods factors satisfy Ozawa's condition (AO). In this appendix, we strengthen this result by showing that all the free Araki-Woods factors satisfy the strong condition (AO) from Definition \ref{AO^++}.

Let $H_{\R}$ be any real Hilbert space and $U : \R \to \mathcal O(H_\R)$ any orthogonal representation. Denote by $H = H_{\R} \otimes_{\R} \C = H_\R \oplus {\rm i} H_\R$ the complexified Hilbert space, by $I : H \to H : \xi + {\rm i} \eta \mapsto \xi - {\rm i} \eta$ the canonical anti-unitary involution on $H$ and by $A$ the infinitesimal generator of $U : \R \to \mathcal U(H)$, that is, $U_t = A^{{\rm i}t}$ for all $t \in \R$. Observe that $j : H_{\R} \to H :\zeta \mapsto (\frac{2}{A^{-1} + 1})^{1/2}\zeta$ defines an isometric embedding of $H_{\R}$ into $H$. Moreover, we have $IAI = A^{-1}$. Put $K_{\R} := j(H_{\R})$. It is easy to see that $K_\R \cap {\rm i} K_\R = \{0\}$ and that $K_\R + {\rm i} K_\R$ is dense in $H$. Write $T = I A^{-1/2}$. Then $T$ is a conjugate-linear closed invertible operator on $H$ satisfying $T = T^{-1}$ and $T^*T = A^{-1}$. Such an operator is called an {\it involution} on $H$. Moreover, we have $\dom(T) = \dom(A^{-1/2})$ and $K_\R = \{ \xi \in \dom(T) : T \xi = \xi \}$. In what follows, we will simply write
$$\overline{\xi + {\rm i} \eta} := T(\xi + {\rm i} \eta) = \xi - {\rm i} \eta, \forall \xi, \eta \in K_\R.$$

We introduce the \emph{full Fock space} of $H$:
\begin{equation*}
\mathcal{F}(H) =\C\Omega \oplus \bigoplus_{n = 1}^{\infty} H^{\otimes n}.
\end{equation*}
The unit vector $\Omega$ is called the \emph{vacuum vector}. For all $\xi \in H$, define the {\it left creation} operator $\ell(\xi) : \mathcal{F}(H) \to \mathcal{F}(H)$ by
\begin{equation*}
\left\{ 
{\begin{array}{l} \ell(\xi)\Omega = \xi, \\ 
\ell(\xi)(\xi_1 \otimes \cdots \otimes \xi_n) = \xi \otimes \xi_1 \otimes \cdots \otimes \xi_n.
\end{array}} \right.
\end{equation*}
We have $\|\ell(\xi)\|_\infty = \|\xi\|$ and $\ell(\xi)$ is an isometry if $\|\xi\| = 1$. For all $\xi \in K_\R$, put $W(\xi) := \ell(\xi) + \ell(\xi)^*$. The crucial result of Voiculescu \cite[Lemma 2.6.3]{VDN92} is that the distribution of the self-adjoint operator $W(\xi)$ with respect to the vector state $\varphi_U = \langle \, \cdot \,\Omega, \Omega\rangle$ is the semicircular law of Wigner supported on the interval $[-2\|\xi\|, 2\|\xi\|]$. 

\begin{df}[Shlyakhtenko, \cite{Sh96}]
Let $H_\R$ be any real Hilbert space and $U : \R \to \mathcal O(H_\R)$ any orthogonal representation. The \emph{free Araki-Woods} von Neumann algebra associated with $(H_\R, U_t)$, denoted by $\Gamma(H_{\R}, U_t)\dpr$, is defined by
\begin{equation*}
\Gamma(H_{\R}, U_t)\dpr := \{W(\xi) : \xi \in K_{\R}\}\dpr.
\end{equation*}
We will denote by $\Gamma(H_{\R}, U_t)$ the unital C$^*$-algebra generated by $1$ and by all the elements $W(\xi)$ for $\xi \in K_\R$.
\end{df}

The vector state $\varphi_U = \langle \, \cdot \,\Omega, \Omega\rangle$ is called the {\it free quasi-free state} and is faithful on $\Gamma(H_\R, U_t)\dpr$. Let $\xi, \eta \in K_\R$ and write $\zeta = \xi + {\rm i} \eta$. Put
\begin{equation*}
W(\zeta) := W(\xi) +  {\rm i} W(\eta) = \ell(\zeta) + \ell(\overline \zeta)^*.
\end{equation*}
Note that the modular automorphism  group $(\sigma_t^{\varphi_U})$ of the free quasi-free state $\varphi_U$ is given by $\sigma^{\varphi_U}_{t} = \Ad(\mathcal{F}(U_t))$, where $\mathcal{F}(U_t) = 1_{\C \Omega} \oplus \bigoplus_{n \geq 1} U_t^{\otimes n}$. In particular, it satisfies
\begin{equation*}
\sigma_{t}^{\varphi_U}(W(\zeta))  =  W(U_t \zeta), \forall \zeta \in K_\R +{\rm i} K_\R, \forall t \in \R. 
\end{equation*}

It is easy to see that for all $n \geq 1$ and all $\zeta_1, \dots, \zeta_n \in K_\R + {\rm i} K_\R$, $\zeta_1 \otimes \cdots \otimes \zeta_n \in \Gamma(H_\R, U_t) \Omega$. We will denote by $W(\zeta_1 \otimes \cdots \otimes  \zeta_n) \in \Gamma(H_\R, U_t)$ the unique element such that 
$$\zeta_1 \otimes \cdots \otimes \zeta_n = W(\zeta_1 \otimes \cdots \otimes \zeta_n) \Omega.$$
We refer to \cite[Section 2]{HR14} for further details. Note that since inner products are assumed to be linear in the first variable, we have $\ell(\xi)^*\ell(\eta) = \overline{\langle \xi, \eta\rangle} 1 = \langle \eta, \xi \rangle 1$ for all $\xi, \eta \in H$. In particular, the Wick formula from \cite[Proposition 2.1]{HR14} is 
\begin{align*}
& W(\xi_1 \otimes \cdots \otimes \xi_r) W(\eta_1 \otimes \cdots \otimes \eta_s) \\
&= W(\xi_1 \otimes \cdots \otimes \xi_r \otimes \eta_1 \otimes \cdots \otimes \eta_s) + 
      \overline{\langle \overline \xi_r, \eta_1\rangle} \, W(\xi_1 \otimes \cdots \otimes \xi_{r - 1}) W(\eta_2 \otimes \cdots \otimes \eta_s)
\end{align*}
for all $\xi_1, \dots, \xi_r, \eta_1, \dots, \eta_s \in K_\R + {\rm i} K_\R$.

The main result of this appendix is the following theorem.

\begin{thm}\label{thm-FAW-AO}
Let $U : \R \to \mathcal O(H_\R)$ be any orthogonal representation on a separable real Hilbert space. Then the von Neumann algebra $\Gamma(H_\R, U_t)\dpr$ satisfies the strong condition (AO).
\end{thm}

\begin{proof}
Put $M := \Gamma(H_\R, U_t)\dpr$ and denote by $(M, \mathcal H, J, \mathfrak P)$ a standard form for $M$. We will use the identification $\mathcal H = \overline {M  \Omega} = \mathcal F(H)$. Put $\widetilde A := \Gamma(H_\R, U_t)$ and $\mathcal C := {\rm C}^*(\ell(\xi) : \xi \in K_\R)$. Observe that the unital C$^*$-algebra $\mathcal C$ is always nuclear. Indeed, if $\dim H_\R < +\infty$, then $\mathcal C$ is an extension of $\mathcal O_{\dim H_\R}$ by the compact operators (see \cite[Proposition 3.1]{Cu77}) and hence is nuclear by \cite[Proposition IV.3.1.3]{Bl06}. If $\dim H_\R = +\infty$, then $\mathcal C \cong \mathcal O_\infty$ and hence is nuclear. By construction, the unital C$^*$-algebra $\widetilde A \subset M \cap \mathcal C$ is exact and $\sigma$-weakly dense in $M$.

Put $K_{\an} := \bigcup_{\lambda > 1} \mathbf 1_{[\lambda^{-1}, \lambda]}(A)(H) \subset K_\R + {\rm i} K_\R$. Observe that $K_{\an} \subset K_\R + {\rm i} K_\R$ is a dense subspace of elements $\eta \in K_\R + {\rm i} K_\R$ for which the map $\R \to K_\R + {\rm i} K_\R : t \mapsto U_t \eta$ extends to an $(K_\R + {\rm i} K_\R)$-valued entire analytic map and that $\overline{K_{\an}} = K_{\an}$. For all $\eta \in K_{\an}$, the element $W(\eta)$ is analytic with respect to the modular automorphism group $(\sigma_t^{\varphi_U})$ and we have $\sigma_z^{\varphi_U}(W(\eta)) = W(A^{{\rm i}z}\eta)$ for all $z \in \C$. Denote by $A \subset \widetilde A$ the unital C$^*$-algebra generated by $1$ and by all the elements $W(\zeta)$ for $\zeta \in K_{\an}$. Since $A$ is uniformly dense in $\widetilde A$, it follows that $A$ is $\sigma$-weakly dense in $M$ and exact. Moreover, for all $\eta \in K_{\an}$, all $n \geq 1$ and all $\xi_1, \dots, \xi_n \in K_\R + {\rm i} K_\R$, using \cite[Proposition 2.1]{HR14} and \cite[Lemma VIII.3.10]{Ta03}, we have
\begin{align*}
J W(\eta) J \,(\xi_1 \otimes \cdots \otimes \xi_n)  &= J W(\eta) J \, W(\xi_1 \otimes \cdots \otimes \xi_n)  \Omega \\
&= W(\xi_1 \otimes \cdots \otimes \xi_n) \, \sigma^{\varphi_U}_{-{\rm i}/2}(W( \eta)^*)  \Omega \\
&= W(\xi_1 \otimes \cdots \otimes \xi_n) \, \sigma^{\varphi_U}_{-{\rm i}/2}(W(\overline \eta))  \Omega \\
&= W(\xi_1 \otimes \cdots \otimes \xi_n)  W(A^{1/2}\overline \eta)  \Omega \\
&= W(\xi_1 \otimes \cdots \otimes \xi_n \otimes A^{1/2}\overline \eta)  \Omega + \overline{\langle \overline \xi_n, A^{1/2} \overline\eta\rangle} \, W(\xi_1 \otimes \cdots \otimes \xi_{n - 1}) \Omega \\
&= \xi_1 \otimes \cdots \otimes \xi_n \otimes A^{1/2}\overline \eta + \overline{\langle \overline\xi_n, A^{1/2} \overline\eta\rangle} \, \xi_1 \otimes \cdots \otimes \xi_{n - 1}.
\end{align*}

Let $\xi \in K_\R$ and $\eta \in K_{\an}$. Using the above equation, we have
\begin{align*}
\ell(\xi) \, J W(\eta) J \, \Omega &= \ell(\xi) \, W(A^{1/2} \overline \eta)  \Omega \\
&= \ell(\xi) \, A^{1/2} \overline \eta \\
&= \xi \otimes A^{1/2} \overline\eta \\
J W(\eta) J \, \ell(\xi) \,  \Omega &= J W(\eta) J \, \xi \\
&= \xi \otimes A^{1/2}\overline\eta + \overline{\langle \overline \xi, A^{1/2}\overline\eta\rangle} \, \Omega.
\end{align*}
For all $n \geq 1$ and all $\xi_1, \dots, \xi_n \in K_\R + {\rm i} K_\R$, using the above equation, we moreover have
\begin{align*}
\ell(\xi) \, J W(\eta) J \, (\xi_1 \otimes \cdots \otimes \xi_n) &= \ell(\xi)  \left(\xi_1 \otimes \cdots \otimes \xi_n \otimes A^{1/2}\overline \eta + \overline{\langle \overline \xi_n, A^{1/2} \overline\eta\rangle} \, \xi_1 \otimes \cdots \otimes \xi_{n - 1} \right) \\
&= \xi \otimes \xi_1 \otimes \cdots \otimes \xi_n \otimes A^{1/2}\overline \eta + \overline{\langle \overline \xi_n, A^{1/2} \overline\eta\rangle} \, \xi \otimes \xi_1 \otimes \cdots \otimes \xi_{n - 1} \\
J W(\eta) J \, \ell(\xi) \,  (\xi_1 \otimes \cdots \otimes \xi_n) &= J W(\eta) J \, (\xi \otimes \xi_1 \otimes \cdots \otimes \xi_n) \\
&= \xi \otimes \xi_1 \otimes \cdots \otimes \xi_n \otimes A^{1/2}\overline \eta + \overline{\langle \overline \xi_n, A^{1/2} \overline\eta\rangle} \, \xi \otimes \xi_1 \otimes \cdots \otimes \xi_{n - 1}.
\end{align*}
This shows that $[\ell(\xi), JW(\eta)J] = \overline{\langle \overline \xi, A^{1/2}\overline\eta\rangle} \,  P_{\C \Omega}$ and hence $[\ell(\xi), JW(\eta)J] \in \mathbf K(H)$. Therefore, we obtain $[\mathcal C, JAJ] \subset \mathbf K(H)$ and hence $M$ satisfies the strong condition (AO).
\end{proof}

\end{document}